\documentclass[leqno]{report}

\usepackage{amsthm}
\usepackage{amssymb}
\usepackage{amsxtra}     
\usepackage{rac}                        
\usepackage{dbl12}
\usepackage{epsfig}
\usepackage{verbatim}
\usepackage{psfrag}
\usepackage{graphicx}

\theoremstyle{plain}
\newtheorem{theorem}{Theorem}

\newtheorem{proposition}[theorem]{Proposition}
\newtheorem{corollary}[theorem]{Corollary}
\newtheorem{lemma}[theorem]{Lemma}
\newtheorem{question}[theorem]{Question}

\theoremstyle{definition}
\newtheorem{definition}[theorem]{Definition}

\theoremstyle{remark}

\newenvironment{definitions}{\begin{definition}\renewcommand{\theenumi}{(\roman{enumi})}\begin{enumerate}}{\end{enumerate}\end{definition}}

\numberwithin{theorem}{chapter}        

\DeclareRobustCommand{\edit}[1]{%
   \ifmmode
   \mathbf{\blacktriangleright #1 \blacktriangleleft }%
   \else
      \textbf{$\blacktriangleright$ #1 $\blacktriangleleft$ }%
   \fi
}

\def\RCSpreamble$#1: #2 ${\expandafter\def\csname RCSpreamble#1\endcsname{#2}}
\RCSpreamble$Id: preamble.tex,v 1.12 2006/04/17 18:38:26 bkasterm Exp bkasterm $

\long\def\symbolfootnote[#1]#2{\begingroup%
\def\thefootnote{\fnsymbol{footnote}}\footnote[#1]{#2}\endgroup}

\newcommand{\restrict}{\upharpoonright}
\newcommand{\partmap}{\rightharpoonup}
\newcommand{\partialmap}{\rightharpoonup}

\renewcommand{\mid}{\mathrel{:}}

\DeclareMathOperator{\supp}{supp}
\newcommand{\forces}{\ensuremath{\Vdash}}

\newcommand{\highcheck}{\check{\vphantom{\A{}}}}
\newcommand{\MA}{\ensuremath{\mathsf{MA}}}
\newcommand{\ZFC}{\ensuremath{\mathsf{ZFC}}}
\newcommand{\CH}{\ensuremath{\mathsf{CH}}}
\newcommand{\GCH}{\ensuremath{\mathsf{GCH}}}
\newcommand{\powset}{\mathcal{P}}
\DeclareMathOperator{\cf}{cf}

\DeclareMathOperator{\Fn}{Fn}
\newcommand{\Non}{\mathrm{non}}

\newcommand{\Ord}{\mathsf{Ord}}
\newcommand{\card}[1]{\lvert #1 \rvert}
\newcommand{\N}{\ensuremath{\mathbb{N}}}
\DeclareMathOperator{\dom}{dom}
\DeclareMathOperator{\ran}{ran}
\newcommand{\A}{\mathcal{A}}
\newcommand{\BS}{\ensuremath{{}^\N\N}}
\DeclareMathOperator{\Sym}{Sym}
\DeclareMathOperator{\id}{Id}
\DeclareMathOperator{\Id}{Id}
\DeclareMathOperator{\lh}{lh}

\DeclareMathOperator{\Th}{Th}
\DeclareMathOperator{\Sk}{Sk}

\newcommand{\seq}[1]{\bar{#1}}

\newcommand{\Nat}{\ensuremath{\mathsf{Nat}}}
\newcommand{\Real}{\ensuremath{\mathsf{Real}}}

\newcommand{\code}[1]{\ulcorner {#1} \urcorner}

\newcommand{\poset}{\ensuremath{{\mathbb{E}}_\A{}}}

\newcommand{\B}{\mathcal{B}}

\newcommand{\ocx}{\#_x}
\newcommand{\use}{\mathsf{use}}

\newcommand{\cof}{\mathrm{cof}}
\newcommand{\non}{\mathrm{non}}
\newcommand{\add}{\mathrm{add}}
\newcommand{\loc}{\mathbb{LOC}}
\newcommand{\Null}{\mathcal{N}}

\DeclareMathOperator{\Dp}{Dp}

\newcommand{\ON}{\mathsf{ON}}

\DeclareMathOperator{\red}{red}

\newcommand{\leqc}{\mathrel{< \!\! \circ}}

\newcommand{\I}{\mathcal{I}}
\DeclareMathOperator{\md}{md}
\newcommand{\leqcirc}{\mathrel{< \! \!\circ}}
\newcommand{\bv}[1]{\, [ \! [ #1 ] \! ] \,}

\newcommand{\finset}[1]{{}^{<\omega}[#1]}
\newcommand{\finsetd}[2]{{}^{#2}[{#1}]}
\newcommand{\finseq}[1]{{}^{<\omega}(#1)}
\newcommand{\funcN}[1]{{}^\N#1}

\begin{document}

\bibliographystyle{plain}    


\titlepage{Cofinitary Groups and Other Almost Disjoint Families of
Reals}{Bart Kastermans}{Doctor of Philosophy}
{Mathematics} {2006}
{Professor Andreas R. Blass, Co-chair\\
 Professor Yi Zhang, Co-chair, Sun Yat-Sen University (China)\\
  Professor Peter G. Hinman\\
  Professor Berit Stensones\\
  Associate Professor James P. Tappenden\\
  Assistant Professor Alexei S. Kolesnikov\\
  Assistant Professor Monica VanDieren}


\initializefrontsections


\setcounter{page}{1}



\startacknowledgementspage{This thesis is the result of research I
have done while I was a PhD student at the University of Michigan.  I
would like to thank everyone who has played a part in my doing this.
  
  In particular I am grateful to Andreas Blass for being an excellent
  advisor, Peter Hinman for many nice discussions and much information
  about recursion theory, and Yi Zhang who got me started on both very
  mad families and cofinitary groups, and of course for inviting me to
  visit him in China at Sun Yat-Sen University during the year 2005.
  
  I also want to thank Marat Arslanov for hosting me in Kazan in the
  summer of 2003, making for a pleasant and productive visit.

  There are also many people who influenced me before I came to
  Michigan: foremost Jan van Mill (Free University in Amsterdam) who
  introduced me to set theory and supported me in my initial work.
  While I was an undergraduate student at the Free University in
  Amsterdam, I learned a lot from the following people: Klaas Pieter
  Hart (Delft University of Technology) who helped me while working
  on my Masters thesis; the teachers of the Master Class in
  Mathematical Logic organized by the Mathematical Research Institute
  in the Netherlands (at the University of Nijmegen: Henk Barendregt,
  Herman Geuvers, and Wim Veldman; at Utrecht University: Albert
  Visser, Ieke Moerdijk, Harold Schellinx, and Jaap van Oosten); and
  the teachers of the logic classes I took at the University of
  Amsterdam: Kees Doets, Dick de Jongh, and Anne Troelstra.}

\tableofcontents

\startthechapters 

\chapter{Introduction}
\label{intro} 
Almost disjoint families and maximal almost disjoint (mad) families
have received a lot of attention in set theory.  In their study many
different varieties have been introduced.  There are the ``standard''
almost disjoint families, and varieties on different spaces and with
different additional conditions.

The general definitions of the notions almost disjoint, almost
disjoint family, and maximal almost disjoint family are as follows.
Take a collection $X$ consisting of objects of a certain (infinite)
cardinality $\kappa$.  Two objects $x,y$ in $X$ are called
\emph{almost disjoint} if the intersection of $x$ and $y$, $x \cap y$,
is of cardinality less than $\kappa$.  A subset $\mathcal{A}$ of $X$
is an \emph{almost disjoint family} if every two distinct objects in
$\mathcal{A}$ are almost disjoint.  Such a subset $\mathcal{A}$ is a
\emph{maximal almost disjoint family} if $\mathcal{A}$ is an almost
disjoint family and there does not exist an almost disjoint family
$\mathcal{B}$ of $X$ such that $\mathcal{A}$ is a proper subset of
$\mathcal{B}$ ($\A$ is maximal with respect to inclusion).

Existence of maximal almost disjoint families is easy to see.  Using
Zorn's lemma (an equivalent of the axiom of choice) and noticing that
the union of an increasing chain (with respect to inclusion) of almost
disjoint families is almost disjoint, we see that there does exist a
maximal almost disjoint family.

Different notions of almost disjointness can be formed from the
general notion in different ways --- we will mention the two we will
use.  One can change the set $X$ of which the almost disjoint families
are subsets.  Usual choices for this are $\powset(\N)$, the
powerset of the natural numbers, or $\BS$, the set of functions from
$\N$ to $\N$.  Other choices for $X$ have been and are receiving
attention, for instance see \cite{BHZ}.  Another
way to change the notion of almost disjointness is by imposing
additional conditions on the family $\A$.  One can for instance impose
a group structure on $\A$ if there is one on $X$.  With additional
conditions imposed, the existence of maximal almost disjoint families
needs to be reinvestigated.

In this thesis we will study some different varieties of almost
disjoint and maximal almost disjoint families (for all of them
$\kappa$ as in the above general definition of almost disjointness
will be equal to $\aleph_0$).  Some of the results and questions are
inspired by similarities and differences among these different
varieties of almost disjoint families.

\section{Subsets of $\N$}

The standard example of almost disjointness is the following.

\begin{definition}
  We call $x,y \subseteq \N$ \emph{almost disjoint} if both are
  infinite and $x \cap y$ is finite.
\end{definition}

For various reasons, which we will explain shortly, we impose the
additional condition that the family is infinite on the related notion
of maximal almost disjoint family.

\begin{definition}
  A family $\mathcal{A} \subseteq \mathcal{P}(\N)$ is a \emph{maximal
  almost disjoint family (of subsets of the natural numbers)} if it
  is infinite, consists of pairwise almost disjoint sets, and is
  not properly contained in another such family.
\end{definition}

Existence of maximal almost disjoint families of subsets of the
natural numbers is still unproblematic: we get such a family by
applying Zorn's lemma to an infinite partition of $\N$ into
infinite parts.

The following results on the cardinality of these maximal almost
disjoint families are well known.

\begin{theorem}
  \label{thm:IntroMACHN}
  Both Martin's axiom ($\MA$) and the continuum hypothesis ($\CH$)
  imply that all maximal almost disjoint families of subsets of the
  natural numbers are of cardinality continuum.
\end{theorem}

\begin{theorem}[\cite{KK80}]  
  \label{thm:IntroCohenN}
  In a model $M$ of $\ZFC + \CH$ there exists a maximal almost
  disjoint family of subsets of the natural numbers that is still
  maximal almost disjoint in any Cohen extension of $M$.
\end{theorem}

\begin{theorem}
  There does not exist a countable maximal almost disjoint family of
  subsets of the natural numbers.
\end{theorem}

With these theorems the following definition becomes reasonable.

\begin{definition}
  Let the cardinal $\mathfrak{a}$ be the least cardinality of a
  maximal almost disjoint family of subsets of the natural numbers.
\end{definition}

The reason for the additional condition that the family be infinite in
the definition of maximal almost disjoint family of subsets of the
natural numbers is to keep the cardinal $\mathfrak{a}$ from being
trivial.  If we did not have the additional condition, then any
partition of $\N$ into finitely many infinite pieces would be a
maximal almost disjoint family.

The following theorem about $\mathfrak{a}$ is well known.

\begin{theorem}[\cite{SHe}]
  Suppose $\kappa$ is a regular uncountable cardinal less than the
  continuum in a model of $\ZFC$.  Then there exists a forcing
  extension preserving cardinals and the cardinality of the continuum,
  in which there is a maximal almost disjoint family of subsets of
  the natural numbers of cardinality $\kappa$.
\end{theorem}

The families $\A$ are subsets of $\powset(\N)$ which is a Polish space
(a separable completely metrizable space).  This is clear from the
fact that $2^\N$ is, and $2^\N$ and $\powset(\N)$ are homeomorphic as
witnessed by the homeomorphism $S \in \powset(\N) \mapsto \chi_S \in
2^\N$, where $\chi_S$ is the characteristic function of $S$.  So we can
apply descriptive set theoretic methods to these families in order to
study their definability.

\begin{question}
  How definable can a maximal almost disjoint family of subsets of the
  natural numbers be?
\end{question}

The following two well known theorems answer this
question.

\begin{theorem}[\cite{ARDM77}]
  \label{thm:IntroNotBorelN}
  A maximal almost disjoint family of subsets of the natural numbers
  cannot be analytic ($\mathbf{\Sigma^1_1}$).
\end{theorem}

\begin{theorem}[\cite{AM89}]
  \label{thm:IntroCoanalyticN}
  The axiom of constructibility implies that there exists a coanalytic
  ($\mathbf{\Pi^1_1}$) maximal almost disjoint family of subsets of
  the natural numbers.
\end{theorem}

\section{Families of Functions}

We can now change the underlying set from $\powset(\N)$ to a different
set and see what we get there.  The first such set we consider is
$\BS$, Baire space, the space of functions from $\N$ to $\N$.  Note
that this is also a Polish space.  The usual definition of almost
disjointness on this space is the following.

\begin{definition}
  We call $g_0,g_1 \in \BS$ \emph{almost disjoint} or \emph{eventually
  different} if the set $\{n \in \N \mid g_0(n) = g_1(n) \}$ is finite.
\end{definition}

This is equivalent to the general notion of almost disjointness on
this space if we consider functions from $\N$ to $\N$ to be subsets of
$\N \times \N$.  The corresponding notion of (maximal) almost disjoint
family, called \emph{maximal almost disjoint family of functions},
also requires no change from the general notion.

The analogues of Theorems \ref{thm:IntroMACHN}, \ref{thm:IntroCohenN},
and \ref{thm:IntroCoanalyticN} can be proved using similar methods.
The following question is still open though (the analogue of Theorem
\ref{thm:IntroNotBorelN}).

\begin{question}
  Does there exist a Borel maximal almost disjoint family of functions?
\end{question}

In working on this question Juris Stepr\=ans introduced the following
strengthening of the notion of maximal almost disjoint family of
functions.

\begin{definitions}
\item
  For two sets $X, Y$ we say $X$ is \emph{almost contained} in $Y$,
  written $X \subseteq^* Y$, if $X \setminus Y$ is finite.

\item
  A function $f \in \BS$ is \emph{finitely covered} by a family
  $\mathcal{A} \subseteq \BS$ if there exist $g_0, \ldots, g_n \in
  \mathcal{A}$ such that $f \subseteq^* \bigcup_{i \leq n} g_i$ is
  finite.
  
\item
  A family $F \subseteq \BS$ is \emph{finitely covered} by a family
  $\mathcal{A}$ if there exists an $f \in F$ that is finitely covered
  by $\mathcal{A}$ (note: only one function in the family needs to be
  finitely covered).
  
\item A family $\mathcal{A} \subseteq \BS$ is a \emph{strongly mad
    family} if it is an almost disjoint family of functions, and for
  every countable $F \subseteq \BS$ that is not finitely covered by
  $\mathcal{A}$ there is a function $g \in \mathcal{A}$ such that for
  all $f \in F$ the intersection $f \cap g$ is infinite.
\end{definitions}

Note that any strongly mad family $\A$ is mad:
if $f$ is almost disjoint from all members of $\A$, then certainly
$\{f\}$ is not finitely covered by $\A$.  Since $\A$ is strongly mad
this means there is $g \in \A$ such that $g \cap f$ is infinite
contradicting that $f$ is almost disjoint from all members of $\A$.

Stepr\=ans showed the following theorem about this notion.

\begin{theorem}[\cite{KSZ}]
  \label{steprans theorem}
  There does not exist an analytic ($\mathbf{\Sigma^1_1}$) strongly
  mad family.
\end{theorem}

This result motivated us to study further the existence of strongly
mad families.  Note that with the additional covering condition
imposed, the usual proof of existence using Zorn's lemma no longer
works.  In studying this existence question it quickly became apparent
that all the results we obtained go through for a natural further
strengthening.

\begin{definition}
  A family $\mathcal{A} \subseteq \BS$ is a \emph{very mad family} if
  it is an almost disjoint family and for every family $F \subseteq
  \BS$ that is not finitely covered by $\mathcal{A}$ such that $|F| <
  |\mathcal{A}|$, there is a function $g \in \mathcal{A}$ such that for
  all $f \in F$ the intersection $f \cap g$ is infinite.
\end{definition}

Note the following about these notions:
\begin{itemize}
\item
   Any very mad family is strongly mad (there do not exist
   countable mad families in Baire space).
\item
   A strongly mad family of cardinality $\aleph_1$ is very mad.
\item
   Under the continuum hypothesis the notions of very mad and strongly
   mad coincide.
\end{itemize}

We then proved the following theorems.

\begin{theorem}[\cite{BK}]
  Martin's Axiom implies that very mad families exist and are of
  cardinality $2^{\aleph_0}$.
\end{theorem}

\begin{theorem}[\cite{BK}]
 Any model of $\ZFC + \CH$ contains a very mad family that is still
 very mad in any Cohen extension of that model.
\end{theorem}

\begin{theorem}[\cite{BK}]
  Suppose $\kappa$ is a regular uncountable cardinal less than the
  continuum in a model of $\ZFC$.  Then there exists a forcing
  extension preserving cardinals and the cardinality of the continuum,
  in which there is a very mad family of cardinality $\kappa$.
\end{theorem}

With these results the question of whether $\ZFC$ suffices to prove
existence of very mad families is still open.

After this we showed with Yi Zhang a companion result to Stepr\=ans'
result (\ref{steprans theorem}).

\begin{theorem}[\cite{KSZ}]
  The axiom of constructibility implies the existence of a coanalytic
  ($\mathbf{\Pi^1_1}$) very mad family.
\end{theorem}

We noticed while working on very mad families that they are not as big
as they seem; it is possible that there is very much space outside of
them.

\begin{definition}
  \label{definition orthogonal}
  Two very mad families $\A$, $\mathcal{B}$ are \emph{orthogonal} if
  neither is finitely covered by the other.
\end{definition}

We then proved the following theorem.

\begin{theorem}[\cite{BK}]
  Martin's axiom implies the existence of a continuum size collection
  of very mad families that are pairwise orthogonal.
\end{theorem}

And even more:

\begin{theorem}
  The continuum hypothesis implies that for every mad family $\A$
  there exists a very mad family $\mathcal{B}$ such that $\A$ and
  $\mathcal{B}$ are orthogonal.
\end{theorem}

\section{Families and Groups of Permutations}

We change the underlying space for the notion of almost
disjointness again.  Note that other than the change of space the
notions here are identical to the general notion.

\begin{definitions}
\item
  Let $\Sym(\N) \subseteq \BS$ denote the group of bijections from the
  natural numbers to the natural numbers with composition as the group
  operation.

\item
  A family of permutations $\mathcal{A} \subseteq \Sym(\N)$ is
  \emph{almost disjoint} if all distinct $f, g \in \mathcal{A}$ are
  almost disjoint.  It is a \emph{maximal almost disjoint family (of
  permutations)} if it is almost disjoint and not properly included
  in another such family.
\end{definitions}

Having an almost disjoint family being a subset of a group makes it
very natural to impose the requirement that the family is a group,
which gives us the following.

\begin{definition}
  A family $\mathcal{A} \subseteq \Sym(\N)$ is a \emph{cofinitary
    group} if it is an almost disjoint family and it is a subgroup of
  $\Sym(\N)$.  It is a \emph{maximal cofinitary group} if it is a
  cofinitary group not properly contained in another cofinitary group.
\end{definition}

Here we should point out that this is not the usual definition of
cofinitary group, although, of course, it is equivalent to it.

\begin{definitions}
\item A permutation $g \in \Sym(\N)$ is \emph{cofinitary} if it has
  only finitely many fixed points.
\item
  We write $G \leq \Sym(\N)$ if $G$ is a subgroup of $\Sym(\N)$.
\end{definitions}

The usual definition is given by the condition in the following theorem.

\begin{theorem}
  A subgroup $G \leq \Sym(\N)$ is a cofinitary group if all its
  non-identity members are cofinitary.
\end{theorem}

\begin{proof}
  For $f,g \in G$, $n$ is a fixed point of $f^{-1}g$ iff $g(n) = f(n)$
  iff $(n, g(n)) \in g \cap f$.
\end{proof}

The existence of both maximal almost disjoint families of permutations
and maximal cofinitary groups follows, as it does for the general
notions, from an argument using Zorn's lemma.

The following theorem was proved by Adeleke \cite{A} and Truss \cite{T}.

\begin{theorem}
  If $H \leq \Sym(\N)$ is a maximal cofinitary group, then $H$ is not
  countable.
\end{theorem}

Also, P. Neumann proved the following result (see, e.g.
\cite[Proposition 10.4]{C}).

\begin{theorem}
  There exists a cofinitary group of cardinality $2^{\aleph_0}$.
\end{theorem}

Thus, {P.} Cameron (in \cite{C}) asked the following question.

\begin{question}
  If the continuum hypothesis ($\CH$) fails, is it possible that there
  exists a maximal cofinitary group $H$ such that $|H| <
  2^{\aleph_0}$?
\end{question}

In \cite{Z}, this question was answered by proving
the following results.

\begin{theorem}
  \label{yi1.4}
  Martin's axiom implies that, if $H \leq \Sym(\N)$ is a
  maximal cofinitary group, then $H$ has cardinality $2^{\aleph_0}$.
\end{theorem}

In the following theorem the notation $M^\mathbb{P} \models \varphi$
means that for any $\mathbb{P}$-generic set $G$ over $M$ the statement
$\varphi$ is true in $M[G]$.

\begin{theorem}
  \label{yi1.5}
  Let $M \models \ZFC + \neg \CH$.  Let $\kappa \in M$ be a regular
  cardinal such that in $M$ $\aleph_1 \leq \kappa < 2^{\aleph_0} =
  \lambda$.  Then there exists a countable chain condition notion of
  forcing $\mathbb{P}$ such that the following statements hold in
  $M^{\mathbb{P}}$:
  {
    \renewcommand{\theenumi}{(\roman{enumi})}
  \begin{enumerate}
    \item $2^{\aleph_0} = \lambda$;
    \item there exists a maximal cofinitary group $H \leq \Sym(\N)$ of
      cardinality $\kappa$.
  \end{enumerate}
  }
\end{theorem}

The results corresponding to Theorem \ref{yi1.4} and Theorem
\ref{yi1.5} for maximal almost disjoint families of permutations can
be proved by similar methods.  So both of the following two cardinal
numbers are non-trivial.

\begin{definition}
  \label{definition ap and ag}
  Define the cardinal $\mathfrak{a}_p$ to be the least cardinality of
  a maximal almost disjoint family of permutations, and similarly,
  $\mathfrak{a}_g$ to be the least cardinality of a maximal cofinitary
  group.
\end{definition}

That these families can behave essentially differently from maximal
almost disjoint families is witnessed by the following theorem.

\begin{theorem}[\cite{BSZ}]
  It is consistent with $\ZFC$ that $\mathfrak{a} < \mathfrak{a}_p =
  \mathfrak{a}_g$.
\end{theorem}

So the difference in which space the almost disjoint family is defined
on can be used to get them to be of different cardinality (in
\cite{Zthesis} it was proved that $\mathfrak{a} <
\mathfrak{a}_p$ is consistent, and in \cite{HSZ} that $\mathfrak{a} <
\mathfrak{a}_g$ is consistent).  The following question is however
still open.

\begin{question}[\cite{YZ03}]
  Is it consistent with $\ZFC$ that $\mathfrak{a}_p$ is distinct from
  $\mathfrak{a}_g$?
\end{question}

That is, does the group structure influence the possible least
cardinalities of almost disjoint families?

We next compare the cardinals $\mathfrak{a}_p$ and $\mathfrak{a}_g$ to
some other well known cardinals which we define first.

\begin{definitions}
\label{def part Cichon}
\item Suppose that $H$ is a group that is not finitely generated.
  Then $H$ can be expressed as the union of a chain of proper
  subgroups.  The \emph{cofinality} of $H$, written $c(H)$, is the
  least $\lambda$ such that $H$ can be expressed as the union of a
  chain of $\lambda$ proper subgroups.

\item $S \subseteq \BS$ is \emph{meager} if it is contained in the union of
  countably many nowhere dense sets.  $\mathcal{M}$ is the collection
  of meager subsets of $\BS$.
  
\item $\Non(\mathcal{M})$, the \emph{uniformity of meager sets}, is
  the size of the smallest non-meager set of reals.

\item $\mathcal{N}$ is the collection of Lebesgue null subsets of
  $\BS$.
  
\item $\add(\mathcal{N})$, the \emph{additivity of null sets}, is the
  least cardinality of a family $\mathcal{F} \subseteq \mathcal{N}$
  such that $\cup \mathcal{F}$ is not a Lebesgue null set.
  
\item $\cof(\mathcal{N})$, the \emph{cofinality of null sets}, is the
  least cardinality of a family $\mathcal{F} \subseteq \mathcal{N}$
  such that for all $N \in \mathcal{N}$ there is $S \in \mathcal{F}$
  such that $N \subseteq S$.
\end{definitions}

We first focus on the cofinality of the symmetric group.  The
following result was proved by {H.}  {D.} Macpherson and {P.} Neumann
in \cite{MN}.

\begin{theorem}
  \label{thm:1.6}
  If $\kappa$ is an infinite cardinal, then $c(\Sym(\kappa)) >
  \kappa$.
\end{theorem}

Upon learning of Theorem \ref{thm:1.6}, {A.} Mekler and {S.} Thomas
independently pointed out the following easy observation (see, e.g.
\cite{ST}). 

\begin{theorem}
  \label{thm:1.7}
  Suppose that $M \models \kappa^\omega = \kappa > \aleph_1$.  Let
  $\mathbb{P} = \Fn(\kappa,2)$ be the partial order of finite partial
  functions from $\kappa$ to $2$.  Then $M^\mathbb{P} \models
  c(\Sym(\N)) = \aleph_1 < 2^{\aleph_0} = \kappa$.
\end{theorem}

Although $\MA$ implies $c(\Sym(\N)) = 2^{\aleph_0}$ (see, e.g.
\cite{ST}), some results indicate that $c(\Sym(\N))$ is rather small
among the cardinal invariants.  We give two examples:

(I) If $\mathfrak{d}$ is the dominating number (the minimum
cardinality of a dominating family in $\BS$), then
\begin{theorem}[\cite{ST}]
  $c(\Sym(\N)) \leq \mathfrak{d}$.
\end{theorem}

(II) A notion of forcing $\mathbb{P}$ is \emph{Suslin} if and only if
$\mathbb{P}$ is a $\mathbf{\Sigma^1_1}$ subset of $\mathbb{R}$ and
both $\leq_{\mathbb{P}}$ and $\bot_{\mathbb{P}}$ are
$\mathbf{\Sigma^1_1}$ subsets of $\mathbb{R} \times \mathbb{R}$, where
$\mathbb{R}$ denotes the reals.  Then (see, e.g. \cite{Z2}).

\begin{theorem}
  \label{thm:1.9}
  Let $M \models \ZFC + \GCH$.  Let $\mathbb{P}$ be a Suslin c.c.c.
  notion of forcing which adjoins reals, and let $\mathbb{Q}$ be the
  finite support iteration of $\mathbb{P}$ of length $\aleph_2$.  Then
  $M^{\mathbb{Q}} \models c(\Sym(\N)) = \aleph_1$.
\end{theorem}

On the other hand, the following is a theorem of $\ZFC$ (see, e.g.
\cite{BSZ}).

\begin{theorem}
  \label{thm:1.10}
  $\non(\mathcal{M}) \leq \mathfrak{a}_p, \mathfrak{a}_g$.
\end{theorem}

As a corollary of Theorems \ref{thm:1.9} and \ref{thm:1.10}, we know
the following.

\begin{corollary}
  It is consistent with $\ZFC$ that $c(\Sym(\N)) = \aleph_1 <
  \mathfrak{a}_p = \mathfrak{a}_g = 2^{\aleph_0} = \aleph_2$.
\end{corollary}

\begin{proof}
  Iteratively add $\aleph_2$ random reals with finite support to a
  ground model $M \models \ZFC + \GCH$.  (That $\non(\mathcal{M}) =
  2^{\aleph_0}$ follows from the fact that any small set of reals
  already appears at some intermediate stage, and that in a random
  extension the set of ground model reals is meager, see \cite[Theorem
  3.20]{KK84}.)
\end{proof} 

The obvious question left to answer is whether $\ZFC \vdash
c(\Sym(\N)) \leq \mathfrak{a}_p, \mathfrak{a}_g$.  The following
theorem shows that this does not hold.

\begin{theorem}[\cite{KY}]
  \label{theoremrelcof}
  \label{theorem intro KZ}
  It is consistent with $\ZFC$ that $\mathfrak{a}_p = \mathfrak{a}_g <
  c(\Sym(\N))$.
\end{theorem}

Having shown that $\mathfrak{a}_g$ can be rather small, we turn to
showing it can be very large.

The cardinals $\add(\mathcal{N})$, $\non(\mathcal{M})$, and
$\cof(\mathcal{N})$, in Definition \ref{def part Cichon}, are some of
the cardinals in Cicho\'n's diagram.  For these it is known (see
\cite[Lemma 1.3.2]{BJ}) that $\add(\mathcal{N}) \leq \non(\mathcal{M})
\leq \cof(\mathcal{N})$ (in fact $\add(\mathcal{N})$ is the smallest
cardinal in Cicho\'n's diagram, and $\cof(\mathcal{N})$ the largest).
We'll construct, in a model of $\ZFC + \CH$, for any two cardinals
$\lambda > \mu \geq \aleph_1$ a c.c.c. notion of forcing, using
Shelah's recent technique of template forcing (\cite{Sh}), such that in
the forcing extension all cardinals in Cicho\'n's diagram are equal to
$\mu$, and $\mathfrak{a}_g = \lambda = 2^{\aleph_0}$ --- that is,

\begin{theorem}
  It is consistent with $\ZFC$ that $\add(\mathcal{N}) = \cof(\mathcal{N}) <
  \mathfrak{a}_g = 2^{\aleph_0}$.
\end{theorem}

\subsection{Definability}

Just as for the other notions of almost disjointness, there is the
question of how definable a maximal cofinitary group can be.  Su Gao
and Yi Zhang proved the following

\begin{theorem}[\cite{SGYZ}]
  The axiom of constructibility implies that there exists a maximal
  cofinitary group with a coanalytic generating set.
\end{theorem}

We improved that to the following

\begin{theorem}[\cite{BKPI11}]
  The axiom of constructibility implies that there exists a coanalytic
  maximal cofinitary group.
\end{theorem}

It is conjectured that there does not exist a Borel maximal cofinitary
group.  With the following lemma proved by Andreas Blass, this shows
that the above theorem is conjecturally the best possible.

\begin{lemma}[Andreas Blass]
  Any analytic maximal cofinitary group is Borel.
\end{lemma}

Unfortunately the conjecture is quite far from being proved.  For
instance the following question is still open.

\begin{question}
  Does there exist a closed maximal cofinitary group?
\end{question}

There are some partial results though.  Su Gao obtained the following
theorem.

\begin{theorem}
  There does not exist a compact maximal cofinitary group.
\end{theorem}

We were able, with methods developed for one of the orbit results
below, to improve this to the following.  A set $S \subseteq \Sym(\N)$
is a $K_\sigma$ set if it is contained in a countable union of compact
sets.

\begin{theorem}
  There does not exist a $K_\sigma$ maximal cofinitary group.
\end{theorem}

Otmar Spinas proved the following theorem related to this question.

\begin{theorem}
  There does not exist a locally compact maximal cofinitary group.
\end{theorem} 

\subsection{Orbits and Isomorphism Types}

Since $\Sym(\N)$ has a natural action on $\N$, defined by $g \in
\Sym(\N)$ maps $n \in \N$ to $g(n)$, we can ask about the orbit
structure of a maximal cofinitary group with this action. We proved
the following results about this action.

\begin{theorem}[\cite{BKOI}]
  A maximal cofinitary group has finitely many orbits.
\end{theorem}

From the standard construction of maximal cofinitary groups using
Martin's axiom or the continuum hypothesis, it is clear that
non-transitive maximal cofinitary groups exist.  The standard
construction easily yields a maximal cofinitary group with any finite
number of finite orbits.  The following theorem required a new idea though.

\begin{theorem}[\cite{BKOI}]
  Martin's axiom implies that for every $n,m \in \N$ with $m \geq 1$
  there exists a maximal cofinitary group with $n$ finite orbits and
  $m$ infinite orbits.
\end{theorem}

These theorems completely characterize the possible orbits of a
maximal cofinitary group on $\N$ with respect to cardinality.  The action of
subgroups of $\Sym(\N)$ on $\N$ generalizes to higher powers of $\N$.

\begin{definition}
  If $G \leq \Sym(\N)$ then the \emph{diagonal action} of $G$ on $\N^k$ for
  some $k \in \N$ is given by $g(n_0, \ldots, n_{k-1}) = (g(n_0),
  \ldots, g(n_{k-1}))$.
\end{definition}

The above theorems are then a first step towards answering the
following question.

\begin{question}
  Does a maximal cofinitary group have finitely many orbits under the
  diagonal action on $\N^k$ for any $k \in \N$?
\end{question}

This question is related to the question on the complexity of maximal
cofinitary groups in the following way (the statements and their
proofs can be found in \cite{WH}).  A group for which the answer to
the question is affirmative is called oligomorphic.  A closed subgroup
of $\Sym(\N)$ is the automorphism group of a countable first-order
structure.  Such a closed subgroup is oligomorphic if and only if it
is the automorphism group of a countably categorical structure.  So
the answer to the question will give information on what sort of
groups closed maximal cofinitary groups could be, of which type of
structure they could be automorphism groups.

Any subgroup of $\Sym(\N)$ is, obviously, of a certain abstract
isomorphism type.  The result of the usual constructions (both forcing
and from Martin's axiom and the continuum hypothesis) is a group that
has a generating set freely generating the group.  In answering
questions such as the one about the complexity of maximal cofinitary
groups, only working with free maximal cofinitary groups is a
restriction: it is possible that a maximal cofinitary group of least
complexity is not a free group.  Having good methods available to
construct maximal cofinitary groups of different isomorphism types
would therefore be beneficial for this work.  So we work on the
following question.

\begin{question}
  What are the possible isomorphism types of maximal cofinitary
  groups?
\end{question}

Related to this is the following question, which is also of
independent interest.

\begin{definition}
  If $G$ is an abstract group, it \emph{has a cofinitary action} if
  there is an embedding $G \hookrightarrow \Sym(\N)$ such that the
  image of the group under this embedding is a cofinitary group.
\end{definition}

\begin{question}
  Which abstract groups have cofinitary actions?
\end{question}

Here we present our results which are some initial steps towards
answering these questions.

\begin{theorem}[\cite{BKOI}]
  Martin's axiom implies that there exists a maximal cofinitary group
  into which any countable group embeds.
\end{theorem}

Any countable group has an obvious cofinitary action, its translation
action on itself; so here we do not get new information about which
groups can act cofinitarily.  We do get a cofinitary group which is
not free.  We also proved that a group for which it is not a priori
clear that it has a cofinitary action can consistently have a
cofinitary action.

\begin{theorem}[\cite{BKOI}]
  There exists a c.c.c. notion of forcing such that in the forcing
  extension the group $\bigoplus_{\alpha \in \aleph_1} \mathbb{Z}_2$
  has a cofinitary action.
\end{theorem}

The proof of this result yields also the following theorem.

\begin{theorem}[\cite{BKOI}]
  In any model of Martin's axiom and the negation of the continuum
  hypothesis the group $\bigoplus_{\alpha \in \aleph_1} \mathbb{Z}_2$
  has a cofinitary action.
\end{theorem}

As Andreas Blass has observed, this group cannot have a maximal
cofinitary action.


\chapter{Very Mad Families}
\label{chap2}

In this chapter we prove our results on very mad families.  For
convenience we will repeat the relevant definitions and give a short
list of notational conventions used.

\begin{definition}
  $\BS$ is the space of functions from $\N$ to $\N$, called
  \emph{Baire space}.
\end{definition}

Often we will think of $f \in \BS$ as a subset of $\N \times \N$ (the
subset $\{ (n,f(n)) \mid n \in \N\}$).  In the definitions below we
will give statements equivalent to the versions where we don't use
this idea; in the remainder of the chapter, however, we will not give
these equivalents.

\begin{definitions}
\item $g_0, g_1 \in \BS$ are \emph{almost disjoint} (also called
  \emph{eventually different}) if $g_0 \cap g_1$ is finite ($\{n \in
  \N \mid g_0(n) = g_1(n) \}$ is finite).

\item
  $\A \subseteq \BS$ is an \emph{almost disjoint family of functions}
  if any two distinct $g_0, g_1 \in \A$ are almost disjoint.
  
\item If $f \in \BS$ and $\A \subseteq \BS$, then $f$ is
  \emph{finitely covered} by $\A$ if there are $g_0, \ldots, g_n \in
  \A$ such that $f \setminus \bigcup_{i \leq n} g_i$ is finite ($\{k
  \in \N \mid f(k) \not \in \{g_i(k) \mid i \leq n\}$ is finite).

\item
  If $F, \A \subseteq \BS$ then $F$ is \emph{finitely covered} by $\A$
  if there exists an $f \in F$ that is finitely covered by $\A$.  (The
  notion of finitely covered is usually used in a negative context: if
  $F, \A \subseteq \A$ then $F$ is not finitely covered by $\A$ if no
  element of $F$ is finitely covered by $\A$.)

\item
  $\A \subseteq \BS$ is a \emph{very mad family} if $\A$ is an almost
  disjoint family of functions, and for any $F \subseteq \BS$ such
  that $|F| < |\A|$ and $F$ is not finitely covered by $\A$, there
  exists $g \in \A$ such that for all $f \in F$, the set $f \cap g$
  is infinite ($\{ n \in \N \mid f(n) = g(n) \}$ is infinite).
\end{definitions}

Now we will give our guide to notation for this
chapter. \label{chap2notation}

$\A$ will be the very mad family under consideration (or parts thereof
that we have already constructed).  $A$ will be a finite subset of
$\A$.  $s$ will be a finite partial function $\N \partmap \N$.  $g$ or
$h$ (with sub- or superscripts) will be an element of $\A$.  $f$ will be any
other element of $\BS$, often an element of $F$, which will be a
family of functions not finitely covered by $\A$.  $\bar{p} = \langle
p_0, \ldots, p_n \rangle$ will be the sequence of length $\lh(\bar{p}) = n
+ 1$.  We write $\pi_i$ for the $i$th projection function;
$\pi_i(\bar{p}) = p_i$.

\section{Martin's Axiom}
\label{sec:CHMA}

In this section we prove that $\MA$ implies the existence of very mad
families.

We define for any subset $\A$ of Baire space a notion of forcing
$\mathbb{P}_\A$, called \emph{eventually different forcing}, see
\cite[page 366]{BJ} and \cite{AM81} (which will also be used in Sections
\ref{sect:orthogonality results} and \ref{sec:sizec}).  It consists of
all conditions of the form $\langle s,A \rangle$ such that
\begin{itemize}
\item
   $s$ is a finite partial function $\mathbb{N} \rightharpoonup
   \mathbb{N}$, and
\item
   $A$ is a finite subset of $\A$.
\end{itemize}
The ordering\footnote{$q \leq p$ means that $q$ is an extension of
$p$.} $\langle s_2,A_2 \rangle \leq \langle s_1,A_1 \rangle$ is
defined by
\begin{equation*}
   s_1 \subseteq s_2 \ \wedge \ A_1 \subseteq A_2 \ \wedge \ \forall g
     \in A_1 \ [g \cap s_2 \subseteq s_1].
\end{equation*}

This notion of forcing is $\sigma$-centered (a strengthening of
c.c.c., for the definition see \cite[p. 23]{BJ}) since there are only
countably many choices for $s$ in a condition and any two conditions
with identical first coordinate are compatible.  This shows that in
fact for the following lemma and theorem the hypothesis $\mathfrak{p}
= 2^{\aleph_0}$ (here $\mathfrak{p}$ is pseudointersection number)
suffices, since this equality implies $\MA(\sigma\text{-centered})$
(\cite{ABhandbook} has a nice presentation of forcing axioms where
this implication is proved; this result originally appeared in
\cite{B12}).

Since $\CH$ implies $\MA$, the results below are also true under
$\CH$.  Of course in the case of $\CH$ you can do essentially the same
construction without mentioning a notion of forcing.

\begin{lemma}[\MA{}]
  \label{lem:stepma}
  Assume that $\A$ is an almost disjoint family of functions with
  $|\A{}| < 2^{\aleph_0}$ and that $F$ is a family not finitely
  covered by $\A$ with $|F| < 2^{\aleph_0}$. Then there exists a
  function $g \not \in \A$ such that:
  {
    \renewcommand{\theenumi}{(\roman{enumi})}
  \begin{enumerate}
    \item $\A \cup \{g\}$ is an almost disjoint family of functions,
      and
    \item for all $f \in F$, the set  $f \cap g$ is infinite.
  \end{enumerate}
  }
\end{lemma}

\begin{proof}
  For each $f \in F$, $h \in \A$ and $n \in \N$ let
  \begin{itemize}
    \item $C_h := \{ \langle s, A \rangle \in \mathbb{P}_\A \mid h \in A \}$;
    \item $D_n := \{ \langle s, A \rangle \in \mathbb{P}_\A \mid n \in
      \dom(s)\}$;
    \item $E_{f,n} := \{ \langle s, A \rangle \in \mathbb{P}_\A \mid
      \exists m \geq n \ f(m) = s(m) \}$.
  \end{itemize}
  
  All these sets are dense in $\mathbb{P}_\A$: Let $\langle {s},{A}
  \rangle \in \mathbb{P}_\A$.  For $C_h$ we have $\langle {s},{A}
  \rangle \geq\langle {s},{A\cup\{h\}} \rangle \in C_h$. For $D_n$ we
  need to find an extension of $s$ such that $n$ is in its domain (if
  this is not already the case).  Since $A$ is finite, there exists an
  $m$ such that $m \neq g_0(n)$ for all $g_0 \in A$ and we take $s
  \cup \{(n,m)\}$.  For $E_{f,n}$ note that since $f$ is not finitely
  covered by $\A$, there is an $m > n$ such that $f(m) \not \in
  \{g_0(m) \mid g_0 \in A\}$.  Then $ \langle {s},{A} \rangle \geq
  \langle {s \cup \{(m,f(m))\}},{A} \rangle \in E_{f,n}$.

  The family
  \begin{equation*}
    \mathcal{D} = \{ C_h \mid h \in \A\} \cup \{D_n \mid n \in
      \mathbb{N} \} \cup \{ E_{f,n} \mid f \in F, n \in \mathbb{N} \}
  \end{equation*}
  has cardinality less than $2^{\aleph_0}$.  Therefore by $\MA$ there
  is a filter $G$ meeting all of the sets in $\mathcal{D}$.  Then $g
  := \bigcup \{s\mid \exists A \subseteq \A \ [\langle {s},{A} \rangle
  \in G]\}$ is the desired function ($\{D_n \mid n \in \N\}$ ensure
  that it is a function with domain all of $\N$, $\{C_h \mid h \in
  \A\}$ together with the definition of the order ensure $g$ is almost
  disjoint from all members of $\A$, and for each $f \in F$,
  $\{E_{f,n} \mid n \in \N\}$ ensure that $f \cap g$ is infinite).
\end{proof}

\begin{theorem}[\MA{}]
  There exists a very mad family, and any such family is of
  cardinality $2^{\aleph_0}$.
\end{theorem}

\begin{proof}
  We shall construct functions $g_\alpha$ such that $\{ g_\alpha \mid
  \alpha < 2^{\aleph_0} \}$ is a very mad family of size
  $2^{\aleph_0}$.  Let $\{f_\alpha \mid \alpha < 2^{\aleph_0} \}$ be
  an enumeration of $\BS$. At stage $\beta$ we do the following.

  Let $F$ be the maximal subset of $\{ f_\alpha \mid \alpha <
  \beta\}$ that is not finitely covered by $A_\beta := \{g_\alpha \mid
  \alpha < \beta \}$ ($F$ is the set of functions in $\{f_\alpha \mid
  \alpha < \beta\}$ that are not finitely covered by $A_\beta$).  By
  Lemma \ref{lem:stepma} there exists $g_\beta$ such that $A_\beta
  \cup \{g_\beta\}$ is an almost disjoint family of functions, and for
  any $f \in F$ the set $f \cap g_\beta$ is infinite (if $F$ is
  empty, we just get a new function almost disjoint from all of
  $A_\alpha$).

  Now let $\A = \{ g_\beta \mid \beta < 2^{\aleph_0} \}$; we claim
  that $\A$ is a very mad family.  Let $F$ be not finitely covered by
  $\A$ and such that $|F| < |\A{}| = 2^{\aleph_0}$.  Since $F$ is of
  cardinality less than $2^{\aleph_0}$, we have $F \subseteq \{
  f_\alpha \mid \alpha < \beta \}$ for some $\beta < 2^{\aleph_0}$
  ($\MA$ implies that $2^{\aleph_0}$ is regular), and then $g_\beta$ will
  meet all members of $F$ infinitely often.

  The second clause follows immediately from Lemma \ref{lem:stepma}.
\end{proof}

\section{In the Cohen Model}

We prove that in any model of the continuum hypothesis
there is a very mad family that survives Cohen
forcing. For this we need the following lemma, which is from
\cite[Lemma 2.2, p. 256]{KK80}.

We let $\Fn(I,2)$ denote the set of finite partial functions $I
\rightharpoonup 2$ ordered by reverse inclusion.

\begin{lemma}
  \label{lem:kk}
  Let $M$ be a model of $\ZFC$, $I, S \in M$, $G$ be
  $\Fn(I,2)$-generic over $M$, and $X \subseteq S$ with $X \in M[G]$.
  Then $X \in M[G \cap \Fn(I_0,2)]$ for some $I_0 \subseteq I$ such
  that $I_0 \in M$ and $(|I_0| \leq |S|)^M$.\hfill $\qed$
\end{lemma}

\begin{theorem}
  Let $M$ be a model of $\ZFC + \CH$.  Then there is a very mad family
  $\A$ in $M$ such that for any $I \in M$ and $\Fn(I,2)$-generic set
  $G$, $M[G] \models$ ``$\check{\A}$ is a very mad family of size
  $\aleph_1$''.
\end{theorem}

\begin{proof}
  We construct a very mad family that survives forcing with
  $\Fn(\mathbb{N},2)$, and then show that this family survives forcing
  with $\Fn(I,2)$ for any $I$.  Note that since the continuum
  hypothesis is true in $M$, strongly and very mad families in $M$ are
  the same thing, and that the very mad family in $M$ will be of size
  $\aleph_1$, so it only has to contain functions capturing any
  countable collection in the extension.

  Since in $M$
  \begin{gather*}
    |\Fn(\N,2) \times \{\tau \mid \tau \text{ is a nice name
      for a subset of } (\mathbb{N} \times \mathbb{N})\highcheck \
      \}| \\
    \leq \aleph_0 \times \left( (\# \text{ anti-chains in } \Fn(\N,2)) \times
      |\{0,1\}^{\aleph_0}|\right)^{|\N \times \N|} \\
    \leq \aleph_0 \times (2^{\aleph_0} \times 2^{\aleph_0})^{\aleph_0} \stackrel{\CH}{=}
      \aleph_1,
  \end{gather*}
  we can enumerate $\Fn(\mathbb{N},2) \times \{\tau \mid \tau \text{
  is a nice name for a subset of } (\mathbb{N} \times
  \mathbb{N})\highcheck \ \}$ as $\langle (p_i, \tau_i) \mid i <
  \omega_1 \rangle$.

  We construct $\A = \{g_\alpha \mid \alpha < \omega_1 \}$, the very
  mad family, recursively.
  
  Assume that all $g_\alpha$ for $\alpha<\beta$ have been defined.  We need
  $g_\beta$ to satisfy:
  \renewcommand{\theenumi}{G\arabic{enumi}}
  \begin{enumerate}
  \item \label{global:1} for all $\alpha < \beta$, the functions
    $g_\alpha$ and $g_\beta$ are almost disjoint, and
  \item \label{global:2}
     if $F_\beta := \{\tau_\alpha \mid\alpha < \beta, p_\beta \forces
     ``\tau_\alpha $ is a total function and $ \tau_\alpha$ is not
     finitely covered by $\{\check{g}_\gamma \mid \gamma <
     \check{\beta} \}$''$\}$, then $(\forall \tau_\alpha \in F_\beta)\
     p_\beta \forces `` |\tau_\alpha \cap \check{g}_\beta| =
     \check{\omega}$''.
   \end{enumerate}

     $p_\beta \forces `` |\tau_\alpha \cap \check{g}_\beta| =
     \check{\omega}$'' is equivalent to 
     \begin{equation*}
       (\forall n) ( \forall q \leq p_\beta) (\exists r \leq q) (\exists m
       \geq n) \ r \forces \text{``}\tau_\alpha(\check{m}) =
       \check{g}_\beta(\check{m})\text{''}.
    \end{equation*}

  Enumerate $\mathbb{N} \times \{q \mid q \leq p_\beta\}$ as $\langle
  (n_i,q_i) \mid i < \omega\rangle$, $\langle g_\alpha \mid \alpha <
  \beta \rangle$ as $\langle g_i' \mid i < \omega\rangle$ and
  $F_\beta$ as $\langle \tau_i' \mid i < \omega\rangle$.

  Recursively define $g_\beta$.  Before stage $s$ we have $g_\beta$
  defined on $\{0, \ldots, n_s\}$.  At stage $s$ we want to define
  $g_\beta$ on $\{n_s+1, \ldots, n_{s+1}\}$ for some $n_{s+1}$ so that
  \ref{global:1} and \ref{global:2} will eventually be satisfied.

  The requirements at this stage will be:

  \renewcommand{\theenumi}{L\arabic{enumi}}
  \begin{enumerate}
  \item \label{local:1}
     $\forall n \in \{n_s+1, \ldots, n_{s+1}\}\ (g_\beta(n) \not \in
     \{g_0'(n), \ldots, g_s'(n)\})$, and
  \item \label{local:2}
     there are $n_{s,0}, \ldots, n_{s,s} > n_s$ all distinct and
     $r_{s,0}, \ldots, r_{s,s} \leq q_s$ such that for all
     $i$ from $0$ to $s$ we have $r_{s,i} \forces ``
     \tau_i'(\check{n}_{s,i}) = \check{g}_\beta(\check{n}_{s,i})$''.
  \end{enumerate}
  \renewcommand{\theenumi}{\arabic{enumi}}
  Requirement \ref{local:1} ensures that $g_\beta$ will satisfy
  \ref{global:1}, and requirement \ref{local:2} ensures that $g_\beta$
  will satisfy \ref{global:2}.
 
  Since $p_\beta \forces\ $``$\tau_i' \text{ is a total function not finitely
  covered by}\, \check{g_0'}, \ldots, \check{g_s'}$'', for every
  $m > n_s$ there is an $n_{s,i} > m$ and $r'_{s,i} \leq q_s$ such that
  $r'_{s,i} \forces `` \tau_i'(\check{n}_{s,i}) \not \in \{
  \check{g_i'}(\check{n}_{s,i}) \mid 0 \leq i \leq s\}$''.  Then below
  $r'_{s,i}$ there is a condition $r_{s,i}$ that decides the value of
  $\tau_i'(\check{n}_{s,i})$.
  
  We use this observation repeatedly to find $n_{s,0} < \cdots <
  n_{s,s}$ all larger than $n_s$, and define $g_\beta(n_{s,i})$ to be
  the number that $\tau_i'(\check{n}_{s,i})$ is forced to be by
  $r_{s,i}'$.  Then we set $n_{s+1} = n_{s,s}$ and set $g_\beta(n)$,
  for $n < n_{s+1}$ such that $g_\beta(n)$ is not defined yet, to be
  any number not in $\{g_i'(n)\mid 0 \leq i \leq s\}$.  This completes
  the construction.
  
  We show that $\A = \{ g_\beta \mid \beta < \omega_1 \}$ is a very
  mad family in the forcing extension $M[G]$, for any $G$ that is
  $\Fn(\mathbb{N},2)$-generic over $M$.  First note that $\A$ is an
  almost disjoint family in $M[G]$, since it is almost disjoint in
  $M$: the functions are almost disjoint by the first requirement.  To
  see that it is very mad, let $F$ be a countable family of functions,
  all of which are not finitely covered by ${\A}$.  Then in
  $M[G]$ we have $F \subseteq \{ \tau_\beta[G] \mid \beta < \omega_1 \}$;
  therefore there exists an $\alpha$ (by the countability of $F$) such
  that $F \subseteq \{ \tau_\beta[G] \mid \beta < \alpha \}$, and this is
  forced by some $p \in G$.  Then at some point $\beta$ in the
  construction (after stage $\alpha$) when we have a $p_\beta$ equal
  to $p$ we will correctly deal with a superset of $F$ (and therefore
  with $F$), by the second requirement.
  
  It remains to show that for any $I$ the family $\A$ is very mad in
  any forcing extension $M[G]$ with $G$ $\Fn(I,2)$-generic over $M$.
  Suppose that in some forcing extension by $\Fn(I,2)$ the family $\A$
  defined above is no longer very mad.  There is then a countable
  family $F$ of functions not finitely covered by $\A$ for which there
  does not exist $g \in \A$ such that for all $f \in F$ the set $f \cap
  g$ infinite.  Code the family in a single real.  Then there is a
  countable $I_0$ such that this real, and therefore the family $F$,
  are in the extension of $M$ by $\Fn(I_0,2) \cap G$ (by Lemma
  \ref{lem:kk}).  Since $\Fn(I_0,2) \cong \Fn(\N,2)$ the above
  argument for $\Fn(\N,2)$ shows that there is a function $g \in \A{}$
  such that for all $f \in F$ the set $f \cap g$ is infinite.  This
  contradicts the existence of such a family $F$. (Note that this also
  shows $\A$ is very mad in $M$, by taking $I = \emptyset$.)
\end{proof}


\section{Orthogonality Results}
\label{sect:orthogonality results}

In this section we'll prove the existence, under Martin's axiom, of
many orthogonal very mad families (see Definition \ref{definition
  orthogonal}).  This intuitively shows that there can be much
``room'' outside of a very mad family.  We also show that under the
continuum hypothesis, for any very mad family there exists a very mad
family orthogonal to it.

Let $\langle \A_\alpha \mid \alpha < \beta \rangle$ be subsets of
Baire space.  Define the notion of forcing $\mathbb{Q}_{\langle
  \A_\alpha \mid \alpha < \beta \rangle}$, to consist of all
conditions $\bar{p}$ that are finite partial functions with domain
contained in $\beta$ and $\bar{p}(\alpha) \in \mathbb{P}_{\A_\alpha}$
(from Section \ref{sec:CHMA}).  Define $\bar{q} \leq \bar{p}$ iff
$\dom(\bar{p}) \subseteq \dom(\bar{q})$ and $\forall \alpha \in
\dom(\bar{p}) \ [\bar{q}(\alpha) \leq_{\mathbb{P}_{\A_\alpha}}
\bar{p}(\alpha)]$.

\begin{lemma}[$\MA$]
  \label{lem:maOrthStep}
  Let $\A_\alpha$ ($\alpha < \beta < 2^{\aleph_0}$) be almost disjoint
  families that are pairwise orthogonal, and let $F_\alpha$ ($\alpha <
  \beta$) be families of functions such that $| F_\alpha | <
  2^{\aleph_0}$ and $F_\alpha$ is not finitely covered by $\A_\alpha$.
  Then there exist functions $g_\alpha \not \in \A_\alpha$ ($\alpha <
  \beta$) such that
  {
    \renewcommand{\theenumi}{(\roman{enumi})}
  \begin{enumerate}
    \item
      for all $\alpha < \beta$, the family $\A_\alpha \cup
      \{g_\alpha\}$ is an almost disjoint family,
    \item
      for all $\alpha < \beta$ and $f \in F_\alpha$, the set $f \cap
      g_\alpha$ is infinite, and
    \item
      for all $\alpha_1, \alpha_2 < \beta$, the families $A_{\alpha_1}
      \cup \{g_{\alpha_1}\}$ and $A_{\alpha_2} \cup \{g_{\alpha_2}\}$
      are orthogonal.
  \end{enumerate}
  }
\end{lemma}

\begin{proof}
  In addition to the dense sets used in the proof of Lemma
  \ref{lem:stepma} for each coordinate $\alpha < \beta$, define for all
  $\alpha_1 \neq \alpha_2 < \beta, n \in \N$ and $a \in
  [\A_{\alpha_2}]^{<\omega}$ the sets (remember $\pi_i$ is the $i$th
  projection function, see page \pageref{chap2notation})

  \begin{align*}
    A_{\alpha_1, \alpha_2, a, n}& := \{ \bar{p} \mid a \subseteq 
      \pi_1(\bar{p}(\alpha_2)) \ \wedge  \exists m > n \big[ m \in
      \dom(\pi_0(\bar{p}(\alpha_1))) \ \wedge \\
    & \!\! m \in \dom(\pi_0(\bar{p}(\alpha_2))) \ \wedge  \pi_0(\bar{p}(\alpha_1))(m) \not \in \{ f(m)
    \mid f \in a \} \cup \{\pi_0(\bar{p}(\alpha_2))(m)\} \big] \}.
  \end{align*}
  These sets are easily seen to be dense, and combined with the dense
  sets from Lemma \ref{lem:stepma} there are still fewer than
  continuum many.  So $\MA$ implies there is a filter $G$ meeting them
  all.  Then set $g_\alpha := \cup \{ \bar{p}(\alpha)
  \mid \bar{p} \in G\}$.

  The dense sets from Lemma \ref{lem:stepma} guarantee the first two items of
  the theorem, and the new dense sets $A_{\alpha_1, \alpha_2, a, n}$
  ensure that the resulting families are still orthogonal: the filter
  intersecting all dense sets $A_{\alpha_1, \alpha_2, a, n}$ for $n
  \in \N$ ensures that $g_{\alpha_1}$ is not finitely covered by $a
  \cup \{g_{\alpha_2}\}$.
\end{proof}

\begin{theorem}[$\MA$]
  There exist very mad families $\A_\alpha$ ($\alpha < 2^{\aleph_0}$)
  such that for all $\alpha_1 \neq \alpha_2 < 2^{\aleph_0}$ the
  families $\A_{\alpha_1}$ and $\A_{\alpha_2}$ are orthogonal.
\end{theorem}

\begin{proof}
  We will construct the families $\A_\alpha$ recursively as $\A_\alpha
  := \bigcup_{\gamma < 2^{\aleph_0}} \A_{\alpha, \gamma}$.
  
  We start with $\A_{\alpha,0} := \emptyset$ for all $\alpha <
  2^{\aleph_0}$, and we take unions at limit stages.  Then at step
  $\beta$ we apply Lemma \ref{lem:maOrthStep} with
  $\mathbb{Q}_{\langle \A_{\alpha, \beta} \mid \alpha < \beta
    \rangle}$ to obtain $\langle g_\alpha \mid \alpha < \beta \rangle$
  and set $\A_{\alpha,\beta+1} := \A_{\alpha, \beta} \cup
  \{g_\alpha\}$ for $\alpha < \beta$, and $\A_{\alpha, \beta+1} :=
  \A_{\alpha, \beta} = \emptyset$ for $\beta \leq \alpha < 2^{\aleph_0}$.
\end{proof}

Now we move towards proving the following theorem whose proof will
take up the remainder of this section.

\begin{theorem}[$\CH$]
  For every very mad family $\A$ there exists a very mad family $\B$
  such that $\A$ and $\B$ are orthogonal.
\end{theorem}

We will consistently use the convention that barred functions are to
be thought of as related to $\B$ where the unbarred versions relate to
$\A$.  First we work out the ideas needed to be able to construct a
family $\B$ not finitely covering $\A$.

\begin{lemma}
  \label{lem:basavoidfc}
  Let $\bar{g}_0, \ldots, \bar{g}_n \in \BS$, $g_0, \ldots g_n \in \A$
  and $W \in [\N]^{\aleph_0}$ such that for all $i \leq n$,
  $\dom(\bar{g}_i \cap g_i) \supseteq W$.  Then
  \begin{equation*}
    g \in \A \ \wedge \ g \restrict W \subseteq^* \bigcup_{i \leq n}
    \bar{g}_i \ \Rightarrow \ \text{there is } i \leq n \text{ such
    that } g = g_i
  \end{equation*}
\end{lemma}

\begin{proof}
  The partial function $g \restrict W$ agrees on infinitely many
  inputs with some $\bar{g}_i \restrict W$.  This means that $g
  \restrict W$ agrees on infinitely many inputs with $g_i$, which in
  turn gives $g = g_i$ by almost disjointness of $\A$.
\end{proof}

From this we immediately get:

\begin{corollary}
  \label{cor:infroom}
  Under the same hypothesis, if $g \in \A$ is not one of the functions
  $g_0, \ldots, g_n$ then $g \restrict W \setminus \bigcup_{i \leq n}
  \bar{g}_i$ is infinite.
\end{corollary}

\begin{definition}
  For a family $\B \subseteq \BS$, a function $H : [\B]^{<\aleph_0}
  \rightarrow ([\N]^{\aleph_0})^2 \times ([\A\,]^{<\aleph_0})^2$ is
  \emph{good for $\B$} if for all $\{\bar{g}_0, \ldots, \bar{g}_n\}
  \in [\B]^{<\aleph_0}$ we have that $H(\{\bar{g}_0, \ldots,
  \bar{g}_n\}) = \langle W_0, W_1, \{g_0^0, \ldots, g_n^0\},$ $
  \{g_0^1,\ldots, g_n^1\} \rangle$ such that for all $i,j \leq n$ and
  $k,l \in \{0,1\}$,
  \begin{equation*}
    i\neq j \ \vee \ k \neq l \ \Rightarrow \ g_i^l \neq g_j^k,
  \end{equation*}
  and for all $ i \leq n$,
  \begin{equation*}
    \dom(\bar{g}_i \cap g_i^0) \supseteq W_0 \ \wedge \
    \dom(\bar{g}_i \cap g_i^1) \supseteq W_1.
  \end{equation*}
\end{definition}

\begin{lemma}
  \label{lem:goodnfc}
  If $\B$ is such that an $H$ good for $\B$ exists,
  then $\B$ does not finitely cover $\A$.
\end{lemma}

\begin{proof}
  Suppose that $g\in \A$ is such that there are $\bar{g}_0, \ldots
  \bar{g}_n \in \B$ such that $g \subseteq^* \bigcup_{i \leq n}
  \bar{g}_i$.  Let $H(\{\bar{g}_0, \ldots, \bar{g}_n\}) = \langle W_0,
  W_1, \{g_0^0, \ldots, g_n^0\}, \{g_0^1, \ldots, g_n^1\} \rangle$.
  Then as $g \restrict W_0 \subseteq^* \bigcup_{i \leq n} \bar{g}_i$,
  we have $g = g_i^0$ for some $i \leq n$ (by Lemma
  \ref{lem:basavoidfc}), but also $g \restrict W_1 \subseteq^*
  \bigcup_{i \leq n} \bar{g}_i$, which gives $g = g_j^1$ for some $j
  \leq n$ (Lemma \ref{lem:basavoidfc} again).  So $g_i^0 = g_j^1$,
  contradicting $H$ being good for $\bar{A}$.
\end{proof}

Now that we know how to take care of not finitely covering $\A$ (it is
enough to ensure existence of a function good for $\B$), we look
into being not finitely covered by $\A$.

\begin{lemma}
  If $\bar{g} \in \BS$ is such that there exist $\langle g_n \mid n
  \in \N \rangle$ with all $g_n \in \A$ different and $\bar{g}$ agrees
  with each $g_n$ on infinitely many inputs, then $\bar{g}$ is not
  finitely covered by $\A$.
\end{lemma}

\begin{proof}
  Suppose $\bar{g} \subseteq^* \bigcup_{i \leq n} g_i'$ for $g_i' \in
  \A$.  Then every $g_n$ agrees infinitely often with some $g_j'$.  As
  there are infinitely many $g_n$, some $g_j'$ infinitely often agrees
  with two different $g_n$.  This contradicts the almost disjointness
  of $\A$.
\end{proof}

Now we are ready for the construction.  With $\CH$ we can enumerate
$\BS$ by $\langle f_\alpha \mid \alpha < \omega_1 \rangle$ and also
using $\omega_1 \cong \N \times \omega_1$ we can enumerate $\A =
\langle g_{n,\alpha} \mid n \in \N \, \wedge \, \alpha < \omega_1
\rangle$ without repetitions.

We will construct $\B$ and $H$ by recursion as $\langle
\bar{g}_\alpha \mid \alpha < \omega_1 \rangle$ and $\bigcup_{\alpha<
\omega_1} H_\alpha$.

We call $H_\alpha$ good if it is good for $\langle \bar{g}_\beta \mid
\beta < \alpha \rangle$ with the following added requirements:
\renewcommand{\theenumi}{H\arabic{enumi}}
\begin{enumerate}
  \item
    \label{Hreq:1}
    if $S \in [\{ \bar{g}_\beta \mid \beta < \alpha\}]^{<\aleph_0}$,
    then $\pi_2(H(S)) \cup \pi_3(H(S)) \subseteq \{ g_{n,\beta} \mid n
    \in \N \ \wedge \ \beta < \alpha \}$;
  \item
    \label{Hreq:2}
    if $S_0, S_1 \in [\{ \bar{g}_\beta \mid \beta <
    \alpha\}]^{<\aleph_0}$ and $S_0 \subseteq S_1$, then $\pi_0(H(S_1))
    \subseteq \pi_0(H(S_0))$ and $\pi_1(H(S_1)) \subseteq
    \pi_1(H(S_0))$.
\end{enumerate}

Note that if $\lambda$ is a limit ordinal and if all $H_\alpha$ for
$\alpha < \lambda$ are good, then $\bigcup_{\alpha < \lambda} H_\alpha$
is good for $\langle \bar{g}_\alpha \mid \alpha < \lambda \rangle$.
For limit ordinals $\lambda$ we take $H_\lambda = \bigcup_{\alpha <
\lambda} H_\alpha$.

Let $\alpha < \omega_1$ and assume that $\langle \bar{g}_\beta \mid
\beta < \alpha \rangle$ is an almost disjoint family orthogonal to
$\A$ and $H_\alpha$ is good.  It suffices to construct
$\bar{g}_\alpha$ and $H_{\alpha+1}$ such that
\renewcommand{\theenumi}{G\arabic{enumi}}
\begin{enumerate}
\item \label{G:1} $\bar{g}_\alpha$ is almost disjoint from $\langle
  \bar{g}_\beta \mid \beta < \alpha \rangle$,
\item \label{G:2} $H_{\alpha + 1}$ is good,
\item \label{G:3} $\bar{g}_\alpha$ agrees on infinitely many inputs
  with infinitely many members of $\A$, and
\item \label{G:4} all $f_\beta$, for $\beta < \alpha$, that are not
  finitely covered by $\langle \bar{g}_\beta \mid \beta < \alpha
  \rangle$ agree with $\bar{g}_\alpha$ on infinitely many inputs.
\end{enumerate}

Now in order to take care of this we do some reenumeration: enumerate
$\{ \bar{g}_\beta \mid \beta < \alpha \}$ by $\langle \bar{g}_n' \mid
n \in \N \rangle$, enumerate the set $\{ f_\beta \mid \beta < \alpha \ 
\wedge \ f_\beta \text{ is not finitely covered by }\! \{ \bar{g}_n'
\mid n \in \N \} \}$ by $\langle f_n' \mid n \in \N \rangle$ and
enumerate the set $[\{ \bar{g}_n' \mid n \in \N \}]^{<\aleph_0}$ by
$\langle S_n \mid n \in \N \rangle$.  Then write $W_{n,0}$ for
$\pi_0(H_\alpha(S_n))$ and $W_{n,1}$ for $\pi_1(H_\alpha(S_n))$.

We will construct $\bar{g}_\alpha$ recursively.  After step $s \in
\N$ we have $\bar{g}_\alpha$ defined on an initial segment $[0,
n_s]$.  During step $s+1$ we satisfy the following requirements.

\renewcommand{\theenumi}{L\arabic{enumi}}
\begin{enumerate}
  \item \label{L:1}
    $\bar{g}_\alpha(n) \not \in \{ \bar{g}_i'(n) \mid i \leq s \}$,
  \item \label{L:2} 
    For all $n \leq s$ there exist $k \in W_{n,0}$ and $l \in W_{n,1}$
    such that $\bar{g}_\alpha(k) = g_{0,\alpha}(k)$ and $\bar{g}_\alpha(l)
    = g_{1,\alpha}$,
  \item \label{L:3}
    For all $n \leq s$ there exists a $k$ where $\bar{g}_\alpha$ was
    not yet defined before this step, such that $\bar{g}_\alpha(k) =
    g_{n,\alpha}(k)$, and
  \item \label{L:4}
    For all $n \leq s$ there exists a $k$ where $\bar{g}_\alpha$ was
    not yet defined before this step, such that $\bar{g}_\alpha(k) =
    f_n(k)$.
\end{enumerate}
\renewcommand{\theenumi}{\arabic{enumi}}

A detailed construction taking care of \ref{L:1}--\ref{L:4} is the
following:
\begin{itemize}
  \item Since $H_\alpha$ is good, \ref{Hreq:1}, \ref{Hreq:2} and
    Corollary \ref{cor:infroom} give that $g_{0,\alpha} \restrict
    W_{n,0} \setminus (\bigcup S_n) \cup( \bigcup_{i \leq s}
    \bar{g}_s')$ is infinite:  \ref{Hreq:2} gives $\pi_0(H_\alpha(S_n \cup
    \{\bar{g}_i \mid i \leq s\})) \subseteq W_{n,0}$, \ref{Hreq:1}
    gives that $g_{0,\alpha}$ is not in $H_\alpha(S_n \cup \{ \bar{g}_i \mid
    i \leq s \})$ so Corollary \ref{cor:infroom} applies.  Similarly
    $g_{1,\alpha} \restrict W_{n,1} \setminus (\bigcup S_n) \cup
    (\bigcup_{i \leq s} \bar{g}_s')$ is infinite.

    So we can choose $w_{s,i}^{0} \in W_{i,0}$, $w_{s,i}^1 \in
    W_{i,1}$ for $i \leq s$, such that $n_s < w_{s,0}^{0} < \cdots <
    w_{s,s}^{0} < w_{s,0}^1 < \cdots < w_{s,s}^1$ with
    $g_{0,\alpha}(w_{s,i}^{0}) \not \in \{\bar{g}_i'(w_{s,i}^{0}) \mid
    i \leq s \}$ and $g_{0,\alpha}(w_{s,i}^{1}) \not \in
    \{\bar{g}_i'(w_{s,i}^{1}) \mid i \leq s \}$.

    Define $\bar{g}_\alpha(w_{s,i}^{0}) := g_{0,\alpha}(w_{s,i}^{0})$
    and $\bar{g}_\alpha(w_{s,i}^{1}) := g_{0,\alpha}(w_{s,i}^{1})$.
    This takes care of \ref{L:2} while respecting \ref{L:1}.
  \item As $g_{i,\alpha}$ is not finitely covered by $\{ \bar{g}_n'
    \mid n \in \N \}$ (Lemma \ref{lem:goodnfc}), for every $i \leq s$
    we can find $n_{s,i}$ such that $w_{s,s}^1 < n_{s,0} < \cdots <
    n_{s,s}$ and $g_{i,\alpha}(n_{s,i}) \not \in \{ \bar{g}_n'(n_s,i)
    \mid n \leq s\}$.

    Define $\bar{g}_\alpha(n_{s,i}) := g_{i,\alpha}(n_{s,i})$.  This
    takes care of \ref{L:3} while respecting \ref{L:1}.
  \item As $f_i'$ is not finitely covered by $\{ \bar{g}_n' \mid
    n \in \N \}$ for every $i \leq s$, we can find $m_{s,i}$
    such that $n_{s,s} < m_{s,0} < \cdots < m_{s,s}$ and
    $f_i'(m_{s,i}) \not \in \{ \bar{g}_j(m_{s,i}) \mid  j \leq s \}$.

    Define $\bar{g}_\alpha(m_{s,i}) := f_i'(m_{s,i})$.  This takes care
    of \ref{L:4} while respecting \ref{L:1}.
  \item Now for any $n$ such that $n_s < n < m_{s,s}$ where
    $\bar{g}_\alpha$ is not defined, define $\bar{g}_\alpha(n)$ to be
    the least number not in $\{\bar{g}_i'(n) \mid i \leq s\}$, and set
    $n_{s+1} = m_{s,s}$.
\end{itemize}

This construction satisfies the requirements \ref{L:1}--\ref{L:4}.
These in turn imply the requirements \ref{G:1}--\ref{G:4}; this is
immediate for all but \ref{G:2}, where we still have to define
\[
H_{\alpha+1} : [\{\bar{g}_\beta \mid \beta \leq \alpha\}]^{<\aleph_0}
\rightarrow ([\N]^{\aleph_0})^2 \times ([\{ g_{n,\beta} \mid n \in \N,
\beta \leq \alpha \}^{<\aleph_0}])^2.
\]
If $S \in [\{\bar{g}_\beta \mid \beta \leq \alpha\}]^{<\aleph_0} $
then $S \subseteq \{\bar{g}_\beta \mid \beta < \alpha\}$ or $S = S'
\cup \{\bar{g}_\alpha\}$ with $S' \subseteq \{\bar{g}_\beta\mid \beta
< \alpha\}$.

In the first case, set $H_{\alpha + 1}(S) := H_\alpha(S)$; otherwise
$S' = S_n$ for some $n \in \N$, and if $H_{\alpha}(S_n) = \langle W_0,
W_1, \{g_0^0, \ldots, g_n^0\},\{g_0^1, \ldots, g_n^1\}\rangle$ set
$H_{\alpha+1}(S) := \langle \{w_{s,n}^0 \mid n \leq s \}, \{ w_{s,n}^1
\mid n \leq s \}, \{g_0^0, \ldots, g_n^0, g_{0,\alpha} \}, \{g_0^1,
\ldots, g_n^1, g_{1,\alpha}\}\rangle$.

This $H_{\alpha+1}$ is good as can be easily checked, completing the
proof.

\section{A Very Mad Family Not of Size Continuum}
\label{sec:sizec}

In this section we show that consistently there exist very mad
families of any uncountable cardinality less than or equal to the
continuum.  The forcing is based on the proof of the similar result
for maximal cofinitary groups by Yi Zhang in \cite{Z}.

\begin{theorem}
  \label{thm:notSizec}
  Let $M$ be a model of \ZFC{} and assume that, in $M$, $\kappa$ is a
  regular cardinal such that $\aleph_1 \leq \kappa <
  2^{\aleph_0} = \lambda$.  Then there exists a c.c.c. forcing
  $\mathbb{P}$ such that $M^\mathbb{P}$ satisfies
  {
    \renewcommand{\theenumi}{(\roman{enumi})}
  \begin{enumerate}
    \item $2^{\aleph_0} = \lambda$, and
    \item there exists a very mad family $\mathcal{A}$ of cardinality
      $\kappa$.
  \end{enumerate}
  }
\end{theorem}

In the proof we will use the c.c.c. poset $\mathbb{P}_\A$ from section
\ref{sec:CHMA}.  First we prove its main property (which is
basically Lemma \ref{lem:stepma} rephrased in the language of
forcing).

\begin{lemma}
  \label{lem:genericobject}
  If $N$ is a model of \ZFC{}, $f \in N \cap {}^\mathbb{N}
  \mathbb{N}$, $\A \subset \BS$, $\A \in N$ and $f$ is not finitely
  covered by $\A$, then the generic function $g \not \in N$ obtained from
  forcing with $\mathbb{P}_\A$ over $N$ satisfies
  {
    \renewcommand{\theenumi}{(\roman{enumi})}
  \begin{enumerate}
  \item $\A \cup \{g\}$ is an almost disjoint family;
  \item $f \cap g$ is infinite.
  \end{enumerate}
  }
\end{lemma}

\begin{proof}
  We use the notation from the proof of Lemma \ref{lem:stepma}.  $g$
  is a total function since for every $n \in \N$, $D_n = \{ \langle
  {s},{H} \rangle \in \mathbb{P}_\A  \mid n \in \dom(s) \}$ is dense
  and in $N$.  $\A \cup \{g\}$ is almost disjoint since for every $h
  \in \A$ the set $C_{h} = \{ \langle {s},{H} \rangle \in
  \mathbb{P}_\A \mid h \in H \}$ is dense and in $N$.  $f \cap g$ is
  infinite as for every $n \in \N$ the set $E_{f,n} = \{\langle
  {s},{H} \rangle \in \mathbb{P}_\A \mid (\exists m \geq n) \ m \in
  \dom(s) \wedge s(m) = f(m)\}$ is dense (by not finitely covering)
  and in $N$.
\end{proof}

\begin{proof}[Proof of Theorem \ref{thm:notSizec}]
  From this lemma we get that forcing with $\mathbb{P}_\A$ produces an
  a.d. family $\A \cup \{g\}$ containing a function $g$ that agrees on
  infinitely many inputs with each function in the ground model that
  is not finitely covered by $\A$.

  $\mathbb{P}$ is the $\kappa$ step finite support iteration of the
  $\mathbb{P}_\A$, where $\A$ at step $\alpha$ consists of the
  generics added so far ($\A_\alpha = \{g_\beta \mid \beta <
  \alpha\}$).
  
  For any $\alpha$, $|\A_\alpha | = |\alpha|$.  So we can use as
  underlying set for $\mathbb{P}_{\A_\alpha}$ the set of $\langle
  {s},{H} \rangle$, $s : \N \rightharpoonup \N$ finite partial and $H
  \subseteq |\alpha|$ ($= |\A|$) finite.  This shows that
  $|\mathbb{P}_{\A_\alpha}| = \max\{\omega, |\alpha|\}$, from which
  $|\mathbb{P}| = \kappa$ follows: since we use a finite support
  iteration of length $\kappa$ we can use finite subsets of $\kappa$
  to indicate at which indices the element is non-maximal, and then
  for each of those an element from $\bigcup_{\alpha < \kappa}
  \mathbb{P}_{\A_\alpha}$ which is in the right coordinate.  That is
  we can take the underlying set of $\mathbb{P}$ to be $\{p \in
  (\bigcup_{\alpha < \kappa}
  \mathbb{P}_{\A_\alpha})^{[\kappa]^{<\omega}} \mid \forall \alpha \in
  \dom(p) \ p(\alpha) \in \mathbb{P}_{\A_\alpha}\}$ (this also shows
  with the previous comment that the poset can be taken to have its
  underlying set in the ground model).

  Since $\mathbb{P}$ is a finite support iteration of {c.c.c.} posets it
  preserves cardinals, and using the {c.c.c.} and that $|\mathbb{P}| =
  \kappa$ and $\kappa^\omega = 2^{\aleph_0}$ it follows that in the
  forcing extension $2^{\aleph_0}$ is still $\lambda$.

  Now we see that the resulting collection $\A = \{ g_\alpha \mid \alpha <
  \kappa \}$ is very mad as follows.

  Let $F$, $|F| < |\A| = \kappa$, be a collection of
  functions not finitely covered by $\A$ (so also not finitely covered
  by any subcollection of $\A$).  Since the forcing is finite support
  and $|F| < \cf(\kappa) = \kappa$, $F$ already appears at some
  stage before $\kappa$.  Then the generic added at any later stage
  will agree on infinitely many inputs with each member of $F$
  (by the lemma).
\end{proof}

\section{Very Mad Families Can be Coanalytic}
\label{sect:VisL}
\label{sect:VMAD:VisL}

In this section we prove the following theorem.

\begin{theorem}
  \label{thm:compl}
  The Axiom of Constructibility implies the existence of a
  $\mathbf{\Pi^1_1}$ very mad family.
\end{theorem}

The proof is based on the proof of the analogous result for maximal
almost disjoint families of subsets of $\N$ by Arnold Miller, see
\cite{AM89}.  For background on constructibility see \cite[Chap.
VI]{KK80}, and \cite{KD} in combination with \cite{AMpre} (the theory
Basic Set Theory is not strong enough for the use Devlin makes of it,
this is analyzed in Mathias paper, and a replacement is offered there
that is sufficient for the results we use).
 
The idea of this proof is that we identify a set of good levels of $L$
(those for which $L_\alpha \cong \Sk(L_\alpha)$ with small witness, as
defined below).  We prove a coding lemma (Lemma~\ref{lem:coding})
allowing us to encode these levels into our construction.  Then we
show that from an encoding of a good level we have access to the limit
level after it (Lemma~\ref{lem:nextlevel}), which allows us to decide
membership (Lemma~\ref{lem:decidemembership}).

In this section we choose the sequence coding $\langle \ldots
\rangle$ and projections $\pi_i$ to be recursive.

\begin{lemma}
  \label{lem:coding}
  Let $A = \{ g_n \mid n \in \N\} \subseteq \BS$ be an almost disjoint
  family, $E \subseteq \mathbb{N} \times \mathbb{N}$ and $F = \{f_n
  \mid n \in \N \} \subseteq \BS$ not finitely covered by $A$.  Then
  there exists a function $g : \mathbb{N} \rightarrow \mathbb{N}$
  almost disjoint from all functions in $A$, such that $E$ is
  recursive in $g$ and $g$ agrees on infinitely many inputs with each
  member of $F$ ($\forall n \in \mathbb{N} \, \card{f_n \cap g} = \aleph_0$).
\end{lemma}

\begin{proof}
  Instead of encoding $E$ directly we encode $\chi$ the
  characteristic function of $\{ \langle n,m \rangle \mid (n,m) \in E
  \}$.

  We define $g$ recursively.  At step $s$ we extend the initial segment of
  $\mathbb{N}$ on which $g$ is defined by doing the following:
  \begin{enumerate}
  \item 
    Find $n_{s,i}$, $i \in [0,s]$, such that $n_s < n_{s,0} < n_{s,1}
    < \cdots < n_{s,s}$, where $n_s$ is the least number where $g$ is
    not defined yet, and $f_i(n_{s,i})$ is different from all
    $g_0(n_{s,i}), \ldots, g_s(n_{s,i})$.  Then define $g(n_{s,i}) =
    f_i(n_{s,i})$.  Also define $n_{s+1}$ to be $n_{s,s}+1$.
  \item
    Define $g(l)$ for $n_s < l < n_{s+1}$ where $g$ is not yet defined
    to be the least number different from all $g_0(l), \ldots,
    g_s(l)$.
  \item Define $g(n_s)$ to be $\langle k,\langle
    n_{s+1},\chi(s)\rangle \rangle$ where $k$ is the least number such
    that $\langle k, \langle n_{s+1}, \chi(s) \rangle \rangle$ is
    different from all $g_0(n_s), \ldots, g_s(n_s)$.  Here the value
    $n_{s+1}$ is the ``pointer'' to the next location where a value of
    $\chi$ can be found.
  \end{enumerate}

  It can now be easily checked that the $g$ constructed satisfies the
  lemma.
\end{proof}

We note that if $A$, $F$, $E$ are members of $L_\alpha$, then $g$ is a
member of $L_{\alpha + 1}$; the proof shows how to define $g$ from
$A$, $F$ and $\chi$; also $E$ and $\chi$ appear at the same level of
the constructible hierarchy.  Also note that the encoding is uniform:
it does not depend on which functions and families we work with.  This
also means that we can talk about \emph{the} relation encoded in $g$
(later this relation will be the inclusion relation of a model on
$(\N, E)$).


\begin{definition}
  For $\alpha > \omega$ we say $L_\alpha \cong \Sk(L_\alpha)$ iff
  there exists $\langle h, \varphi, \seq{p} \rangle$ (the witness)
  such that:
  \begin{enumerate}
  \item $h$ is a Skolem function for all $\Sigma_k$ formulas for
    $L_\alpha$, for some $k \geq 1$,
  \item $h[ \N \times (\N \cup \seq{p}) ] \cong L_\alpha$, and
  \item $h(n,x) = y \Leftrightarrow L_\alpha \models
    \varphi(\seq{p},n,x,y)$.
  \end{enumerate}
\end{definition}
\begin{lemma}
  \label{lem:unboundedLsk}
  The set $\{ \alpha \mid L_\alpha \cong \Sk(L_\alpha)\}$ is unbounded
  in $\omega_1$.
\end{lemma}

\begin{proof}
  First recall from G\"odel's proof of $\mathsf{CH}$ in $L$ that every
  constructible real is in $L_{\omega_1}$.  From this using the fact
  all $L_\beta$, $\beta < \omega_1$, are countable it follows that
  the set $\{\beta < \omega_1 \mid \exists r \ [\, r \in L_{\beta+1}
  \setminus L_\beta \wedge r \in \BS\,]\,\}$ is unbounded in
  $\omega_1$.  So it is sufficient to prove for each $\beta$ in this
  set that $L_{\beta+\omega} \cong \Sk(L_{\beta+\omega})$.
  
  Therefore let $r$ be definable over $L_\beta$ from a finite
  sequence of parameters $\bar{q}$, $r = \{ \langle m, n \rangle \mid
  L_\beta \models \psi( m, n, \seq{q})\}$, and such that $r \not \in
  L_\beta$.  Then $r \in L_{\beta + \omega}$ so that
  $L_{\beta+\omega} \models \exists r \ \forall m,n \in \omega \left(
  (m,n) \in r \leftrightarrow \psi^{L_\beta}(m,n,\seq{q})\right)$.

  Let $h : \N \times L_{\beta+\omega} \rightarrow L_{\beta+\omega} $
  be a definable Skolem function for $\Sigma_{k+2}$ formulas with $k
  \in \N$ such that $\psi \in \Sigma_k$.
  
  As $X = h[\N \times (\N \cup \seq{q} \cup \{L_\beta\})] \prec_{k+2}
  L_{\beta+\omega}$ we have $(X,\in) \models
  \psi^{L_\beta}(n,m,\seq{q})$ iff $(L_\beta,{\in}) \models
  \psi^{L_\beta}(n,m,\seq{q})$ and $(X,\in) \models \exists r \ 
  \forall m,n \in \omega \left((m,n) \in r \leftrightarrow
    \psi^{L_\beta}(m,n,\seq{q}\right)$, which shows $r$ is in
  $(X,\in)$.
  
  By the condensation lemma \cite[Theorem II.5.2]{KD} we have a $\pi$
  such that $\pi: (X,{\in}) \cong (L_\alpha,{\in})$, $\alpha \leq \beta +
  \omega$ and $\alpha$ is a limit ordinal; this $\pi$ is the identity
  on transitive sets, in particular on the natural numbers.  From this
  we get $(X, \in) \models \psi^{L_\beta}(n,m,\bar{q})$ iff
  $(L_\alpha, \in) \models \psi^{\pi L_\beta}(n,m,\pi{\seq{q}})$ and
  $(L_\alpha, \in) \models \exists r \ \forall m,n \in \omega
  \left((m,n) \in r \leftrightarrow
    \psi^{\pi L_\beta}(m,n,\pi\seq{q}\right)$, which shows that $r$ is in
  $L_\alpha$.  So since $\alpha \leq \beta+\omega$, $r \not \in
  L_\beta$, $r \in L_{\alpha}$, and $\alpha$ is a limit ordinal, we get
  $\alpha = \beta +\omega$.  This means $X \cong L_{\beta+\omega}$ as
  was to be shown.
\end{proof}

\begin{definition}
  $\langle h, \varphi, \seq{p}\rangle$ is a \emph{small witness} for
  $L_{\alpha} \cong \Sk(L_{\alpha})$ if it is a witness for $L_\alpha
  \cong \Sk(L_\alpha)$ and there exist $r$, $\psi$, and $\beta$ such
  that
  \begin{enumerate}
  \item $L_\beta$ is in $\bar{p}$;

  \item If $\bar{q}$ is $\bar{p}$ with $L_\beta$ removed, then $r =
    \{(m,n) \mid L_\beta \models \psi(m,n,\bar{q})\}$;
    
  \item $\alpha = \beta + \omega$;
    
  \item $h$ is a Skolem function for $\Sigma_{k+2}$ formulas for
    $L_\alpha$ and $k$ such that $\psi \in \Sigma_k$.
  \end{enumerate}
  
  That is $\langle h, \varphi, \seq{p} \rangle$ is related to $r$ and
  $\beta$ as in the proof of the previous lemma.
\end{definition}

The reason for this definition is that if we have a small witness, we
can check in $L_{\alpha + \omega}$ that $\langle h, \varphi,
\seq{p}\rangle$ is obtained from $r$ as in the proof of the lemma.
For a general witness we don't know at which level of $L$ the
isomorphism $h[ \N \times (\N \cup p) ] \cong L_{\alpha}$
will appear.

The proof of the lemma shows that the set $\{ \alpha < \omega_1 \mid
L_\alpha \cong \Sk(L_\alpha)$ with a small witness$\}$ is unbounded.
Enumerate it in increasing order by $\langle \beta_\gamma \mid \gamma
< \omega_1\rangle$\label{egl}.  Note that by absoluteness of the notion
of small witness and the fact that limit levels of the constructible
hierarchy are closed under certain simple recursions, we have that
$L_{\beta_\gamma + \omega} \models$``$ \langle \beta_{\gamma'} \mid
\gamma'\leq \gamma \rangle $ is an initial segment of the increasing
enumeration of ordinals $\alpha$ such that $L_\alpha \cong
\Sk(L_\alpha)$ with small witness''.

\begin{lemma}
  \label{lem:Eavailable}
  If $L_\alpha \cong \Sk(L_\alpha)$, then there is an $E \subseteq
  \N \times \N$ such that $E \in L_{\alpha + \omega}$ and
  $(L_\alpha, {\in}) \cong (\N, E)$.
\end{lemma}

\begin{proof}
  Let $L_\alpha \cong \Sk(L_\alpha)$ be witnessed by $\langle h,
  \varphi, \seq{p} \rangle$.  First notice that $\Th(\langle L_\alpha,
  {\in}, \seq{p}\rangle) \in L_{\alpha+\omega}$: we follow the ideas
  from pages 40 and 41 of \cite{KD}.  The theorem Devlin proves there is
  not correct, see \cite{AMpre}, but the method can be used here.  We
  have a function $f$ such that $f(0)$ is the set of all primitive
  formulas of set theory, and $f(i+1)$ is the set of all formulas
  formed from the formulas in $f(i)$ by conjunction, disjunction,
  implication, and quantification.  Then we construct a function $g$
  such that $g(i)$ is a set of pairs, first coordinate a formula $\varphi$ from
  $f(i)$, second coordinate a sequence $\bar{x}$ of elements of
  $L_\alpha$ such that $\varphi(\bar{x})$ is true in $(L_\alpha,
  {\in}, \bar{p})$.  All
  these elements are in $L_{\alpha + n}$ for some $n$.  Then at
  $L_{\alpha + n + 1}$ we can construct all $g \restrict k$ for $k \in
  \N$.  So at $L_{\alpha + n + 2}$ we can use the recursive definition
  of $g$ to construct it.  From $g$ we get $\Th(\langle L_\alpha,
  {\in}, \seq{p} \rangle)$ as the subset of the image consisting of
  all formulas with no free variables.  (Note that this, and the
  following, are all uniform with respect to the sequence $\bar{p}$,
  but, for notational convenience, we'll leave it implicit as a
  parameter.)

  Let $e: \N \rightarrow \N \times (\N \cup \bar{p})$ be 
  the definable bijection
  \begin{equation*}
    e(n) =
    \begin{cases}
      (\pi_0(n), \bar{p}_{\pi_1(n)}), \quad & \text{ if } \pi_1(n) <
      \lh(\seq{p});\\
      (\pi_0(n), \pi_1(n) - \lh(\seq{p})), \quad & \text{ otherwise},
    \end{cases}
  \end{equation*}
  and $\varphi_e$ the formula defining $e$, i.e. $\varphi_e(n,x,y)
  \Leftrightarrow e(n) = (x,y)$ (this formula defines $e$ in any
  $L_{\alpha+4}$ with $\alpha > \omega$ and $\bar{p} \in L_\alpha$ and is
  absolute for these levels).

  Define $\tilde{e}: \N \rightarrow \N$ from this by setting
  $\tilde{e}(0) = 0$ and $\tilde{e}(n+1) = k$ where $k$ is the least
  number bigger than $\tilde{e}(n)$ such that $\ulcorner
  \psi(k,\tilde{e}(n)) \urcorner \in \Th(\langle L_\alpha, \in,
  \seq{p}\rangle )$, where $\psi(k, \tilde{e}(n))$ is the formula
  \begin{align*}
    \forall l \leq \tilde{e}(n) \ \forall y_0, y_1 \ \big[ \
      \varphi(\seq{p}, \pi_0(e(l)), \pi_1(e(l)), y_0)\ & \wedge\
      \varphi(\seq{p}, \pi_0(e(k)), \pi_1(e(k)), y_1) \rightarrow \\
    \big(\exists z \ (z \in y_0 \wedge z \not \in y_1) \ & \vee \ (z \not
      \in y_0 \wedge z \in y_1)\big)\ \big],
  \end{align*}
  in which $\varphi$ is the formula defining $h$, and which after
  elimination of $e$ in favor of its definition becomes
  \begin{align*}
    \forall l \leq \tilde{e}(n) \ \forall l_0,l_1, k_0,k_1 \ \big\{ \ \varphi_e(l,l_0,l_1) \ &\wedge \
      \varphi_e(k,k_0,k_1) \rightarrow \\
    \forall y_0,y_1\, \big[ \ \varphi(\seq{p}, l_0, l_1,
      y_0)\ & \wedge\ \varphi(\seq{p}, k_0, k_1, y_1)
      \rightarrow \\
    \big(\,\exists z \ (z \in y_0 \wedge z \not \in y_1) \ & \vee \ (z \not
      \in y_0 \wedge z \in y_1)\,\big)\ \big] \ \big\}.
  \end{align*}

  Note $\psi(k,\tilde{e}(n))$ is the formula expressing $\forall l
  \leq \tilde{e}(n)\ h(e(k)) \neq h(e(l))$, and a G\"odel number for
  $\psi(k,\tilde{e}(n))$ can be recursively obtained from $k$ and
  $n$ (the function $(k,m) \mapsto \ulcorner \forall x \theta_m(x)
  \rightarrow \psi(k,x) \urcorner$ (where $\theta_m(x)$ is the formula
  defining $m \in \omega$) is in $L_{\omega+\omega}$, but
  $\tilde{e}$ which is recursively defined from it and $\Th(\langle
  L_\alpha, \in, \bar{p}\rangle)$ can be constructed at the level of
  $L$ after $\Th(\langle L_\alpha, \in, \bar{p} \rangle)$ is constructed).

  Let $\varphi_{\tilde{e}}(n,m)$ be such that
  $\varphi_{\tilde{e}}(n,m) \Leftrightarrow \tilde{e}(n) = m$.

  These definitions have been made so that $h \circ e \circ \tilde{e}
  : \N \rightarrow L_\alpha$ is an enumeration of $h[ \N
  \times (\N \cup \seq{p}) ]$ without repetitions.  We will set up
  the model $(\N, E)$ such that the number $m \in \N$ will
  represent the set $h(e(\tilde{e}(m)))$.  It is then clear that we
  want $n \mathrel{E} m$ iff $h(e(\tilde{e}(n))) \in
  h(e(\tilde{e}(m)))$.
  
  We show $E \in L_{\alpha+k}$ for some $k \in \N$ by eliminating all
  functions in favor of their definitions in the statement
  $h(e(\tilde{e}(n))) \in h(e(\tilde{e}(m)))$, and then noting this
  statement is true of $(n,m)$ iff the G\"odel number of the formula
  resulting from substituting terms defining $n$ and $m$ in this
  formula is in $\Th(\langle L_\alpha, E ,\seq{p} \rangle)$.

  First eliminating $h$, we get
  \begin{align*}
    \forall z_n,z_m\,\big[\,\big\{ &\varphi(\seq{p}, \pi_0(e(\tilde{e}(n))),
      \pi_1(e(\tilde{e}(n))), z_n) \ \wedge \\
    & \varphi(\seq{p}, \pi_0(e(\tilde{e}(m))), \pi_1(e(\tilde{e}(m))),
    z_m)\,\big\} \rightarrow z_n \in z_m\,\big].
  \end{align*}
  Then eliminating $e$ we get
  \begin{align*}
    \forall x_n,y_n, x_m,y_m \,\big\{ \ \varphi_e(\tilde{e}(n),x_n,y_n) \
      &\wedge \ \varphi_e(\tilde{e}(m), x_m,y_m) \rightarrow \\
    \forall z_n,z_m \,\big[ \ \varphi(\seq{p}, x_n,y_n,z_n) \ & \wedge \
      \varphi(\seq{p}, x_m,y_m,z_m) \rightarrow z_n \in z_m \,\big] \
      \big\}.
  \end{align*}
  After eliminating $\tilde{e}$ this gives
  \begin{align*}
    \forall l_n,l_m \, \big( \ \varphi_{\tilde{e}}(n,l_n) \ & \wedge \
      \varphi_{\tilde{e}}(m,l_m) \rightarrow \\
  \forall x_n,y_n, x_m,y_m \,\big\{ \ \varphi_e(l_n,x_n,y_n) \
    &\wedge \ \varphi_e(l_m, x_m,y_m) \rightarrow \\
  \forall z_n,z_m \,\big[ \ \varphi(\seq{p}, x_n,y_n,z_n) \ & \wedge \
    \varphi(\seq{p}, x_m,y_m,z_m) \rightarrow z_n \in z_m \, \big] \,
    \big\} \, \big).
  \end{align*}
  This is a formula in the language $\{{\in}, \seq{p}\}$ with free
  variables $n$ and $m$.  The recursive function $G$ that to $(n,m)$
  assigns the formula
  \begin{align*}
    \forall u, v \ \theta_n(u) \wedge \theta_m(v) \rightarrow \\
    \forall l_n,l_m \, \big( \ \varphi_{\tilde{e}}(u,l_n) \ & \wedge \
      \varphi_{\tilde{e}}(v,l_m) \rightarrow \\
  \forall x_n,y_n, x_m,y_m \,\big\{ \ \varphi_e(l_n,x_n,y_n) \
    &\wedge \ \varphi_e(l_m, x_m,y_m) \rightarrow \\
  \forall z_n,z_m \,\big[ \ \varphi(\seq{p}, x_n,y_n,z_n) \ & \wedge \
    \varphi(\seq{p}, x_m,y_m,z_m) \rightarrow z_n \in z_m \, \big] \,
    \big\} \, \big).
  \end{align*}
  is in $L_{\alpha+l}$ for some $l \in \N$ (note:
  $\varphi_{\tilde{e}}$ uses $\Th(\langle L_\alpha, \in,
  \bar{p}\rangle)$ as a parameter).
  
  This shows we can define $E$ over $L_{\alpha + l}$ by $(n,m) \in E$
  iff $G(n,m) \in \Th(\langle L_\alpha, {\in}, \seq{p}\rangle)$.
\end{proof}

We now define functions (as in \cite[page 217]{PGH}) relating the natural
numbers and the real numbers to their representatives in $(\N,E)$.

Define for any $(\N, E) \cong L_\alpha$, $\omega < \alpha < \omega_1$,
a recursive function $\Nat_E : \N \rightarrow \N$ by
\begin{align*}
  \Nat_E(0) &= \text{the unique } u \in \N \text{ such that }
    \forall l \in \N \ (\neg \, l \mathrel{E} u)\\
  \Nat_E(n+1) & = \text{the unique } u \in \N \text{ such that }
    \forall l \in \N [ (l \mathrel{E} u) \leftrightarrow\\
  & \qquad\big((l \mathrel{E} \Nat_E(n)) \vee (l = \Nat_E(n))\big)]
\end{align*}

Using this we can define $\Real_E: \BS \rightarrow \N$ a partial
function by
\begin{align*}
  \Real_E(r) = \text{the} & \text{ unique (if exists) } u \in \N
    \text{ such that } \\
  & \forall n,m [r(n) = m \leftrightarrow (\N, E) \models
    u(\Nat_E(n)) = \Nat_E(m)]
\end{align*}


If $L_\alpha \cong \Sk(L_\alpha)$, then there exists $\pi : (L_\alpha,
\in) \cong (\N, E)$.  So the sets $\mathbb{R}_\alpha = \BS \cap
L_\alpha$ and $\mathbb{R}_E := \{ n \in \N \mid (\N, E) \models n
\text{ is a real}\}$ are mapped to each other by the isomorphism.  We
have in fact that if $r \in \mathbb{R}_\alpha$ then $r(k) = l$ iff
$\pi(r)(\Nat_E(k)) = \Nat_E(l)$ is true in $(\N, E)$.  So we can
define in $L_{\alpha + \omega}$ an enumeration $e_\alpha : \N
\rightarrow R_\alpha$ of all reals in $L_\alpha$ as follows:

First let $e : \N \rightarrow \mathbb{R}_E$ be the bijection $e(0) =
\min\{\mathbb{R}_E\}$ and $e(n+1) = \min \{ m \in \mathbb{R}_E \mid m
> e(n) \}$.  Then $e_\alpha$ is $e$ composed with the map defined by
\begin{gather*}
  \{ (n,r) \in \mathbb{R}_E \times \mathbb{R}_\alpha \mid \forall k,l
    \in \N \ (\N, E) \models n(\Nat_E(k)) = \Nat_E(l)
    \leftrightarrow r(k) = l \} = \\
  \{ (n,r) \in \mathbb{R}_E \times \mathbb{R}_\alpha \mid \pi(r) = n
    \}.
\end{gather*}

Now we are ready for the construction of the very mad family $\A$
which we will show is coanalytic.  It will be recursively enumerated as
$\langle g_\alpha \mid \alpha < \omega_1 \rangle$.

To define $g_\gamma$ from $\langle g_\alpha \mid \alpha < \gamma
\rangle$ we use Lemma \ref{lem:coding} with $A = A_\gamma = \langle
g'_n \mid n \in \N \rangle$, $F = F_\gamma = \langle f_n \mid n
\in \N\rangle$ and $E$ as described below.

By induction we will have the set $\{ g_\alpha \mid \alpha < \gamma
\}$ in $L_{\beta_\gamma + \omega}$ ($\beta_\gamma$ as defined on page
\pageref{egl}), and by a recursion in $L_{\beta_\gamma + \omega}$ we
get the enumeration $\langle g_{\gamma'} \mid \gamma' < \gamma
\rangle$ in $L_{\beta_\gamma + \omega}$.  We can recursively find an
enumeration $\langle g'_n \mid n \in \N \rangle$ of it in
$L_{\beta_\gamma + \omega}$ by letting $g'_n$ be the $n^{\text{th}}$
member in the enumeration $e_{\beta_\gamma}$ of
$\mathbb{R}_{\beta_\gamma}$ which is in $\{g_\alpha \mid \alpha <
\gamma\}$.

We then recursively define $f_n$ to be the $n^{\text{th}}$ member in
the enumeration of $\mathbb{R}_{\beta_\gamma}$ which is not finitely
covered by $\{g_\alpha \mid \alpha < \gamma\}$.  This enumeration will
also be in $L_{\beta_\gamma + \omega}$.

By Lemma \ref{lem:Eavailable} we have an $E$ such that $(\N, E)
\cong (L_{\beta_\gamma},{\in})$ in $L_{\beta_\gamma + \omega}$.

After application of Lemma \ref{lem:coding} (and the observation
following it) we get $g_\gamma \in L_{\beta_\gamma + \omega}$.  This
finishes the construction.  Note that this construction is absolute
for $L_{\beta_\gamma + \omega}$.

Clearly $\A$ is an a.d. family, and if $F \subseteq \BS$ with
$\card{F} < \card{A} = \aleph_1$, then there is a $\beta < \omega_1$
such that $F \subseteq L_\beta$.  Now if $F$ is not finitely covered
by $\A$ then for every $f \in F$ and every $\gamma$ with $\beta_\gamma
\geq \beta$ the set $f \cap g_\gamma$ is infinite, which shows that $\A$
is a very mad family.

Now what remains to be seen is that this $\A$ is $\Pi^1_1$ definable.
\begin{lemma}
  \label{lem:nextlevel}
  If $(\N, E) \cong (L_\alpha, \in)$ and $g \in L_{\alpha+\omega}$
  encodes $E$ as in Lemma \ref{lem:coding}, then there is a formula
  $\varphi$ only containing quantifiers over the natural numbers such
  that
  \begin{align*}
    \varphi(\langle E_\omega, r, u \rangle, g) & \Leftrightarrow
      (\N, E_\omega) \cong (L_{\alpha + \omega}, {\in}) \ \wedge \\
    & r \text{ is the satisfaction relation for } (\N, E_\omega)
      \ \wedge \\
    & u = \Real_{E_\omega}(g).
  \end{align*}
\end{lemma}

\begin{proof}
  In the definition below we refer directly to $E$; that this can be
  replaced by $g$ is easy.

  We define
  \begin{align*}
    \varphi(\langle E_\omega,r,u \rangle, g) & \equiv
      \mathsf{Sat}(E_\omega, r) \ \wedge \ \mathsf{EonEvens}(E_\omega,
      E) \ \wedge \\
    & \qquad \mathsf{Levels}(E_\omega, E,r) \ \wedge \
      \Real_{E_\omega}(g) = u,
  \end{align*}
  where:

  $\mathsf{Sat}$: The formula $\mathsf{Sat}(E_\omega, r)$ states that
  $r$ is the satisfaction relation for $E_\omega$: (sketch)
  \begin{align*}
    r( \langle \code{\varphi}, \seq{m} \rangle ) = 1 \Leftrightarrow &
      (\code{\varphi} = \code{ x = y} \wedge m_0 = m_1) \ \vee \\
    & (\code{\varphi} = \code{x \in y} \wedge E_\omega(m_0, m_1)) \ \vee \\
    & (\code{\varphi} = \code{\forall x \psi(x)} \wedge \forall n \
    r(\langle \code{\psi}, \langle n,\seq{m} \rangle\rangle) = 1) \
      \vee \\
    & (\code{\varphi} = \code{\neg \psi} \wedge r(\langle
      \code{\psi}, \seq{m}\rangle) = 0) \ \vee \\
    & (\code{\varphi} = \code{\psi_1 \vee \psi_2} \wedge ( r(\langle
      \code{\psi_1}, \seq{m}\rangle) = 1 \vee r(\langle \code{\psi_2},
      \seq{m}\rangle) = 1))
  \end{align*}
  
  $\mathsf{EonEvens}$: $\mathsf{EonEvens}(E_\omega, E)$ states that
  $E$ is isomorphic to an initial segment of $E_\omega$ and lives on
  the even natural numbers.
  \begin{align*}
    \mathsf{EonEvens}(E_\omega, E) & \equiv \forall i,j \,\big( \neg (2i+1
      \mathrel{E_\omega} 2j) \ \wedge \ (2i \mathrel{E_\omega} 2j
      \leftrightarrow i \mathrel{E} j)\,\big)
  \end{align*}

  $\mathsf{Levels}$: Here we need a bijection $\pi : \N \times
  \N \rightarrow \N$ such that $\pi(0,0) = 1$ and $\pi(0,k+1)$
  enumerates the evens; we can easily find such a bijection which is
  recursive. 

  Then $\mathsf{Levels}(E_\omega, E, r)$ is the conjunction of
  $\mathsf{SLevels}(E_\omega, E)$ and $\mathsf{ELevels}(E_\omega, E,
  r)$ where $\mathsf{SLevels}$ states $\pi(l,0)$ is the $l$-th level
  after $(\N, E)$:
  \begin{align*}
    \forall l,i,j \,&\big( \,\big[ \, i < l \rightarrow \pi(i,j)
      \mathrel{E_\omega} \pi(l,0) \, \big] \ \wedge \ \pi(l,j + 1)
      \mathrel{E_\omega}
      \pi(l,0) \,\big) \wedge\\
    \forall l,i,j \, &\big( \pi(i,j) \mathrel{E_\omega} \pi(l,0)
      \rightarrow ( i < l \vee (i = l \wedge j > 1) ) \big),
  \end{align*}
  and $\mathsf{ELevels}(E_\omega, E, r)$ that $k \mapsto \pi(l,k +
  1)$ is an enumeration of the new sets at the level $l^{\text{th}}$
  after $(\N, E)$.  First we find an enumeration, $k \mapsto
  \mathsf{ge}(l,k)$, of formulas and parameters that can be used to
  define sets at the $l^{\text{th}}$ level:

  Let $S$ be the set $\{ (n, \seq{x} ) \mid n$ is the G\"odel number
  of a formula with $\lh(\seq{x}) + 1$ free variables $\wedge \ \seq{x}
  \in {}^{<\N}\N \}$.  Then define $\mathsf{ge} : \N \times
  \N \rightarrow S$ such that $\mathsf{ge}[\{ (l,k) \mid k \in
  \N \}] = \{ (n, \seq{x}) \in S \mid \seq{x} \in {}^{<\N}(\{ \pi(l, k+1)
  \mid k \in \N\} \cup \{ \pi(j,k) \mid j < l \, \wedge \, k \in
  \N\}) \}$.  Such a function $\mathsf{ge}$ can clearly
  be chosen to be recursive.

  We want to define $\tilde{\mathsf{ge}} : \N \times \N
  \rightarrow S$ such that $k \mapsto \tilde{\mathsf{ge}}(l, k)$
  enumerates only the data needed to define new sets at level $l$, and
  does so without repetition.  For this we do some preliminary work.
  
  First note that $(\pi(l,0), E_\omega) \models \varphi(x)$ is
  equivalent to $(\N, E_\omega) \models (\varphi(x))^{\pi(l,0)}$ which
  in turn is equivalent to $r(\langle \ulcorner (\varphi)^{\pi(l,0)}
  \urcorner, x \rangle = 1$.  The map $\mathsf{rel}: \N \times \N
  \rightarrow \N$ defined by $(\ulcorner \varphi \urcorner, l) \mapsto
  \ulcorner (\varphi)^{\pi(l,0)} \urcorner$ and $0$ if the first
  component of the input is not the G\"odel number of a formula is
  recursive.

  Then we define a formula $\mathsf{new}(n,\seq{x},l)$ such that it
  is true of $(n, \seq{x}, l)$ iff $n = \ulcorner \varphi \urcorner$
  and $\{ y \mathrel{E_\omega} \pi(l,0) \mid (\pi(l,0), E_\omega)
  \models \varphi(\seq{x},y) \}$ is different from all $\pi(j,k)$ for
  $j \leq l$ and $k \in \N$.  This means that
  the set determined by $(n,\seq{x})$ at level $l$ didn't exist before
  level $l$ (and is not the collection of all sets before level $l$,
  which is $\pi(l,0)$).  The formula expressing this is:
  \begin{align*}
    \mathsf{new}(n,\seq{x},l)  \equiv \forall j,k & \ j \leq l
      \rightarrow \exists j', k' \big\{ j' \leq l \ \wedge \\
    & \big[ \, \big(\pi(j',k') \mathrel{E_\omega} \pi(j,k) \wedge
      r(\langle \mathsf{rel}(n,l), \langle x,\pi(j',k')\rangle
      \rangle) = 0 \,\big) \ \vee \\
    & \phantom{\big[ \,}\big( \neg \pi(j',k') \mathrel{E_\omega} \pi(j,k) \wedge
      r(\langle \mathsf{rel}(n,l), \langle x,\pi(j',k')\rangle \rangle
      ) = 1 \, \big) \, \big] \, \big\}
  \end{align*}

  We also need a formula $\mathsf{nb}(l,m)$ that is true of $(l,m)$
  iff the set defined by $\mathsf{ge}(l,m)$ from $\pi(l,0)$ is not
  also defined by $\mathsf{ge}(l,m')$ with $m' < m$.
  \begin{align*}
    \mathsf{nb}(l,m) \equiv & \ \forall m' \leq m \exists j,k ( j < l
      \vee (j = l \wedge k > 0)) \ \wedge \\
    & \quad \big[ \ \big( \, r( \langle \mathsf{rel}(n,l), \langle x,
      \pi(j,k) \rangle \rangle ) = 1 \ \wedge \\
    & \quad \phantom{\big[ \ \big( \, } r( \langle
      \mathsf{rel}(\pi_0({\mathsf{ge}}(l,m')),l), \langle
      \pi_1({\mathsf{ge}}(l,m')), \pi(j,k) \rangle \rangle) = 0 \,
      \big) \ \vee \\
    & \quad \phantom{\big[ \ } \big( \, r( \langle \mathsf{rel}(n,l),
      \langle x, \pi(j,k) \rangle \rangle ) = 0 \ \wedge \\
    & \quad \phantom{\big[ \ \big( \,} r( \langle
      \mathsf{rel}(\pi_0({\mathsf{ge}}(l,m')),l), \langle
      \pi_1({\mathsf{ge}}(l,m')), \pi(j,k) \rangle \rangle) = 1\big) \
      \big]
  \end{align*}

  Now we can define $\tilde{\mathsf{ge}}$:
  \begin{align*}
    \tilde{\mathsf{ge}}(l,0) & = \mathsf{ge}(l,k) \text{ for } k \text{
      the least number such that }\\
    & \quad \quad \quad (n,x) = \mathsf{ge}(l,k) \text{
      defines a new set} \\
    & = \mathsf{ge}(l,k) \text{ for } k \text{ the least number such
      that for } \\
    & \quad \quad \quad (n,x) = \mathsf{ge}(l,k) \text{ we have }
      \mathsf{new}(n,x,l)
  \end{align*}
  and
  \begin{align*}
    \tilde{\mathsf{ge}}(l,m+1) & = \mathsf{ge}(l,k) \text{ for } k
      \text{ the least number such that } (n,x) = \mathsf{ge}(l,k) \\
    & \quad  \text{ defines a new set}
     \text{ that is not already defined by }
      \mathsf{ge}(l,\tilde{k}) \\
      & \quad \text{ with } \tilde{k} \text{ less
      than the } k \text{ used in } \tilde{\mathsf{ge}}(l,m) \\
    & = \mathsf{ge}(l,k) \text{ for } k \text{ the least number such
      that for } (n,x) = \mathsf{ge}(l,k)\\
    & \quad  \text{ we have } \mathsf{new}(n,x,l) \wedge \mathsf{nb}(l,k)
  \end{align*}

  Now the formula $\mathsf{ELevels}$ can be defined:

  \begin{align*}
    \forall l,k [ & \pi(l+1,k+1) \text{ is defined from } \pi(l,0)\\
    & \quad \quad \quad 
      \text{ by the formula and parameters in }
      \tilde{\mathsf{ge}}(l,k)] \\
    & \Leftrightarrow \forall l,k, n,x ( (n,x) = \tilde{\mathsf{ge}}(l,k)
      \rightarrow \\
    & \quad \quad \quad [\forall y r(\langle \mathsf{rel}(n,l), \langle x, y
      \rangle \rangle) = 1 \leftrightarrow y \mathrel{E_\omega} \pi(l,k)]
  \end{align*}
\end{proof}

Note that with these formulas, if $(\N, E)$ is wellfounded, then
so is $(\N, E_\omega)$ (which is the main reason for the lemma to be
done the way it is).

Let $\xi_\text{s} \in \Sigma_1$ and $\xi_\text{p} \in \Pi_1$ be the
formulas witnessing that the class $H = \{(x,\gamma) \mid x = L_\gamma
\}$ is uniformly $\Delta_1^{L_\alpha}$ for $\alpha > \omega$ a limit
ordinal (see \cite[Lemma 2.7]{KD}: the proof of this lemma uses some
results from earlier in the book which are not correct, but in
\cite{AMpre} (key result on p. 45) it is shown that there is a theory
which is strong enough to prove these results and which is true at
$L_\alpha$ for $\alpha$ a limit ordinal).

Let $E \subseteq \N \times \N$ be such that $(\N, E)$ is
wellfounded, and let $r$ be its satisfaction relation.  Then let
$\chi(E,r)$ be the formula (we write $(\N,E) \models \theta$ for
$r(\ulcorner \theta \urcorner)$)
\begin{gather*}
  \forall n,m \in \N \ \big[ \, (\N, E) \models \text{``}n
    \text{ is an ordinal'' } \rightarrow \quad \quad \quad \quad \quad
    \quad \mbox{}\\
  \mbox{} \quad \quad \quad\quad \quad \quad\quad \quad \quad (\N, E) \models
    \xi_\text{p}(m,n) \leftrightarrow (\N,E) \models
    \xi_\text{s}(m,n) \, \big] \wedge\\
  (\N, E) \models \text{ ``there is no largest ordinal''}
    \ \wedge \\
  \exists n \in \N \ (\N, E) \models ``n = \omega\text{''} \ \wedge \\
  (\N, E) \models \forall x \exists y \, \big( \, y \text{ is an
    ordinal } \wedge \forall z \, ( \xi_\text{p}(z,y) \rightarrow x
    \in z ) \, \big).
\end{gather*}
Then the image $X$ of the Mostowski collapse of $(\N, E)$
satisfies that $H$ is $\Delta_1^X$, there is no largest ordinal, $\omega
\in X$, and $V=L$.
This gives us that $X = (V)^X = (L)^X = L_\alpha$ for $\alpha = X \cap
\Ord$ a limit ordinal $> \omega$.

\begin{lemma}
  \label{lem:decidemembership}
  \label{lem:vmad:decidemembership}
\begin{align*}
  g \in \mathcal{A} \Leftrightarrow & \text{ the model encoded in } g
  \text{ is wellfounded } \wedge \\
  &  \forall \langle E_\omega, r, u \rangle \
  \varphi(\langle E_\omega, r, u\rangle, g) \wedge \chi(E_\omega,r)
  \rightarrow r( \code{u \in \A }, \seq{\emptyset}) = 1.
\end{align*}
\end{lemma}

\begin{proof}
  By induction on $\gamma < \omega_1$ we show that for all reals in
  $L_{\beta_\gamma}$ the equivalence holds.  So assume that $g \in
  L_{\beta_\gamma}$ and for all $\gamma' < \gamma$ we have the
  equivalence for all reals in $L_{\beta_{\gamma'}}$.

  If $g \in \A$, then $g$ uniformly encodes $(\N, E)$ such that $(\N,
  E) \cong (L_{\beta_{\gamma'}}, {\in})$ with $\gamma' < \gamma$.  The
  unique model $(\N, E_\omega)$ satisfying $\varphi(\langle E_\omega,
  r, u\rangle, g)$ has $(\N, E_\omega) \cong (L_{\beta_{\gamma'} +
  \omega},\in)$, so also satisfies $\chi$.  And in the description of
  the construction we have shown that $(L_{\beta_\gamma' + \omega},
  \in) \models g \in \A$, i.e. $(\N, E_\omega) \models \code{u \in
  \A}$ where $u$ represents $g$ in the model.

  If the model encoded by $g$ is wellfounded and we have $\forall
  \langle E_\omega, r, u \rangle \ \varphi( \langle E_\omega, r, u
  \rangle, g) \wedge \chi(E_\omega, r) \rightarrow r( \code{u \in \A},
  \seq{\emptyset}) = 1$, then the unique $\langle E_\omega, r, u
  \rangle$ for which $\varphi(\langle E_\omega, r, u \rangle, g)$
  has that $(\N, E_\omega)$ is wellfounded and satisfies
  $\chi(E_\omega, r)$.
  So there is a countable limit $\beta > \omega$ such that $(\N,
  E_\omega) \cong (L_\beta, \in)$.  Since $(\N, E_\omega) \models u
  \in \A$, we have $(L_\beta, \in) \models g \in \A$, which by
  absoluteness gives $g \in \A$.
\end{proof}

Since the formula on the right hand side of the equivalence is clearly
$\Pi^1_1$, this completes the proof of the theorem.

\chapter{Cofinitary Groups}
\label{chap3}

In
this chapter we prove our results on (maximal) cofinitary groups.  We
will again repeat the definitions for convenience, and give a short
list of notational conventions used.  After this, in Section
\ref{sec:MCGintro}, we give some of the basic ideas used in almost all
constructions.

\begin{definitions}
\item
  $\Sym(\N)$ is the group of bijections of the natural numbers with
  group operation composition.

\item
  A bijection $g \in \Sym(\N)$ is \emph{cofinitary} if it has finitely many
  fixed points, or is the identity.

\item
  A subgroup $G \leq \Sym(\N)$ is \emph{cofinitary} if all of its
  members are cofinitary.

\item
  A subgroup $G \leq \Sym(\N)$ is \emph{maximal cofinitary} if it is
  a cofinitary group and is not properly contained in another
  cofinitary group.
\end{definitions}
 
Now we will give our guide to notation in this chapter.

$G$ and $H$ will be (maximal) cofinitary groups.  $g$ with sub- and
superscripts will be elements of this group, elements under
consideration to be added to the group, or finite approximations
of elements thereof.
\section{Basics}
\label{sec:MCGintro}

In this section we will explain the construction of a maximal
cofinitary group from the continuum hypothesis.  That there exists a
maximal cofinitary group follows immediately from the wellorder of
$\BS$ that exists under $\CH$, but here we use it to introduce the
basic ideas for constructing a maximal cofinitary group.  This
construction is from earlier work by Yi Zhang, see \cite{Z4}, and gives
in a convenient setting the basic ideas used frequently later on.

We will construct the group by constructing a sequence of generators
$\langle g_\alpha \mid \alpha < \omega_1 \rangle$, such that $\langle
\{g_\alpha \mid \alpha < \omega_1 \} \rangle$ (the group generated by
the set $\{g_\alpha \mid \alpha < \omega_1\}$) is a maximal cofinitary
group.  This sequence will be constructed recursively.

The important step is adding to a given countable cofinitary group $G$
a new generator $g$ such that $\langle G,g \rangle$ is still
cofinitary and iteration of this construction $\omega_1$ many times
gives a maximal cofinitary group.

We first examine how to ensure we get a cofinitary group after adding
a new generator.

\begin{definitions}
\label{definitions WG etc}
\item
  For $G$ and $H$ two groups, we write $G * H$ for the free product of
  $G$ and $H$.

\item
  \label{definition WG}
  Define for a group $G \leq \Sym(\N)$ the group $W_G$ to be $G * F(x)$
  with $F(x)$ the free group on the generator $x$.
\end{definitions}

If $w(x) \in W_G$ then $w(x)$ has the reduced form
\begin{equation}
  \label{reduced form}
  \tag{**}
  w(x) = g_0 x^{k_0} g_1 x^{k_1} \cdots x^{k_{l-1}} g_{l},
\end{equation}
with $g_i \in G$ ($i \leq l$) and $k_i \in \mathbb{Z} \setminus
\{0\}$ ($i \leq l - 1$).

\begin{definition}
  For $w(x)$ as in \eqref{reduced form}, its length $\lh(w)$ is
  defined to be $l+1 + \Sigma_{i=0}^{l-1} k_i$.
\end{definition}

  The set $W_G$ is useful in this context because
of the following lemma.

\begin{lemma}
  If $h \in \langle G,g \rangle$, then there is a $w(x) \in W_G$ such
  that $w(g) = h$.
\end{lemma}

So for the group $\langle G,g \rangle$ to be cofinitary it suffices to
ensure that each $w(g)$ has finitely many fixed points or is the
identity.

The bijection $g$ will be constructed recursively from finite
approximations, i.e. $g = \bigcup_{s \in \N} g_s$ with $g_s :
\N \partmap \N$ finite and injective.  For this our strategy will be
to avoid as many fixed points as possible.

\begin{definition}
  $z \in W_G$ is a conjugate subword of $w \in W_G$ if there exists a
  $u$ such that without cancellation $w=u^{-1} z u$.
\end{definition}

If for a conjugate subword $z$ of $w$ the partial permutation $z(g_s)$
has a fixed point, then for any extension to a total permutation $g$
this fixed point will give rise to a fixed point of $w$.  So the best
we can hope for extending finite approximations is encoded in the
following definition.

\begin{definition}
  \label{definition good extension}
  For $w \in W_G$ and finite one-to-one functions $p,q$ such that
  $p \subseteq q$, we say that $q$ is a \emph{good extension} of $p$
  with respect to $w$ if the following condition is satisfied:

  \noindent for each $l \in \N$ such that
  \begin{equation*}
    w(p)(l) \text{ is undefined and } w(q)(l)=l,
  \end{equation*}
  there are subwords $u$ and $z$ of $w$ and $n \in \N$ such that
  \begin{gather*}
    w = u^{-1} z u \text{ without cancellation},\\
    u^{-1}(q)(n) = l \text{, and } z(p)(n)=n.
  \end{gather*}
\end{definition}

The following lemma shows that this definition works.

\begin{lemma}
  \label{lem:finfxpts}
  If $G$ is any countable cofinitary group, $w(x) \in W_G$, and $g =
  \bigcup_{s \in \N} g_s$ where all $g_s$ are finite and for some $t$
  and all $s \geq t$, $g_{s+1}$ is a good extension of $g_s$ with
  respect to
  $w$ and all of its subwords, then $w(g)$ will have finitely many
  fixed points.
\end{lemma}

We prove the following more quantitative version.

\begin{lemma}
  \label{lem:enoughforcof}
  In the context of the lemma above $w(g)$ will have the same number
  of fixed points as $z(g_t)$, where $z$ is the shortest conjugate
  subword of $w$.
\end{lemma}

We use the following definitions related to a word $w \in W_G$; these
are also used frequently in later sections.

\begin{definitions}
\item
  Define $w_{(i)}$ to be the $i^\text{th}$ letter in $w$ counted from the
  right (for example if $w = g_0 x^2 g_1$, then $w_{(0)} = g_1$, $w_{(1)} =
  x$, $w_{(2)} = x$, and $w_{(4)} = w_{(\lh(w))} = g_0$).

\item For $p : \N \rightharpoonup \N$ a partial function, $w(x) \in
  W_G$ and $n \in \N$, we define the \emph{evaluation path} for $n$ in
  $w(p)$ to be the sequence $\langle l_i \in \N \mid i \leq j
  \rangle$, with $l_0 := n$, $l_{i+1} := w_i(p)(l_i)$ and
  \[
    j :
= \begin{cases} \lh(w), &\text{if }w(p)(l) \text{ is defined};\\
      \max\{ i \mid w_{i}(p)(l_i)\text{ is  defined}\} + 1,
      &\text{otherwise}.
      \end{cases}
  \]
\item
  Define $w \restrict i$ to be $w_{(i-1)} w_{(i-2)} \cdots w_{(0)}$, the
  initial segment (from the right) of length $i$ of the word $w$.  The
  evaluation path for $n$ can be expressed then as $l_i = (w \restrict
  i)(p)(n)$
\item
  The pairs $(l_i, l_{i+1})$ of $p$ are the pairs of $p$ \emph{used}
  in this evaluation.  For a general function $f$ (possibly partial)
  we call $(n,f(n))$ a pair from $f$.
\end{definitions}

\begin{proof}[Proof of Lemma \ref{lem:enoughforcof}]
  By case analysis:

  Case 1: $w$ has only one conjugate subword (itself).  Then $w(g)$
  will have as many subwords as $w(g_t)$.  Since there are no
  proper conjugate subwords, in the definition of good extension every
  time $z$ will equal $w$.  It is then clear no new fixed points can
  be introduced.

  Case 2: $w$ has more than one conjugate subword.  Let $z$ be the
  minimal conjugate subword.  Notice that in any expression
  $w=u^{-1}z' u$ we have that $z$ is a conjugate subword of $z'$, so
  we can divide further.  A fixed point of $z'$ at some stage $s$
  ``comes from'' a fixed point of $z$ defined at that same stage: let
  $z'(g_s)(n) = n$, and $z' = v^{-1} z v$.  Then $z = v z' v^{-1}$
  and $(v z' v^{-1} )(g_s)(v(g_s)(n)) = (vz(g_s))(n) = v(g_s)(n)$.  And
  by the previous case, the total number of fixed points of $z(g)$ is
  equal to the total number of fixed points of $z(g_t)$.
\end{proof}

The nice and amazing thing is that enough good extensions exist to
achieve what we need.  This was shown with the following lemma(s) by Yi
Zhang, see \cite{Z} and \cite{Z4}.

\begin{lemma}
  
  If $G$ is a countable cofinitary group, $p$ a finite injective
  function $\N \partmap \N$, and $w \in W_G$ then
{
    \renewcommand{\theenumi}{(\roman{enumi})}
\begin{enumerate}
\item \textbf{(Domain Extension Lemma)} \label{lem:DomExtLem} for all
  $n \in \N \setminus \dom(p)$, for all but finitely many $k \in \N$,
  the extension $p \cup \{(n,k)\}$ is a good extension of $p$ with
  respect to
  $w$.
\item \textbf{(Range Extension Lemma)} \label{lem:RanExtLem} for all
  $k \in \N \setminus \ran(p)$, for all but finitely many $n \in \N$ the extension $p
  \cup \{(n,k)\}$ is a good extension of $p$ with respect to $w$.
\item \textbf{(Hitting $f$ Lemma)} \label{lem:HitfLem} for all $f \in
  \Sym(\N) \setminus G$ such that $\langle G,f \rangle$ is cofinitary,
  for all but finitely many $n \in \N$ the extension $p \cup
  \{(n,f(n))\}$ is a good extension of $p$ with respect to $w$.
\end{enumerate}
}
\end{lemma}

Since in each of the items in this lemma there are only finitely many
choices for a particular word to get an extension that is not good,
these items are also true for finite lists of words.  Also, as is
clear from the proof, the hitting $f$ lemma also works for $f$ that
are infinite partial functions.

With these preparations we can prove that the continuum hypothesis
implies the existence of a maximal cofinitary group.  With the
continuum hypothesis enumerate $\Sym(\N)$ as $\langle f_\alpha \mid
\alpha < \omega_1 \rangle$.

We construct a generating set of this group recursively as $\langle
g_\alpha \mid \alpha < \omega_1 \rangle$.  At step $\beta$ we have
already constructed $\langle g_\alpha \mid \alpha < \beta \rangle$,
generating a countable cofinitary group $G_\beta = \langle \{g_\alpha
\mid \alpha < \beta \} \rangle$ such that for all $\alpha < \beta$
either $f_\alpha \in G_\beta$ or $\langle G_\beta,f_\alpha \rangle$ is not
cofinitary.

$g_\beta$ will be constructed from finite approximations $g_s$, i.e.
$g_\beta = \bigcup_{s \in \N} g_{\beta,s}$.  $W_{G_\beta}$ is a countable set;
enumerate it as $\langle w_n \mid n \in \N \rangle$.

At sub-stage $s \in \N$ we have already constructed $g_{\beta,s}$
($g_{\beta,0} := \emptyset$), and we will construct $g_{\beta,s+1}$
from $g_{\beta,s}$ in three steps:
\begin{itemize}
\item $g_{\beta,s}^1 = g_{\beta,s} \cup \{(n,k)\}$, where $n$ is the
  least number not in the domain of $g_{\beta,s}$ and $k$ is the least
  number for which $g_{\beta,s}^1$ is a good extension of
  $g_{\beta,s}$ with respect to $w_0, \ldots, w_s$ ($k$ exists by the
  domain extension lemma).
  
\item $g_{\beta,s}^2 = g_{\beta,s}^1 \cup \{(n,k)\}$, where $k$ is the
  least number not in the range of $g_{\beta,s}^1$ and $n$ is the
  least number for which $g_{\beta,s}^2$ is a good extension of
  $g_{\beta,s}^1$ with respect to $w_0, \ldots, w_s$ ($n$ exists by
  the range extension lemma).
  
\item $g_{\beta,s+1} = g_{\beta,s}^2$ if $\langle G_\beta, f_\beta
  \rangle$ is not cofinitary, or $f_\beta \in G_\beta$; otherwise
  $g_{\beta,s+1} = g_{\beta,s}^2 \cup \{(n,f(n))\}$, where $n$ is the
  least number such that $g_{\beta,s+1}$ is a good extension of
  $g_{\beta,s}^2$ with respect to $w_0, \ldots, w_s$ ($n$ exists by
  the hitting $f$ lemma).
\end{itemize}

$\langle G_\beta, g_\beta \rangle$ will be cofinitary: for any $h \in
\langle G_\beta,g_\beta \rangle$ there is a $w \in W_{G_\beta}$ such
that $h = w(g_\beta)$.  And for any $w \in W_{G_\beta}$ there is a
stage $t$ such that $w$ and all of its subwords are included in $w_0,
\ldots, w_t$, so Lemma \ref{lem:finfxpts} applies.

The group $G := \langle \{g_\alpha \mid \alpha < \omega_1 \} \rangle$
is maximal cofinitary: for any $f_\beta$ either $f_\beta \in
G_{\beta+1} \subseteq G$, or $\langle G_{\beta+1}, f_\beta \rangle$ is
not cofinitary (if $\langle G_\beta, f_\beta \rangle$ is cofinitary
and $f_\beta \not \in G_\beta$, we infinitely often add a pair from
$f$ to $g_\beta$ in the third step above), so also $\langle G, f_\beta
\rangle$ is not cofinitary.

Note that $G$ is a group that is freely generated by the $g_\alpha$
($\alpha < \omega_1$): the method of good extensions, without
additional work, always leads to free groups.  For any word $w$ as
soon as we take good extensions with respect to it and its subwords we do not
add any fixed points we are not forced to add.  And because at this
stage we have only a finite approximation to the new generator, there
are only finitely many of those.

\section{Orbits}

\subsection{No Maximal Cofinitary Group Has Countably Many Orbits}
\label{sec:NotCountablyManyOrbits}

This subsection is devoted to the proof of the following theorem.  Note
that we are only looking at orbits of the action of $G$ on $\N$
obtained from the inclusion $G \subseteq \Sym(\N)$.

\begin{theorem}
  \label{thm:MCGNotCountOrb}
  A maximal cofinitary group has only finitely many orbits.
\end{theorem}

Suppose $G$ is a cofinitary group with infinitely many orbits.  Fix an
enumeration without repetitions $\langle O_i \mid i \in \N \rangle$ of
all orbits of $G$.  From these data we define a function $h$ such that
$h \not \in G$ and $\langle G, h \rangle$ is cofinitary, so $G$ is not
maximal.  We will first define $h$, show some of its properties and
finally show how these properties can be used to show that $h$ is as
required.

We define $h : \N \rightarrow \N$ by a sequence of finite
approximations $h_s$, $s \in \N$.  Set $h_0 := \emptyset$, and suppose
$h_s$ has been defined.  Let $n := \min\big((\N \setminus \dom(h_s))
\cup (\N \setminus \ran(h_s))\big)$ and $m := \min O_j$, where $j$ is
the least number such that $O_j \cap \big(\dom(h_s) \cup
\ran(h_s)\big) = \emptyset$.  Then set $h_{s+1} := h_s \cup \{\langle
n,m \rangle \}$ if $n \not \in \dom(h_s)$ and $h_{s+1} := h_s \cup
\{\langle m,n \rangle\}$ otherwise.

Clearly $h \in \Sym(\N) \setminus G$; we only need to verify that for
all $w(x) \in W_G$ the function $w(h)$ has finitely many fixed points,
or is the identity.  It will in fact be the case that all $w(h)$
(except the identity word) have only finitely many fixed points;
showing this will take some work.

First note that for all $O_i$ and $O_j$ there is at most one pair
$\langle a,b \rangle \in h$ such that $a \in O_i$ and $b \in O_j$.
But in fact \emph{much} more is true.  This much more is described by
the following definition, which also describes the picture from which
this proof developed.

\begin{definition}
  The \emph{$G$-orbits tree of $h$} has vertex set $\{O_j \mid j \in
  \N\}$. It has an edge between $O_j$ and $O_i$ if there is an $n
  \in O_j$ such that $h(n) \in O_i$.
\end{definition}

We need to see that this in fact defines a tree.  Suppose not, then
there is a cycle $O_{n_0}, O_{n_1}, \ldots, O_{n_l} = O_{n_0}$ and for
all $0 \leq i < l$ vertex $O_{n_i}$ is connected to vertex
$O_{n_{i+1}}$.  This means that for every $0 \leq i < l$ there is a
pair $\langle a,b \rangle \in h$ such that $a \in O_{n_i}$ and $b \in
O_{n_{i+1}}$ or $a \in O_{n_{i+1}}$ and $b \in O_{n_i}$.  By the
observation above these pairs are unique.  Let $s \in \N$ be the least
$s$ such that all pairs $\langle a,b \rangle$ used in this cycle are
in $h_{s+1}$.  Since $s$ is least with this property the unique pair
$\langle a,b \rangle \in h_{s+1} \setminus h_s$ is used in the cycle.
Then $\langle a,b \rangle$ connects some $O_{n_j}$ with one of its
neighbors, $O_{n_{j-1}}$ or $O_{n_{j+1}}$.  But each of these is
already connected to its other neighbor, so $a \in O_k$ and $b \in
O_l$ such that $O_k \cap (\dom(h_s) \cup \ran(h_s)) \neq \emptyset$
and $O_l \cap (\dom(h_s) \cup \ran(h_s)) \neq \emptyset$.  This
however means that the pair $\langle a,b \rangle$ does not satisfy the
defining criterion for inclusion in $h_{s+1}$; so we have the
contradiction we were looking for.

The next definition gives us a way to talk about the process of
evaluating a word $w(h)$ on a number $n$.  The orbit path defined
here can be looked at as a walk on the vertices of the
$G$-orbit tree of $h$.

\begin{definition}
  For $m \in \N$, $w(x) = g_0 x^{k_0}g_1 \cdots x^{k_{l-1}} g_l \in
  W_G$ and $h \in \Sym(\N)$ we define the \emph{orbit path} of $n$ in
  $w(h)$ to be the sequence of orbits the evaluation passes through
  --- that is $\bar{l} = \langle l_i \mid 0 \leq i \leq
  \lh(w) \rangle$ where $l_i = j$ iff $z_i \in O_j$
  with $\bar{z}$ the evaluation path for $w$ on $n$.
\end{definition}

One of the essential features of the function $h$ we have defined is
that for any $n \in \N$ and $w(x) \in W_G$, the evaluation path for
$n$ and the orbit path of $n$ determine each other.  This equivalence
will be useful as in a word with infinitely many fixed points the
action on the orbits allows us to conclude that one of the $g_i$ in
$w(x)$ has infinitely many fixed points which will be the, at that
time, desired conclusion.

We are now ready to finish the proof of Theorem
\ref{thm:MCGNotCountOrb}.  Suppose, towards a contradiction, that
there is a $w \in W_G$ such that $w(h)$ has infinitely many fixed
points.  We will show that each fixed point of $w(h)$ gives rise to a
fixed point in some $g_i$ appearing in $w$.

Let $n$ be one of the fixed points of $w(h)$ and $\bar{l}$ its orbit
path.  Let $l_i \in \bar{l}$ be such that $O_{l_i}$ is the first
vertex realizing the maximal distance from $O_{l_0}$ in the $G$-orbit
tree of $h$.  The orbit $O_{l_{i-1}}$ preceding $O_{l_i}$ is closer to
$O_{l_0}$, so $w_{(i)} = x$ or $x^{-1}$ (application of any member of $G$
will not change the orbit we are in) and $(w \upharpoonright i)(n) \in
O_{l_{i-1}}$ and $ (w \upharpoonright i+1)(n) \in O_{l_i}$ with
$\langle (w \upharpoonright i)(n), (w \upharpoonright i + 1)(n)
\rangle \in h$ or $\langle (w \upharpoonright i+1)(n), (w
\upharpoonright i)(n) \rangle \in h$.

Assume the former; the other case is analogous.

Since the $G$-orbit tree of $h$ is a tree, $O_{l_{i-1}}$ and $O_{l_i}$
are connected by an edge and $O_{l_{i-1}}$ is strictly closer to
$O_{l_0}$ than $O_{l_i}$, all the other neighbors of $O_{l_i}$ are
strictly closer to $O_{l_0}$ than $O_{l_i}$.  This means that the
first vertex after $O_{l_i}$ different from $O_{l_i}$ has to be equal
to $O_{l_{i-1}}$.  But as $\langle (w \upharpoonright i)(n), (w
\upharpoonright i + 1)(n) \rangle$ is the only pair in $h$ allowing
direct passage between $O_{l_{i-1}}$ and $O_{l_i}$ this means we have
to apply $h^{-1}$ with input $(w \upharpoonright i+1)(n)$ to get back
to $O_{l_{i-1}}$.

We have the following situation in the orbit path of $n$:
\begin{align*}
  (w \upharpoonright i)(n) \in O_{l_{i-1}}& \stackrel{h}{\rightarrow}
  (w \upharpoonright i+1)(n) \in O_{l_i} \rightarrow \cdots \\
  & \rightarrow (w\upharpoonright i+i)(n) \in O_{l_i}
  \stackrel{h^{-1}}{\rightarrow} (w \upharpoonright i)(n) \in
  O_{l_{i-1}}
\end{align*}

Now, in between arriving at $O_{l_i}$ and leaving $O_{l_i}$ we obviously stay
in the same orbit.  This means that between arriving at $O_{l_i}$ and
leaving we can only apply members of $G$.  By the shape of $w$ we
apply exactly one member $g_j$ of $G$.  And by the work above this member
has to fix $(w \upharpoonright i+1)(n)$.

We now know that every fixed point of $w(h)$ gives rise to a fixed
point in some $g_i$ appearing in $w$.  There is a $j$ such that
infinitely many fixed points of $w(h)$ give rise to a fixed point in
that $g_j$.  No two such fixed points of $w(h)$ can be associated to
the same fixed point of $g_j$ as for different points the $j$th
members of their respective evaluation paths are never equal. (Note
that here we are considering $(g_j, j)$ the group element together
with an indication of where it occurs in the word.  It is possible
that one group element occurs more than once.)  This shows that this
$g_j$ appearing in $w$ has infinitely many fixed points, contradicting
that it is a member of the cofinitary group $G$.

\subsection{A Maximal Cofinitary Group With Finitely Many Infinite Orbits}
\label{sec:FinitelyManyOrbits}

In this subsection we prove the following theorem:

\begin{theorem}
  \label{thm:FinitelyInfOrbits}
  The continuum hypothesis implies that for every $n \in \N \setminus
  \{0\}$ and $m \in \N$ there exists a maximal cofinitary group with
  exactly $n$ infinite orbits and exactly $m$ finite orbits.
\end{theorem}

\begin{proof}
  Let $n \in \N \setminus \{0\}$ and $m \in \N$ be given.  Choose a
  partition $F_0 \mathrel{\dot{\cup}} F_1 \cdots F_{m-1}
  \mathrel{\dot{\cup}} O_0 \mathrel{\dot{\cup}} O_1 \cdots O_{n-1} =
  \N$ with $F_i$ finite and $O_j$ infinite; these will be the orbits.

  We have to ensure that the group $G$ we construct satisfies the
  following four conditions.
  \begin{enumerate}
  \item
    \label{enum:inforb:orbits}
    $G$ is transitive on all $F_i$ and $O_j$,
  \item
    \label{enum:inforb:respectingorbits}
    $G$ respects all $F_i$ and $O_j$,
  \item
    \label{enum:inforb:cofinitary}
    $G$ is cofinitary,
  \item
    \label{enum:inforb:maximal}
    $G$ is \emph{maximal} cofinitary.
  \end{enumerate}

We will construct
sequences of generators $\langle g^{F_i}_\alpha \mid \alpha < \omega_1
\rangle$ ($i < m$) and $\langle g^{O_i}_\alpha \mid \alpha < \omega_1
\rangle$ ($i < n$) such that $g^{F_i}_\alpha \in \Sym(F_i)$ and
$g^{O_i}_\alpha \in \Sym(O_i)$.  We then define $G$ to be the group
generated by $\langle g_\alpha \mid \alpha < \omega_1 \rangle$, where
$g_\alpha = (\bigcup_{i < m} g^{F_i}_\alpha) \cup (\bigcup_{i < n}
g^{O_i}_\alpha)$.

To ensure condition \ref{enum:inforb:orbits} we choose the first member in each
of the sequences to generate a transitive group.

Condition \ref{enum:inforb:respectingorbits} is ensured by the way the elements
of the group $G$ are obtained from the sequences of generators we construct.

To ensure condition \ref{enum:inforb:cofinitary} we use the method of good
extensions on the infinite orbits.  Since then the groups on the
infinite orbits are freely generated and cofinitary, the same will be
true for $G$.

To ensure condition \ref{enum:inforb:maximal} we pick an enumeration $\langle
f_\alpha \mid \alpha < \omega_1 \rangle$ of $\Sym(\N)$ and ensure in
step $\alpha$ of the construction (which will be recursive with
$\omega_1$ steps) that either $f_\alpha \in G$ or there is a $w \in
W_G$ such that $w(f_\alpha)$ has infinitely many fixed points but is
not the identity.

We now give the details of the construction.  We recursively
construct the sequences of generators on the infinite orbits, and
in the first step fix the sequences on the finite orbits.

For each $i < m$ choose a permutation of $F_i$ that generates a
transitive group.  Then set each $g^{F_i}_\alpha$ to be equal to that
permutation.  If we make sure that the permutations on the infinite
orbits are cofinitary and generate the group freely, this takes care
of the finite orbits (the group generated by the $g^{F_i}_\alpha$
clearly is transitive on $F_i$, and as the generators on the $O_i$
generate the group freely we don't have to consider relations between
different $g^{F_i}_\alpha$).  So from now on we will restrict our
attention to the infinite orbits.

For each $i < n$ choose a transitive permutation of $O_i$, and set
$g^{O_i}_0$ equal to that permutation.

Note that the group generated by $g_0$ is free, cofinitary, and
transitive when restricted to each of the sets $F_i$ ($i < m$) and
$O_i$ ($i < n$).

At step $\beta < \omega_1$ we have constructed the sequences $\langle
g^{O_i}_\alpha \mid \alpha < \beta \rangle$ ($i< n$).  And the
$g^{O_i}_\alpha$ freely generate a cofinitary subgroup of $\Sym(O_i)$.
Also for all $\alpha < \beta$ we have ensured that either $f_\alpha
\in \langle g_{\alpha} \mid \alpha < \beta \rangle$ or there is a
$w(x) \in W_{\langle \{ g_\alpha \mid \alpha < \beta \} \rangle}$ such that
$w(f_\alpha)$ has infinitely many fixed points, but is not the
identity.

There are two cases.

\textsc{Case 1}: $f_\beta \cap (O_0 \times O_0)$ is infinite.

Then on in $\Sym(O_0)$ we perform the construction of $g^{O_0}_\beta$
using the domain, range, and hitting $f$ lemma (which, as we remarked,
also works for infinite partial functions), so construct a bijection
such that $f \cap g^{O_0}$ is infinite, but not all of $g^{O_0}$.  On
the other infinite orbits, we only use domain and range extension to
construct a new generator.

\textsc{Case 2}: Otherwise there is an $i \neq 0$ such that $f_\beta \cap
(O_0 \times O_i)$ is infinite.

Let $R := \ran(f_\beta \cap (O_0 \times O_i)$.  Then $R$ is an
infinite subset of $O_i$.  We now use domain and range extension to
construct $g^{O_i}_\beta$ such that $g^{O_i}_\beta \cap (R \times R)$
is infinite.  This means that $f_\beta^{-1} g^{O_i}_\beta f_\beta$ is
an infinite partial function $O_0 \partmap O_0$.  And we can construct
$g^{O_0}_\beta$ by using the domain extension, range extension, and
hitting $f$ lemmas, so that $g^{O_0}_\beta \cap f_\beta^{-1}
g^{O_i}_\beta f_\beta$ is infinite, but not all of $g^{O_0}$.  This
shows there is a $w(x) \in W_{G_\beta}$ ($w(x) = g_\beta^{-1} x^{-1}
g_\beta x$) such that $w(f_\beta)$ has
infinitely many fixed points but is not the identity.

 On the other infinite orbits, we only use domain and range
extension to construct a new generator.

In both cases we get freely generated cofinitary groups, and in both
cases we take care of the function $f_\beta$.  Therefore we have
constructed a maximal cofinitary group.
\end{proof}


\section{Isomorphism Types}

\renewcommand{\poset}{\mathbb{P}}

\subsection{A Maximal Cofinitary Group Universal for Countable Groups}
\label{unifCountableGrp}
This subsection is dedicated to proving the following theorem.

\begin{theorem}
  The continuum hypothesis implies that there exists a maximal
  cofinitary group $G$ such that every countable group $H$ embeds into
  $G$.
\end{theorem}

\subsubsection*{Domain and Range Extension}

Here we generalize the domain and range extension lemmas.

\begin{definitions}
\item
  Let $G \leq \Sym(\N)$ and let $\langle x_n \mid n \in \N \rangle$ be
  a sequence of variables.  Then $W_{G,n}$ is $G * F(x_0, \ldots,
  x_n)$, the free product of $G$ and the free group on the generators
  $x_0, \ldots, x_n$; i.e. $W_{G,n}$ is the set of words of the form
  $w = w(x_0, \ldots, x_n) = g_0\bar{x}_0 g_1\bar{x}_1 \cdots
  \bar{x}_kg_{k+1}$, where $g_i \in G$ for $0 \leq i \leq k+1$, $g_i
  \neq \Id$ for $1 \leq i \leq k$, and $\bar{x}_i$ is a nonempty
  reduced word in $F(x_0, \ldots, x_n)$.

\item
  For $\bar{p} = p_0, \ldots, p_n$, $\bar{q} = q_0, \ldots, q_n$ with
  all $p_i : \N \rightharpoonup \N$ finite injective, all $q_i : \N
  \rightharpoonup \N$ finite injective and $q_i \supseteq p_i$, we
  call $\bar{q}$ a \emph{good} extension of $\bar{p}$ with respect to
  $w \in W_{G,n}$ if the following holds:

  For all $l \in \N$ such that
  \begin{equation*}
    w(\bar{q})(l) = l \text{ and } w(\bar{p})(l) \text{ is undefined}
  \end{equation*}
  there exists an $k \in \N$ and
  $z,u$ subwords of $w$ such that
  \begin{equation}
    \label{eqn:defgoodext}\tag{$\dagger$}
    \begin{split}
      w = u^{-1}zu & \text{ without cancellation},\\
      z(\bar{p})(k) = &\ k \text{ and } u(\bar{q})(l) = k.
    \end{split}
  \end{equation}

  For $n=0$ this reduces to the original definition of good extension.

\item
  If $G$ is countable, then $W_{G,n}$ is countable; enumerate it
  by $\langle w_k \mid k \in \N \rangle$.  Now we define a partial
  order $\poset_{G,n} = \langle P, \leq \rangle$.  $P$ is the set of
  length $n+1$ sequences of finite injective functions $\N \rightarrow
  \N$, and $\bar{q} \leq \bar{p}$ if $\bar{q}$ is a good extension of
  $\bar{p}$ with respect to all words $\{ w_k \mid k \leq \card{\bar{p}}\}$
  and their subwords.
\end{definitions}

This order depends on the enumeration $\langle w_k \mid k \in \N
\rangle$ of $W_{G,n}$, but in a way that will never matter to us.

We can now state and prove the versions of the domain and range
extension lemmas we need.

\begin{lemma}[Domain Extension Lemma]
  Let $G$ be a countable cofinitary group, $\bar{p} = p_0, \ldots,
  p_n$ finite injective functions, $i \leq n$, $a \in \N \setminus
  \dom(p_i)$ and $w \in W_{G,n}$ then for all but finitely many $b \in
  \N$ the sequence $p_0, \ldots, p_i \cup \{(a,b)\}, \ldots, p_n$ is a
  good extension of $\bar{p}$ with respect to $w$.
\end{lemma}

\begin{proof}
  Using the original domain and range extension lemmas we can in turn
  extend each $p_j, j \neq i$, to a permutation $\widetilde{p_j}$ so
  that $\langle G, \{\tilde{p}_j | j \leq n, j \neq i\} \rangle$ is a
  cofinitary group, and the $\tilde{p}_j, j \leq n, j \neq i$, are
  free generators.

  Then, by the original domain extension lemma, for all but finitely
  many $b \in \N$, $p_i \cup \{(a,b)\}$ is a good extension of $p_i$ with
  respect to the word $w' = w(\tilde{p}_0, \ldots, \tilde{p}_{i-1},
  x,$ $\tilde{p}_{i+1}, \ldots, \tilde{p}_n)$.

  We claim that for such $b$ the sequence $p_0, \ldots, p_i \cup
  \{(a,b)\}, \ldots, p_n$ is a good extension of $\bar{p}$ with
  respect to $w$, which we see as follows:
  For any fixed point $s \in \N$ of $w(p_0, \ldots, p_i \cup
  \{(a,b)\}, \dots, p_n)$ that is not a fixed point of $w(\bar{p})$ we
  have that $s$ is a fixed point of $w(\tilde{p}_0, \ldots, p_i \cup
  \{(a,b)\}, \dots, \tilde{p}_n)$ that is not a fixed point of
  $w(\tilde{p}_0, \ldots, p_i, \dots, \tilde{p}_n)$.  So as $p_i \cup
  \{(a,b)\}$ is a good extension of $p_i$ with respect to $w'$ we have
  subwords $u$, $z$ of $w'$ as in \eqref{eqn:defgoodext} (for words
  with only one variable).  As the $\tilde{p}_m$ were free generators
  by replacing in them the $\tilde{p}_l$ by $x_l$ we get $w =
  u^{-1}zu$ as in \eqref{eqn:defgoodext} (for words with n variables),
  showing this is a good extension.
\end{proof}

\begin{lemma}[Range Extension Lemma]
  Let $G$ be a countable cofinitary group, $\bar{p} = p_0, \ldots,
  p_n$ finite injective functions, $i \leq n$, $l \in \N \setminus
  \ran(p_i)$ and $w \in W_{G,n}$ then for all but finitely many $k \in
  \N$ the sequence $p_0, \ldots, p_i \cup \{(k,l)\}, \ldots, p_n$ is a
  good extension of $\bar{p}$ with respect to $w$.
\end{lemma}

\begin{proof}
  From the domain extension lemma by using $\bar{p}^{-1} =
  p_0^{-1}, \ldots, p_n^{-1}$ and $\tilde{w} = w^{-1}$.
\end{proof}

\begin{corollary}[Domain and Range Extension Lemma]
  Let $G$ be a countable cofinitary group.  Then the sets $D_{i,k} :=
  \{ \bar{q} \in P \mid k \in \dom(q_i)\}$ and $R_{i,k} := \{ \bar{q}
  \in P \mid k \in \ran(q_i)\}$, $i \leq n$ and $k \in \N$, are dense
  in $\poset_{G,n}$.
\end{corollary}

\label{same reasoning used}
From this corollary it follows that we can add $n$ permutations $\bar{f}$
at a time to a countable cofinitary group $G$ in such a way that $\langle G,
\bar{f} \rangle$ is cofinitary and the group generated by $\bar{f}$ is
free.

Enumerate $\{ D_{i,k} \mid i \leq n \text{ and } k \in \N \} \cup
\{R_{i,k} \mid i \leq n \text{ and } k \in \N \}$ by $\langle D_s \mid
s \in \N \rangle$.  Then let $\bar{p}_0 = \emptyset$, and recursively
choose $\bar{p}_{s+1} \leq \bar{p}_s$ such that $\bar{p}_{s+1} \in D_s$.

Clearly this gives that all $f_i := \bigcup_{s \in \N} (\bar{p}_s)_i$,
where $(\bar{p}_s)_i$ is the $i$-th component of $\bar{p}_s$, are
bijections.  It remains to verify that the group $\langle G, \bar{f}
\rangle$ is cofinitary.  For this it is sufficient to show that for
any $w \in W_{G,n}$ the permutation $w(\bar{f})$ is cofinitary.  In
fact it will only have finitely many fixed points.

Let $w \in W_{G,n}$; so $w = w_k$ for some $k \in \N$.
$\card{\bar{p}_k} \geq k$ so that for $s \geq k$, $\bar{p}_{s+1}$ is a
good extension of $\bar{p}_s$ with respect to $w_k$ and its subwords.
Suppose, towards a contradiction, that $w(\bar{f})$ has infinitely
many fixed points.  Let $w'$ be the shortest subword of $w$ such that
$w'(\bar{f})$ has infinitely many fixed points.

However, $w'$ being a subword of $w$, from stage $k$ on we only take
good extensions with respect to this word, and $w'(\bar{p}_k)$ has
only finitely many fixed points (as it is a finite permutation).  So
for every stage $s > k$, when $w'(\bar{p}_s)$ has a fixed point that
$w'(\bar{p}_{s-1})$ does not have, we find a subword $w''$ of $w'$ and
a fixed point for $w''(\bar{p}_{s-1})$ (since $\bar{p}_s$ is a good
extension).  Every time we find a different fixed point for some
subword $w''$ of $w'$.  Since there are only finitely many subwords
$w''$, one of them has infinitely many fixed points, contradicting the
assumption that $w'$ was the shortest subword with infinitely many
fixed points.


\subsubsection*{Finitely Generated Groups}

Here we'll show how to add an isomorphic copy of any
finitely generated group to any countable cofinitary group.

So let $H$ be a finitely generated group, with generators $h_0,
\ldots, h_n$.  Let $x_0, \ldots, x_n$ be variables, where obviously
the intention is that $x_i$ will represent $h_i$, and we write
$\bar{x}$ for the sequence of $x$'s.  Let $W_{H,\Id}$ be the set of
words $w(\bar{x}) \in F(\bar{x})$ such that $w(\bar{h}) := w(h_0,
\ldots, h_n) = \Id$, i.e. the set of words representing the identity
(which is a normal subgroup of $F(\bar{x})$).  It follows that $H$ is
isomorphic to $F(\bar{x}) / W_{H,\Id}$.

To make $H$ act cofinitarily we plan to construct 
permutations of $\N$ $f_0, \ldots, f_n$ such that
\begin{itemize}
  \item for all $w \in W_{H,\Id}$, $w(\bar{f}) = \Id$, and
  \item for all $w \in F(\bar{x}) \setminus W_{H, \Id}$, $w(\bar{f})$
    has finitely many fixed points.
\end{itemize}
If we do this we clearly get the following theorem:

\begin{theorem}
  Any finitely generated group has a cofinitary action.
\end{theorem}

But we want to start with $G$ a countable cofinitary group and add
$\bar{f}$ such that
\begin{itemize}
  \item the group generated by $\bar{f}$ is isomorphic to $H$,
  \item $\langle G, \bar{f} \rangle$ is cofinitary.
\end{itemize}

So let $G$ be a countable cofinitary group, and set $W_{G,H} := G *
(F(\bar{x})/W_{H,\Id})$.  This is a countable set, so we can enumerate
it by $\langle w_n \mid n \in \N \rangle$.

For sequences $\bar{p}$ and $\bar{q}$ of length $n+1$ of finite
injective functions $\N \rightarrow \N$,   we call $\bar{q}$ a
$(G,H)$-good extension of $\bar{p}$ if for all $i \leq n$ we have
$p_i \subseteq q_i$ and for all $i \leq \Sigma_j \card{p_j}$, all $l \in
\N$ and all subwords $w$ of $w_i$, if $w(\bar{q})(l) = l$ then there
are subwords $u_1, u_2 ,z$ of $w$, $m \in \N$ and $z' \in F(\bar{x})$
such that
\begin{itemize}
  \item $w = u_1^{-1} z u_2$ without cancellation,
  \item $u_1(\bar{h}) = u_2(\bar{h})$,
  \item $z(\bar{h}) = z'(\bar{h})$,
  \item $z'(\bar{p})(m) = m$, and
  \item $u_2(\bar{q})(l) = m$.
\end{itemize}
Note that for $w_1, w_2 \in W_{G,n}$, if $w_1(\bar{h}) = w_2(\bar{h})$
then these words represent the same group element of $W_{G,H}$.  Also
note that the notion of $(G,H)$-good extension depends on which
generators for $H$ we have chosen and which enumeration of $W_{G,H}$.
But this dependence will never matter to us.  Lastly, if $\bar{q}$ is
a good extension of $\bar{p}$ with respect to all words $w_i$, $i \leq
\card{\bar{p}}$, and all their subwords, then $\bar{q}$ is a
$(G,H)$-good extension of $\bar{p}$.

We define a poset $\poset_{G,H} = \langle P, {\leq} \rangle$
associated with the groups $G$ and $H$.  The elements of $P$ are
$\bar{p} = (p_0,\ldots, p_n)$, finite sequences of length $n+1$ of finite
injective partial maps $\N \rightarrow \N$, such that for all $w \in
W_{H,\Id}$ we have $w(\bar{p}) \cong \Id$ ($w(\bar{p})$ is the
identity where defined). $\bar{q} \leq \bar{p}$ if $\bar{q}$ is a
$(G,H)$-good extension of $\bar{p}$.

Note that if $\bar{q}$ is a good extension of $\bar{p}$ with respect
to all words $w_i$, $i \leq \card{\bar{p}}$, and their subwords, then
not necessarily $\bar{q} \in P$.

If $\bar{p} \in P$ we say $\bar{q}$, a finite sequence of length $n+1$
of injective partial maps $\N \rightarrow \N$, is obtained from
$\bar{p}$ by \emph{applying relations} if
\begin{equation*}
  (a,b) \in q_i
    \quad \Leftrightarrow \quad
      \exists w' [x_i w' \in W_{H, \Id} \wedge w'(\bar{p})(b) = a].
\end{equation*}

\begin{lemma}
  \label{lemma:applying relations}\label{lem:applying relations}
  If $\bar{q}$ is obtained from $\bar{p}$ by applying relations, then
{
\renewcommand{\theenumi}{(\roman{enumi})}
  \begin{enumerate}
  \item \label{applRel:extension} for all $i \leq n$, $p_i \subseteq
    q_i$;
  \item \label{applRelApplRel} applying relations to $\bar{q}$ doesn't
    add anything;
  \item \label{inPoset} $\bar{q} \in P$;
  \item \label{ExtensionIsGood} $\bar{q} \leq \bar{p}$
  \end{enumerate}
}
\end{lemma}

\begin{proof}
  \ref{applRel:extension} is immediate since $x_i x_i^{-1} \in
  W_{H,\Id}$.
  
  For \ref{applRelApplRel}: if $x_j w \in W_{H,\Id}$ and $w(\bar{q})(b)
  = a$, then there exists a $w'$ such that $(x_j w)^{-1} x_j w' \in
  W_{H,\Id}$ and $w'(\bar{p})(b)=a$.
  
  For \ref{inPoset}: if $w \in W_{H,\Id}$ is such that $w(\bar{q})
  \ncong \Id$, then there is a word $w'$ such that $w^{-1}w' \in
  W_{H,\Id}$ such that $w'(\bar{p}) \ncong \Id$.  Also $\bar{q}$ is
  finite since any map in this sequence only has pairs $(a,b)$ in it
  where $a,b \in \bigcup_{i \leq n} \left( \dom(p_i) \cup \ran(p_i)
  \right)$.
  
  For \ref{ExtensionIsGood}: if $w \in W_{G,H}$ is such
  that $w(\bar{q})(n)$ is defined, then there is a word $w' \in
  W_{G,H}$ such that $w'(\bar{h}) = w(\bar{h})$ and $w'(\bar{p})(n)$
  is defined and has the same value.
\end{proof}

Now we define for $i \leq n$ and $k \in \N$ the sets $D_{i,k} := \{
\bar{p} \in P \mid k \in \dom(p_i) \}$ and $R_{i,k} := \{ \bar{p} \in
P \mid k \in \ran(p_i) \}$.

\begin{lemma}
  \label{lem:fgdense}
  For all $i \leq n$ and $k \in \N$, the sets $D_{i,k}$ and $R_{i,k}$
  are dense.
\end{lemma}

\begin{proof}
  The proofs for $D_{i,k}$ and $R_{i,k}$ are basically the same, so
  we'll only show that $D_{i,k}$ is dense for some $i \leq n$ and $k
  \in \N$.  (Also that $R_{i,k}$ is dense follows from $D_{i,k}$ being
  dense for any group and set of generators; just consider the group
  generated by the inverses of the generators.)

  So let $i \leq n$, $k \in \N$ and $\bar{p} \in P$.  We need to find
  $\bar{r} \in P$ s.t. $\bar{r} \leq \bar{p}$ and $\bar{r} \in
  D_{i,k}$.
  
  First apply relations to $\bar{p}$ to get $\bar{q}$.  If $\bar{q}
  \in D_{i,k}$ we are done by \ref{inPoset} and \ref{ExtensionIsGood}
  of the previous lemma, so suppose $\bar{q}$ is not in $D_{i,k}$.
  
  Now by the domain extension lemma, for all but finitely many $l \in
  \N$, $(q_0, \ldots, q_i \cup \{(k,l)\}, \ldots, q_n)$ is a good
  extension of $\bar{q}$ with respect to all words $w_i$, $i \leq
  \card{\bar{p}}$, and their subwords.  Choose such $l$ such that $l >
  \max\{ \{k\} \cup \bigcup_{i \leq n} \left(\dom(q_i) \cup
    \ran(q_i)\right)\}$.
  
  Let $\bar{r} = (q_0, \ldots, q_i \cup \{(k,l)\}, \ldots, q_n)$ for
  this $l$.  If $\bar{r} \in P$ then clearly $\bar{r} \leq \bar{p}$
  since a good extension with respect to all words $\{ w_i \mid i \leq
  \card{\bar{p}}\}$ and their subwords is a $(G,H)$-good extension.
  So, towards a contradiction, suppose that there is a $w \in
  W_{H,\Id}$ such that $w(\bar{r}) \ncong \Id$ and let $w$ be the
  shortest such word.  Since $w(\bar{q}) \cong \Id$ (since $\bar{q}
  \in P$ by \ref{inPoset} of the previous lemma), this new computation
  $w(\bar{r})(a) \neq a$ uses the pair $(k,l)$.  Since $l > \max\{
  \bigcup_{i \leq n} \left(\dom(q_i) \cup \ran(q_i)\right)\}$ this
  pair has to be used at the beginning or the end; if it is used in a
  location in the middle then it needs to be used again immediately
  (in the opposite direction).  This would mean $w$ has a subword
  $x_i^{-1} x_i$ or $x_i x_i^{-1}$ contradicting its minimality.  If
  $(k,l)$ is used at the beginning and the end then $w = x_i w' x_i^{-1}$,
  $l = a$ and $w(\bar{r})(a) = a$ contrary to the assumption that $a$
  is not a fixed point.  So $(k,l)$ is used only once either at the
  beginning of the word or at the end of the word.  This however
  contradicts $\bar{q}$ having been obtained by applying relations, if
  there is a word $w$ such that $w = w' x_i$ and $w'(\bar{q})(k)$ is
  defined, then $(k, w'(\bar{q})(k)) \in q_i$ (for exact
  correspondence with the definition of applying relations consider
  $x_i (w')^{-1}$), showing $k$ was already in the domain of $q_i$
  which means $\bar{q} \in D_{i,k}$.
\end{proof}

From this lemma we get the following theorem by reasoning similar to
that at the end of the previous subsection (page
\pageref{same reasoning used}).

\begin{theorem}
  For $G$ a countable cofinitary group and $H$ a finitely generated
  group, there exist $\bar{f}$ such that the group generated by
  $\bar{f}$ is isomorphic to $H$ and $\langle G, \bar{f} \rangle$ is
  cofinitary.
\end{theorem}
\subsubsection*{Countable Groups}

Here we'll adapt the method from the previous section to
any countable group $H$ (as opposed to only finitely generated).

So let $H$ be a countable group, with generators $\vec{h} := \langle
h_n \mid n \in \N \rangle$.  Let $\vec{x} := \langle x_n \mid n \in \N
\rangle$ be variables, where obviously the intention is that $x_n$
will represent $h_n$.  Let $W_{H,\Id}$ be the set of words $w(\vec{x})
\in F(\vec{x})$ such that $w(\vec{h}) = \Id$; i.e. the set of words
representing the identity.  With this we have as before that $H$ is
isomorphic to $F(\vec{x}) / W_{H,\Id}$.

We want as before to start with $G$ a countable cofinitary group and add
$\vec{f}$ such that
\begin{itemize}
  \item the group generated by $\vec{f}$ is isomorphic to $H$,
  \item $\langle G, \vec{f} \rangle$ is cofinitary.
\end{itemize}

So let $G$ be a countable cofinitary group, and set $W_{G,H} := G *
(F(\vec{x})/W_{H,\Id})$.  This is a countable set, so we can enumerate
it by $\langle w_n \mid n \in \N \rangle$.

If $\vec{p} := \langle p_n \mid n \in \N \rangle$ is a sequence of
finite injective functions $\N \partialmap \N$ we define its support
$\supp(\vec{p}) := \{ i \in \N \mid p_i \neq \emptyset \}$.

The notion of $(G,H)$-good extension is the same as before except that
$\vec{p}$ and $\vec{q}$ are sequence of length $\omega$ of finite
injective functions, and in stead of finite sequences $\bar{y}$ we use
infinite sequences $\vec{y}$ (where $y$ is $x$, $p$, $q$, or $h$).

The same remarks apply to this order as in the last section.

We define a poset $\poset_{G,H} = \langle P, {\leq} \rangle$
associated with the groups $G$ and $H$.  The only difference in the
definition with the poset from the last section is that its elements
are length $\omega$ sequences with finite support of finite injective
partial maps.

Now we get to the major difference with the work before.  If we apply
all relations to an element of the poset $\poset_{G,H}$, then we are
likely to get a sequence with infinite support.  To avoid this we
``localize'' the applying of relations as follows.

If $\vec{p} \in P$ and $A$ is a finite set of natural numbers, we say
$\vec{q}$, a sequence of injective partial maps $\N \rightarrow \N$,
is obtained from $\vec{p}$ by \emph{$A$-applying relations} if
\begin{equation*}
  (a,b) \in q_i
    \quad \Leftrightarrow \quad
      \left(\exists w' \left[ x_i w' \in W_{H, \Id} \wedge w'(\vec{p})(b) =
    a\right] \right) \wedge \left( p_i \neq \emptyset \text{ or } i \in A \right).
\end{equation*}

Lemma \ref{lem:applying relations} where we replace applying relations
by $A$-applying relations can be proved as before.

Now we define for $i, k \in \N$ the sets $D_{i,k} := \{ \vec{p} \in
\mathbb{P}_{G,H} \mid k \in \dom(p_i) \}$ and $R_{i,k} := \{ \vec{p}
\in \mathbb{P}_{G,H} \mid k \in \ran(p_i) \}$.

\begin{lemma}
  For all $i \leq n$ and $k \in \N$, the sets $D_{i,k}$ and $R_{i,k}$
  are dense.
\end{lemma}

\begin{proof}
  The proof is basically the same as the proof of Lemma
  \ref{lem:fgdense} except that we don't apply relations, but we
  $\{i\}$-apply relations.
\end{proof}

From this lemma we get the following theorem by reasoning similar to
previous section.

\begin{theorem}
  \label{thm:countgrp}
  For $G$ a countable cofinitary group and $H$ a countable
  group, there exist $\vec{f}$ such that the group generated by
  $\vec{f}$ is isomorphic to $H$ and $\langle G, \vec{f} \rangle$ is
  cofinitary.
\end{theorem}

From this we get the following theorem

\begin{theorem}
  \label{thm:mcgcountembed}
  The Continuum Hypothesis implies that there exists a maximal
  cofinitary group $G$ such that every countable group $H$ embeds into
  $G$.
\end{theorem}

\begin{proof}
  Enumerate all countable groups $\langle H_\alpha \mid \alpha <
  \omega_1 \rangle$ and all permutations of $\N$ $\langle h_\alpha \mid \alpha
  < \omega_1 \rangle$.  At stage $\alpha$ we have already constructed
  the group $G_\alpha$ such that for all $\gamma < \alpha$ the group
  $H_\gamma$ embeds into $G_\alpha$ and $h_\gamma \in G_\alpha$ or
  $\langle G_\alpha, h_\gamma \rangle$ is not cofinitary.  Then extend
  the group $G_\alpha$ by $H_\alpha$ (by applying Theorem
  \ref{thm:countgrp} with $H$ equal to $H_\alpha$ and $G$ equal to
  $G_\alpha$) obtaining the group $G_\alpha'$.  When $h_\alpha \not
  \in G_\alpha'$ and $\langle G_\alpha', h_\alpha \rangle$ is
  cofinitary, add a function hitting $h_\alpha$ infinitely often (this
  is done by using the domain and range extension lemmas, and the
  hitting $f$ lemma).
\end{proof}

It should be noted that this is not the first proof of Theorem
\ref{thm:countgrp}: in \cite{C} we found that Truss and Adeleke had
proved the following theorem:

\begin{theorem}
  Let $G$ and $H$ be countable cofinitary groups.  Then there is a
  permutation $f : \N \rightarrow \N$ such that $\langle G, fHf^{-1}
  \rangle$ is cofinitary.
\end{theorem}

\subsection{Forcing Cofinitary Actions}

In this subsection we start our investigation of consistency of the
existence of cofinitary actions of groups for which we can't outright
find one.  We'll focus on the group $H := \bigoplus_{\alpha <
  \aleph_1} \mathbb{Z}_2$ (as this is one we understand already, and
the general case is not yet quite clear).

Let $\vec{e} := \langle e_\alpha \mid \alpha < \aleph_1 \rangle$ be
the standard generating set for $H$ ($e_\alpha(\alpha) = 1$ and
$e_\alpha(\beta) =0$ for all $\beta \neq \alpha$).

Let $\vec{x} := \langle x_\alpha \mid \alpha < \aleph_1 \rangle$ be a
sequence of variables, and $W_{H,\Id}$ the set of words $w(\vec{x})$
with $w(\vec{e}) = \Id$.  Then $F(\vec{x}) / W_{H,\id}$ is isomorphic
to $H$.  Also note that for this group $W_{H, \id}$ is generated by
$\{x_\alpha^2 \mid \alpha < \aleph_1 \} \cup \{ [ x_\alpha, x_\beta ]
\mid \alpha, \beta < \aleph_1\}$ (where $[x_\alpha,
x_\beta] = x_\alpha^{-1} x_\beta^{-1}x_\alpha x_\beta$, the commutator
of $x_\alpha$ and $x_\beta$).

We define a poset $\poset_{H} = \langle P, {\leq} \rangle$.  The set
$P$ consists of those sequences $\vec{p}$ of length $\aleph_1$ of
partial finite functions $N \partialmap \N$ with finite support (i.e.
$\{\alpha < \omega_1 \mid (\vec{p})_\alpha \neq \emptyset \}$ is finite)
for which $w(\vec{p})\cong \Id$ for all $w \in W_{H,\Id}$.  For $I$ a
finite subset of $\aleph_1$ we write $x_I$ for $\{x_\alpha \mid \alpha
\in I\}$.  Then we choose for any finite set $I \subseteq \aleph_1$ an
enumeration $\langle w_n^{I} \mid n \in \N \rangle$ of $F(x_I)$.  If
$\vec{p} \in P$, then $W_{\vec{p}}$ is the set of words $\bigcup_{I
  \subseteq \supp(\vec{p})} \{ w_n^{I} \mid n \leq \card{\vec{p}} \}$
and all their subwords.  If $\vec{p}, \vec{q} \in P$, then $\vec{q}
\leq \vec{p}$ iff for any word $w$ in $W_{\vec{p}}$ if $w(\vec{q})(l)
= l$ for some $l \in \N$, then there are $u_1, u_2 ,z$ subwords of $w$,
$m \in \N$ and $z' \in F(\vec{x})$ such that
\begin{itemize}
  \item $w = u_1^{-1} z u_2$ without cancellation,
  \item $u_1(\vec{e}) = u_2(\vec{e})$,
  \item $z(\vec{e}) = z'(\vec{e})$,
  \item $z'(\vec{p})(m) = m$, and
  \item $u_2(\vec{q})(l) = m$.
\end{itemize}

We first prove the following lemma.

\begin{lemma}
  $\poset_H$ is c.c.c..
\end{lemma}

\begin{proof}
  Let $A$ be an uncountable set of elements from $P$.

  Apply the $\Delta$-system lemma to $\{\supp(\vec{p}) \mid \vec{p}
  \in A\}$ to get $A'$, an uncountable subset of $A$, such that
  $\{\supp(\vec{p}) \mid \vec{p} \in A'\}$ is a $\Delta$-system
  with root $r_{\aleph_1}$.
  
  Then apply the $\Delta$-system lemma to $A'$ to get $\tilde{A}$, an
  uncountable subset of $A'$, a $\Delta$-system with root $\vec{r}$.
  Note that $\supp(\vec{r}) \subseteq r_{\aleph_1}$.

  In the coordinates in $r_{\aleph_1}$ there are only countably many
  possible finite extensions of $\vec{r}$, remove countably many
  members from $\tilde{A}$ so that for all the remaining members we
  have that $\vec{p} \upharpoonright r_{\aleph_1} = \vec{r}$.
  
  Let $C = \{ \vec{p} \mid \vec{p}$ a sequence of length $\omega$ of
  finite partial functions $\N \partialmap \N$ with finite
  support$\}$.  Any $\vec{p} \in P$ has finite support $I = \{\gamma_0
  < \gamma_1 < \cdots < \gamma_k\}$, so to $\vec{p}$ we can assign
  $F(\vec{p}) = \langle p_{\gamma_0}, p_{\gamma_1}, \ldots,
  p_{\gamma_k}, \emptyset, \ldots \rangle$.  Then $F : P \rightarrow
  C$.  Since $C$ is a countable set, there is an uncountable subset
  $\bar{A}$ of $\tilde{A}$ such that all members of $\bar{A}$ map to
  the same element.
  
  We claim that all members of $\bar{A}$ are compatible.  So let
  $\vec{p}, \vec{q} \in \bar{A}$, and let $\vec{s} = \vec{p} \cup
  \vec{q} = \langle p_\alpha \cup q_\alpha \mid \alpha < \aleph_1
  \rangle$.  We get $\vec{s} \leq \vec{p}$ since $\vec{s}$ does not
  extend any $p_\alpha \neq \emptyset$, and $W_{\vec{p}}$ only
  contains words in variables $x_\alpha$ where $p_\alpha \neq
  \emptyset$.  By the same reasoning $\vec{s} \leq \vec{q}$.  Also all
  members of $\vec{s}$ are of order two, so satisfy the relations
  $\{x_\alpha^2 = \Id \mid \alpha < \aleph_1 \}$.  If we can show that
  $\vec{s}$ also satisfies all $\{[x_\alpha, x_\beta] = \Id \mid \alpha,
  \beta < \aleph_1\}$ it follows that $\vec{s} \in P$.  Now suppose
  there are $\alpha, \beta \in \supp(\vec{s})$ s.t. $[x_\alpha,
  x_\beta](\vec{s}) \ncong \id$.  Then $s_\alpha$ and $s_\beta$ are
  different, and one is from $\vec{p}$ the other from $\vec{q}$.
  W.l.o.g. $s_\alpha = p_\alpha$ and $s_\beta = q_\beta$.  Since all
  elements of $\bar{A}$ map, under $F$, to the same element, there is
  a $\alpha' \in \supp(\vec{p})$ s.t. $q_\beta = p_{\alpha'}$.  But
  now we see that $[x_\alpha,x_\beta](\vec{s}) = [x_\alpha,
  x_{\alpha'}](\vec{p}) \cong \Id$, a contradiction.  So $\vec{s}$
  satisfies all relations after all, and $\vec{s} \in P$ showing the
  compatibility of $\vec{p}$ and $\vec{q}$.
\end{proof}

It is clear from our forcing and the reasoning in the last sections
that a $\poset_H$ generic set will give a cofinitary action for
$\check{H}$.  It remains to see that $\check{H} = H$ in the forcing
extension.  But since the forcing is c.c.c. no cardinals get
collapsed, and if we take for $H$ the group with underlying set the
finite subsets of $\aleph_1$ with the obvious group operation, we get
that $\check{H} = H$, showing the following theorem.

\begin{theorem}
  There is a c.c.c. notion of forcing $\mathbb{P}_H$ such that
  $V^{\mathbb{P}_H} \models \bigoplus_{\alpha < \aleph_1} \mathbb{Z}_2$
  has a cofinitary action.
\end{theorem}

From this we get the following theorem.

\begin{theorem}
  Martin's Axiom together with the negation of the continuum
  hypothesis implies $\bigoplus_{\alpha < \aleph_1} \mathbb{Z}_2$ has
  a cofinitary action.
\end{theorem}

\begin{proof}
  We don't need a generic set for all dense sets, only for the
  $\aleph_1$ many $D_{\alpha,i}$ and $R_{\alpha,i}$.
\end{proof}

\section{Cardinal Characteristics}

\subsection{Diamond Results}

In \cite{MHD}, Justin Moore, Michael Hru\v{s}\'{a}k and Mirna
D\v{z}amonja introduced weakenings of the diamond principle related to
cardinal characteristics.  We'll first study the effect of one of
these weakenings of the diamond principle on families related to the
symmetric group of the natural numbers.

\begin{definitions}
\item
  ${=}^\infty$ is the relation on Baire space, $\BS$, of
  \emph{infinite equality}: for $f,g \in \BS$, we have $f =^\infty
  g$ iff $\{n \in \N \mid f(n) = g(n) \}$ is infinite.

\item
  A function $F: {}^{<\omega_1} 2 \rightarrow \BS$ is a Borel function
  if for all $\delta < \omega_1$ the function $F\upharpoonright
  {}^\delta 2 : {}^\delta 2 \rightarrow \BS$ is Borel.

\item
  $\Diamond(\BS,{=^\infty})$ is the following guessing principle:

  \noindent for every Borel function $F : {}^{<\omega_1} 2 \rightarrow
  \BS$, there is a function $G: \omega_1 \rightarrow \BS$ such that for
  every $f : \omega_1 \rightarrow 2$ the set
  \begin{equation*}
    \{ \delta < \omega_1 \mid F(f \upharpoonright \delta) =^\infty
    G(\delta) \}
  \end{equation*}
  is stationary.

  A $G$ related to $F$ in this way is called a $\Diamond(\BS,
  {=^\infty})$-sequence for $F$.
\end{definitions}

We will study the effect of this $\Diamond$-principle (which is weaker
than ordinary $\Diamond$) on the cardinal invariants $\mathfrak{a}_p$
and $\mathfrak{a}_g$ (see Definition \ref{definition ap and ag}).

\begin{theorem}
  \label{thm:ap}
  $\Diamond(\BS,{=^\infty})$ implies $\mathfrak{a}_p = \aleph_1$.
\end{theorem}

\begin{proof}
  We will define a map $F$ such that the
  $\Diamond(\BS,{=^\infty})$-sequence for it will help us build a
  sequence of permutations $\langle p_\alpha \mid \alpha < \omega_1
  \rangle$ which will be a maximal almost disjoint family of
  permutations of $\N$.
  
  Let $C \subseteq \omega_1$ be the set of ordinals $\delta \in
  \omega_1$ such that $\delta = \omega + \omega \cdot \delta$.  $C$ is
  closed and unbounded in $\omega_1$.  Also fix a bijection $\theta :
  {}^\N 2 \rightarrow \Sym(\N)$, that is Borel.  Let $P_\delta$ be the
  set $\{(\langle p_\alpha \mid \alpha < \delta \rangle, p) \mid
  \{p_\alpha \mid \alpha < \delta \} \cup \{p\}$ is a set of
  permutations$\}$.  Pick $\delta \in C$ and let $\pi: \omega
  \mathrel{\dot{\cup}} \delta \times \omega \rightarrow \delta$ be an order
  isomorphism (where the order on $\omega \mathrel{\dot{\cup}} \delta \times
  \omega$ is such that all elements of $\omega$ are less then all
  elements of $\delta \times \omega$, and $\delta \times \omega$ is
  ordered lexicographically).  Then for $x \in {}^\delta 2$, define
  $D(x) = ( \langle \theta(x \restrict \pi[\{\alpha\} \cdot \omega]),
  \theta(x \restrict \pi[\omega]))$.  $D$ is a bijection ${}^\delta 2
  \rightarrow P_\delta$.

  To define $F : {}^{<\omega_1} 2 \rightarrow \BS$, by coding we let
  its range be ${}^\N (\N \cup {}^{<\omega}(\N \times \N))$.  For
  $\delta \not \in C$ we let $F \restrict {}^\delta 2$ be any constant
  map; for $\delta \in C$ by the coding in the previous paragraph we
  let its domain be the set of pairs $( \langle p_\alpha \mid \alpha <
  \delta \rangle, p)$ with $\{p_\alpha \mid \alpha < \delta \} \cup
  \{p\}$ a family of permutations and define $F$ as follows.  
  
  Fix for every $\delta< \omega_1$ a bijection $e_\delta : \N
  \rightarrow \delta$.

  If $\{p_\alpha \mid \alpha < \delta \} \cup \{p\}$ is not almost
  disjoint, we define $F(\langle p_\alpha \mid \alpha < \delta \rangle,
  p) (n) = n$.  Otherwise, we define $F(\langle p_\alpha \mid \alpha <
  \delta \rangle, p)(n)$ to be $\big((k_0,p(k_0)),(k_1,p(k_1)),
  \ldots,$ $(k_{6n},p(k_{6n}))\big)$ with
  \begin{itemize}
  \item $k_0$ the least number such that $p(k_0) \not \in \{
    p_{e_\delta(j)}(k_0) \mid j \leq n \}$, and
  \item $k_{i+1}$ the least number strictly bigger than $k_i$ with
    $p(k_{i+1}) \not \in \{p_{e_\delta(j)}(k_{i+1}) \mid j \leq n \}$.
  \end{itemize}
  Since the family is almost disjoint, these $k_i$ exist.
  
  For any $\delta \in C$ the function as defined above restricted to
  those $(\langle p_\alpha \mid \alpha < \delta \rangle,p)$ for which
  $\{ p_\alpha \mid \alpha < \delta\} \cup \{p\}$ is an almost
  disjoint family is continuous, and the set of $(\langle p_\alpha
  \mid \alpha < \delta \rangle, p)$ for which $\{ p_\alpha \mid \alpha
  < \delta\} \cup \{p\}$ is an almost disjoint family is a Borel set.
  Composing this with the Borel coding $D$ this shows that $F$ is a
  Borel function.

  Let $G: \omega_1 \rightarrow \BS$ be a
  $\Diamond(\BS,{=^\infty})$-sequence for this $F$.  We define
  $G(\delta)(n)$ to be a \emph{valid guess} for $\langle p_\alpha \mid
  \alpha < \delta \rangle$, a family of almost disjoint permutations,
  iff
  \begin{itemize}
    \item $G(\delta)(n) = \big( (k_0, o_0), (k_1,o_1), \ldots ,
      (k_{6n},o_{6n})\big)$ for some $k_i, o_i \in \N$,
    \item all $k_i$ are distinct, and
    \item all $o_i$ are distinct and $o_i \not \in
      \{p_{e_\delta(j)}(k_i) \mid j \leq n \}$.
  \end{itemize}

  Note that for any $\delta < \omega_1$, $n \in \N$, and any
  permutation $p$ almost disjoint from all $p_\alpha$, if $F(\langle
  p_\alpha \mid \alpha < \delta \rangle, p)(n) = G(\delta)(n)$ then
  $G(\delta)(n)$ is a valid guess for $\langle p_\alpha \mid \alpha <
  \delta \rangle$.
  
  Now we use $G$ to construct $\langle p_\alpha \mid \alpha < \omega_1
  \rangle$ recursively.  Suppose $\langle p_\alpha \mid \alpha <
  \delta \rangle$ has been defined.  Then define $p_\delta$
  recursively, $p_\delta := \bigcup_{s \in \N } p_{\delta,s}$, where
  \renewcommand{\theenumi}{P\arabic{enumi}}
  \begin{enumerate}
    \item $p_{\delta,0} := \emptyset$,
    \item $p'_{\delta,s+1} := p_{\delta, s}$ if $G(\delta)(s)$ is not
      a valid guess for $\langle p_\alpha \mid \alpha < \delta
      \rangle$,
    \item \label{item:step3}
      $p'_{\delta, s+1} := p_{\delta,s} \cup \{ (k_i,o_i) \}$ if
      $G(\delta)(s) = \big((k_0,o_0), (k_1,o_1), \ldots,
      (k_{6s},o_{6s})\big)$ is a valid guess for $\langle p_\alpha
      \mid \alpha < \delta \rangle$ and $i$ is least such that $k_i
      \not \in \dom(p_{\delta,s})$ and $o_i \not \in
      \ran(p_{\delta,s})$,
    \item \label{item:step4}$p''_{\delta,s+1} := p'_{\delta,s+1} \cup
      \{(a,b)\}$ where $a$ is the least number not in
      $\dom(p'_{\delta,s+1})$ and $b$ is the least number not in
      $\ran(p'_{\delta,s+1})$ and not in $\{p_{e_\delta(j)}(a) \mid j
      \leq s \}$, and
    \item \label{item:step5}$p_{\delta,s+1} := p''_{\delta,s+1} \cup
      \{(c,d)\}$ where $d$ is the least number not in
      $\ran(p''_{\delta,s+1})$ and $c$ is the least number not in
      $\dom(p''_{\delta,s+1})$ and not in $\{p^{-1}_{e_\delta(j)}(d)
      \mid j \leq s\}$.
  \end{enumerate}

  Note that $|p_{\delta,s}|$ is at most $3s$.  This means we can do
  step \ref{item:step3}, as the requirement $k_i \not \in
  \dom(p_{\delta,s})$ excludes at most $3s$ pairs in $G(\delta)(s)$,
  $o_i \not \in \ran(p_{\delta,s})$ excludes at most another $3s$
  pairs in $G(\delta)(s)$, and $G(\delta)(s)$ has $6s+1$ pairs, always
  leaving at least one pair.
  
  Now $p_{\delta}$ is a permutation (by \ref{item:step4} and
  \ref{item:step5}) almost disjoint from all $p_\alpha$, $\alpha <
  \delta$.  This completes the construction of $\langle p_\alpha \mid
  \alpha < \omega_1 \rangle$.

  It remains to see that this almost disjoint family of permutations
  is maximal.  We do this by contradiction; suppose, therefore, that
  there is a permutation $p$ almost disjoint from all $p_\alpha$,
  $\alpha < \omega_1$.  Then the set
  \begin{equation*}
    \{ \delta < \omega_1 \mid F(\langle p_\alpha \mid \alpha <
      \delta\rangle, p) =^\infty G(\delta) \}
  \end{equation*}
  is stationary.  Remember that we use a coding for the inputs of the
  function $F$, and note that the sequence $\delta \mapsto (\langle
  p_\alpha \mid \alpha < \delta \rangle,p)$ determines a path $f:
  \omega_1 \rightarrow 2$ in the tree ${}^{<\omega_1}2$.
  
  Now let $\delta \in C$ be a member of this set.  Then $F(\langle
  p_\alpha \mid \alpha < \delta \rangle, p) = ^\infty G(\delta)$,
  which means there are infinitely many $n$ such that $G(\delta)(n)$
  is a valid guess for $\langle p_\alpha \mid \alpha < \delta
  \rangle$, and all the pairs in $G(\delta)(n)$ belong to $p$.  So we
  hit $p$ infinitely often with $p_\delta$ (by \ref{item:step3}),
  which is a contradiction.
\end{proof}

We now start to work towards the second theorem of this subsection;
this will be a theorem similar to the above but for $\mathfrak{a}_g$.

\begin{definition}
  For $w \in W_G$ (see Definition \ref{definitions WG etc}) and $p
  \subseteq q$ we call $q$ a \emph{very good extension} of $p$ with
  respect to $w$ if $w(q)$ has no more fixed points than $w(p)$.
\end{definition}

Note that a very good extension is a good extension (see Definition
\ref{definition good extension}). 

The following two lemmas show that we can construct a function $F$
similar to the $F$ in the proof of Theorem \ref{thm:ap} but for
maximal cofinitary groups.

\begin{lemma}
  Let $H$ be a cofinitary group, $f \in \Sym(\N) \setminus H$ such
  that $\langle H,f \rangle$ is a cofinitary group and $w \in W_H$.
  Then for every $k \in \N$ there exists a finite set $S$ of pairs
  from $f$ such that for every finite injective map $p$ with $|p|$
  less than $k$ there exists a pair $(a,b)$ in $S$ such that $p \cup
  \{(a,b)\}$ is a very good extension of $p$ with respect to $w$.
\end{lemma}

\begin{proof}
  First we will find an infinite subset $f'$ of $f$ such that $w(f')$
  has no fixed points, then we'll show that a big enough finite subset
  of $f'$ exists.  The first step ensures that we don't have to worry
  about fixed points caused by pairs from $f$ alone.  The second part
  is done by counting how many pairs from $f'$ could combine with
  pairs from $p$ to cause a fixed point.

  Obtaining $f'$ from $f$ is done differently depending on whether
  $w(f)$ is the identity or not.

  If $w(f)$ is not the identity, then it has only finitely many fixed
  points.  Let $f'$ be equal to $f$ with for each of those finitely
  many fixed points one pair from $f$ used in the evaluation path of
  that fixed point removed.  We have ensured that $w(f')$ has no
  fixed points.

  If $w(f)$ is the identity, then we know there is more than one
  occurrence of $x$ in $w(x)$ (since $f \not \in H$).  So either there
  is an occurrence of $x^2$ or $x^{-2}$, or there is a subword of the
  form $x^{\epsilon_0} g x^{\epsilon_1}$, with $\epsilon_i \in
  \{{-1},{+1}\}$ and $g \in H$.  In either case there are only
  finitely many evaluation paths of $w(f)$ that use the same pair from
  $f$ in both these selected occurrences of $x$ (use that $f$ has only
  finitely many fixed points for the first case, and that $f \not \in
  H$ for the second case).  Remove these finitely many pairs from $f$
  to obtain $f''$.

  Now we have to find an infinite subset $f'$ of $f''$ such that
  $w(f')$ is nowhere defined (which in this case is equivalent to not
  having fixed points).
  
  We do this by recursively defining an enumeration $\{e_n \mid n \in
  \N\}$ of $f'$.  Let ${\prec}$ be a wellorder of $\N \times \N$.
  Then define $e_0$ to be the ${\prec}$-least pair $(a,b)$ in $f''$,
  and $e_{n+1}$ to be the ${\prec}$-least pair in $f''$ that is not
  used in any evaluation path where a pair in $\{e_0, \ldots, e_n\}$
  is also used.

  We end up with an infinite $f'$ such that $w(f')$ is indeed nowhere
  defined.
  
  Now we examine for a given $k$ and $p$, an injective map with $|p| =
  l \leq k$, how many pairs $(a,b)$ of $f'$ can have that $p \cup
  \{(a,b)\}$ is not a very good extension of $p$ for $w$.

  First there are at most $2l$ pairs $(a,b)$ from $f'$ that have $a
  \in \dom(p)$ or $b \in \ran(p)$.  Remove these from $f'$ to obtain
  $\tilde{f}$.  Now we look at $w(p \cup \tilde{f})$; any fixed point
  of $w(p \cup \tilde{f})$ that was not a fixed point of $w(p)$ has an
  evaluation path where both pairs from $p$ and from $\tilde{f}$ are
  used.  If we remove one pair from $\tilde{f}$ for each of those
  evaluation paths to obtain $\hat{f}$ the partial permutation $w(p
  \cup \hat{f})$ will only have fixed points that $w(p)$ already had.
  
  So we only have to find an upper bound for the number of evaluation
  paths using pairs from both $p$ and $\tilde{f}$.  This upper bound
  is attained if, for each occurrence of $x$ in $w$ and any pair of
  $p$, it gets completed to an evaluation path with all pairs from
  $\tilde{f}$.  This gives us $|p| \cdot \ocx(w)$ as an upper bound,
  where $\ocx(w)$ is the number of occurrences of $x$ and $x^{-1}$ in
  $w$.

  So in total at most $2l + l \cdot \ocx(w)$ pairs $(a,b)$ of $f'$ are
  such that $p \cup \{(a,b)\}$ is not a very good extension of $p$
  with respect to $w$.
  
  This means that if we take $S$ to consist of any $2k + k \cdot
  \ocx(w) + 1$ pairs of $f'$ we have a set as desired.
\end{proof}

We need and easily get the following stronger lemma.

\begin{lemma}
  \label{lem:Sexists}
  Let $H$ be a cofinitary group, $f \in \Sym(\N) \setminus H$ such
  that $\langle H,f \rangle$ is a cofinitary group and $w_0, \ldots,
  w_n \in W_H$.  Then for every $k \in \N$ there exists a finite set
  $S$ of pairs from $f$ such that for every injective map $p$ with
  $|p|$ less than $k$ there exists a pair $(a,b) \in S$ such that $p
  \cup \{(a,b)\}$ is a very good extension of $p$ for all the words
  $w_0, \ldots, w_n$.
\end{lemma}

\begin{proof}
  By applying the method used in the first half of the proof of the
  last lemma $n+1$ times we can find an infinite $f' \subseteq f$ such
  that none of $w_0(f'), \ldots, w_n(f')$ have fixed points.  Then
  using the method in the second half of the proof of the last lemma
  also $n+1$ times we can find how big a subset $S$ of $f'$ we have to
  choose.
\end{proof}

Now we are ready to state and prove the second theorem of this
subsection.

\begin{theorem}
  \label{thm:ag}
  $\Diamond(\BS,{=^\infty})$ implies $\mathfrak{a}_g = \aleph_1$.
\end{theorem}

\begin{proof}
  We use the same strategy as in the proof of the previous theorem: we
  define a function $F$ whose $\Diamond(\BS,{=^\infty})$-sequence
  helps us build a maximal cofinitary group $\langle \{g_\alpha \mid
  \alpha < \omega_1\}\rangle$.
  
  By coding, as in the proof of Theorem \ref{thm:ap}, we let its
  domain be the set of pairs $(\langle g_\alpha \mid \alpha < \delta
  \rangle, g)$ with $\delta < \omega_1$ and $\{g_\alpha \mid \alpha <
  \delta \} \cup \{g\}$ a family of permutations.  This coding works
  on a club $C \subseteq \omega_1$, which is enough.  Also by coding
  we let its range be ${}^\N (\N \cup {}^{<\omega}(\N \times \N))$. We
  also fix for every $\delta < \omega_1$ a bijection $e_\delta : \N
  \rightarrow \delta$.

  For $\langle g_\alpha \mid \alpha < \delta \rangle$ a sequence of
  permutations we let $n \mapsto \tilde{w}_n$ be an enumeration of
  $W_{\langle \{g_\alpha \mid \alpha < \delta\} \rangle}$.
  
  Now we can define $F$.  On the levels $\delta < \omega_1$ where the
  chosen coding for the input does not work, define $F$ to be any
  constant map.  On the levels where the coding does work, define
  $F(\langle g_\alpha \mid \alpha < \delta \rangle,g)(n)$ to be either
  $m$, the least code for $\big( (k_0,g(k_0)), (k_1,g(k_1)), \ldots,
  (k_N, g(k_N))\big)$ such that for every injective partial map $p :
  \N \rightharpoonup \N$ with $|p| \leq 3n$ there is a pair
  $(k_i,g(k_i))$ coded in $m$ such that $p \cup \{(k_i,g(k_i))\}$ is a
  very good extension of $p$ with respect to all words $\tilde{w}_0,
  \ldots, \tilde{w}_n$, or $0$ if such a code does not exist.
  
  Note that by Lemma \ref{lem:Sexists}, if $\{g_\alpha \mid \alpha <
  \delta \} \cup \{g\}$ generates a cofinitary group and $g \not \in
  \langle \{g_\alpha \mid \alpha < \delta \} \rangle$, then there is
  such a code $m$.  Also note that the function $F$ is Borel (which
  can be shown in the same way we showed $F$ in the proof of
  Theorem \ref{thm:ap} to be Borel).

  Let $G : \omega_1 \rightarrow \BS$ be a $\Diamond(\BS,
  {=^\infty})$-sequence for this $F$. We define $G(\delta)(n)$ to be a
  \emph{valid guess} for $\langle g_\alpha \mid \alpha < \delta
  \rangle$, a family of permutations that generates a cofinitary group,
  iff
  \renewcommand{\theenumi}{V\arabic{enumi}}
  \begin{itemize}
    \item $G(\delta)(n) = \big( (k_0, o_0), (k_1, o_1), \ldots, (k_N,
      o_N) \big)$ for some $k_i, o_i \in \N$ and $N \in \N$,
    \item all $k_i$ are distinct,
    \item all $o_i$ are distinct, and
    \item for every partial injective map $p : \N \rightharpoonup \N$
      with $|p| \leq 3n$ there is a pair $(k_i,o_i)$ such that $p \cup
      \{(k_i,o_i)\}$ is a very good extension of $p$ with respect to
      all words $\tilde{w}_0, \ldots, \tilde{w}_n$.
  \end{itemize}
  Note that for any $\delta < \omega_1$, $n \in \N$, and any
  permutation $g$ such that $g \not \in \langle \{g_\alpha \mid \alpha <
  \delta \} \rangle$ and $\langle \{g_\alpha \mid \alpha < \delta \}
  \cup \{g\} \rangle$ is cofinitary, if $F(\langle g_\alpha \mid
  \alpha < \delta \rangle, g )(n) = G(\delta)(n)$ then $G(\delta)(n)$
  is a valid guess for $\langle g_\alpha \mid \alpha < \delta
  \rangle$.
  
  Now we use $G$ to construct recursively $\langle g_\alpha \mid
  \alpha < \omega_1 \rangle$, a sequence of permutations which
  generates a maximal cofinitary group.  So suppose $\langle g_\alpha
  \mid \alpha < \delta \rangle$ has been constructed.  Then construct
  $g_\delta := \bigcup_{s \in \N} g_{\delta,s}$ recursively by:
  \renewcommand{\theenumi}{P\arabic{enumi}}
  \begin{enumerate}
    \item $g_{\delta,0} := \emptyset$,
    \item $g'_{\delta,s+1} := g_{\delta,s}$ if $G(\delta)(s)$ is not a
      valid guess for $\langle g_\alpha \mid \alpha < \delta
    \rangle$,
    \item \label{item:agstep3}
      $g'_{\delta,s+1} := g_{\delta,s} \cup \{ (k_i,o_i) \}$ if
      $G(\delta)(s) = \left( (k_0,o_0), \ldots, (k_N,o_N) \right)$ is
      a valid guess for $\langle g_\alpha \mid \alpha < \delta
      \rangle$ and $i$ is least such that $p \cup \{(k_i,o_i)\}$ is a
      very good extension of $p$ for all words $\tilde{w}_0, \ldots,
      \tilde{w}_n$,
    \item \label{item:agstep4} $g''_{\delta,s+1} := g'_{\delta,s+1}
      \cup \{(a,b)\}$ where $a$ is the least number not in
      $\dom(g'_{\delta,s+1})$ and $b$ is the least number such that
      $g'_{\delta,s+1} \cup \{(a,b)\}$ is a good extension of
      $g'_{\delta,s+1}$ for all words $\tilde{w}_0, \ldots,
      \tilde{w}_n$ (this $b$ exists by the domain extension lemma),
      and
    \item \label{item:agstep5} $g_{\delta,s+1} := g''_{\delta,s+1}
      \cup \{(c,d)\}$ where $d$ is the least number not in
      $\ran(g''_{\delta,s+1})$ and $c$ is the least number such that
      $g''_{\delta,s+1} \cup \{(c,d)\}$ is a good extension of
      $g''_{\delta,s+1}$ with respect to all words $\tilde{w}_0,
      \ldots, \tilde{w}_n$ (this $c$ exists by the range extension
      lemma).
  \end{enumerate}

  Note that $|g_{\delta,s}|$ is at most $3s$ which means we can always
  perform step \ref{item:agstep3} when applicable.
  
  Now $g_\delta$ is a permutation (by \ref{item:agstep4} and
  \ref{item:agstep5}) such that $\{g_\alpha \mid \alpha < \delta \}
  \cup \{g_\delta\}$ generates a cofinitary group; this completes the
  construction of $\langle g_\alpha \mid \alpha < \omega_1 \rangle$.

  It remains to see that this group is \emph{maximal} cofinitary.
  Suppose, towards a contradiction, that there is $g \in
  \Sym(\N)$ such that $g \not \in \langle \{ g_\alpha \mid \alpha <
  \omega_1 \} \rangle$ and that $\langle \{g_\alpha \mid \alpha <
  \omega_1 \},g\rangle$ is a cofinitary group.  Then the set
  \begin{equation*}
    \{ \delta < \omega_1 \mid F(\langle g_\alpha \mid \alpha < \delta
    \rangle, g) =^\infty G(\delta) \}
  \end{equation*}
  is stationary.  Remember that we use a coding for the inputs of the
  function $F$, and note that the sequence $\delta \mapsto (\langle
  g_\alpha \mid \alpha < \delta \rangle, g)$ determines a path $f :
  \omega_1 \rightarrow 2$ in the tree ${}^{<\omega_1} 2$.  Now let
  $\delta \in c$ be a member of this set. Then $F(\langle g_\alpha
  \mid \alpha < \delta \rangle, g) =^\infty G(\delta)$, which means
  that for infinitely many $n$, the value $G(\delta)(n)$ is a valid
  guess for $\langle g_\alpha \mid \alpha < \delta \rangle$ and all
  pairs in $G(\delta)(n)$ belong to $g$. This means we hit $g$
  infinitely often with $g_\delta$ (by \ref{item:agstep5}), which is a
  contradiction.
\end{proof}

Combining Theorem \ref{thm:ap} and Theorem \ref{thm:ag} with
\begin{theorem}[\cite{MHD}]
  $\Diamond(\BS,{=^\infty})$ is true in the Miller model.
\end{theorem}
\noindent we have

\begin{corollary}
  $\mathfrak{a}_p =\mathfrak{a}_g = \aleph_1$ is true in the Miller model.
\end{corollary}

Then with $\mathfrak{g}\leq c(\Sym(\N))$, from \cite{BL}
($\mathfrak{g}$ is the groupwise density, for a definition see
\cite{B}, and $c(\Sym(\N))$ is the cofinality of the symmetric group,
see Definition \ref{def part Cichon}), and the fact, from \cite{B},
that the cardinal $\mathfrak{g}$ is $\aleph_2$ in the Miller model, we
obtain the following theorem.

\begin{theorem}
  In the Miller model $\mathfrak{a}_p = \mathfrak{a}_g = \aleph_1 <
  \aleph_2 = c(\Sym(\N))$.
\end{theorem}

\subsection{Template Forcing Result}

In this subsection we prove the following theorem (for
$\add(\mathcal{N})$ and $\cof(\mathcal{N})$ see Definition \ref{def
  part Cichon}).

\begin{theorem}
  \label{thm:template}
  \label{theorem template}
  The continuum hypothesis implies that for all regular cardinals
  $\lambda > \mu > \aleph_1$ with $\lambda = \lambda^\omega$, there is
  a forcing extension in which $\add(\Null) = \cof(\Null) = \mu$ and
  $\mathfrak{a}_g = 2^{\aleph_0} = \lambda$.
\end{theorem}

This proof is derived from the work of J\"org Brendle in \cite{JB} and
suggestions from Tapani Hytinnen.  We will show the details for the
arguments Brendle only outlines, but refer to \cite{JB} for the
arguments that are given completely there.

\subsubsection*{The Basic Forcing}

Our basic forcing will be localization forcing.  Here we collect the
basic facts about this forcing for the convenience of the reader.

\begin{definition}
For $S$ a set, and $n \in \N$ define
\begin{itemize}
\item $\finset{S}$ to be the set of finite subsets of $S$.  

\item ${}^{\leq n}[S]$ to be the set of subsets of $S$ of size $\leq n$.

\item $\finsetd{S}{n}$ to be the set of size $n$ subsets of $S$.

\item $\finseq{S}$ to be the set of finite sequences from $S$.

\item $\funcN{S}$ to be the set of functions from $\N$ to $S$.
\end{itemize}
\end{definition}

The localization forcing notion $\loc = \langle P, \leq \rangle$ is
defined by
\begin{itemize}
  \item $P$ is the set of all pairs $(\sigma, \varphi) \in
    \finseq{\finset{\N}} \times \funcN{(\finset{\N})}$ such that
    $|\sigma(i)| \leq i$ for all $i < \lh(\sigma)$ and $|\varphi(i)|
    \leq \lh(\sigma)$ for all $i$.
  \item $(\tau, \psi) \leq (\sigma, \varphi)$ if $\lh(\tau) \geq
    \lh(\sigma)$, $\sigma \subseteq \tau$, $\varphi(j) \subseteq
    \tau(j)$ for all $\lh(\sigma) \leq j < \lh(\tau)$ and $\varphi(j)
    \subseteq \psi(j)$ for all $j$.
\end{itemize}

Note that here $\sigma \subseteq \tau$ means that $\sigma = \tau
\restrict \lh(\sigma)$.

\begin{lemma}
  $\loc$ is $\sigma$-linked.
\end{lemma}

\begin{proof}
  A sufficient (but not necessary) condition for $(\sigma_1,
  \varphi_1)$ and $(\sigma_2, \varphi_2)$ to be compatible is:
  \begin{equation*}
    \sigma_1 = \sigma_2 \ \wedge \ (\varphi_1(i) = \varphi_2(i) \text{
      for all } i \leq 2\lh(\sigma_1)).
  \end{equation*}
  There are only countably many choices for $\sigma$ and then only
  countably many for the part of $\varphi$ that matters.
\end{proof}

A \emph{slalom} is a function $\phi: \N \rightarrow
\finset{\N}$ such that for all $n \in \N$ we have $|\phi(n)|
\leq n$.  We say a slalom \emph{localizes a real} $f \in
\BS$ if there is an $m \in \N$ such that for all $n \geq
m$ we have $f(n) \in \phi(n)$.

\begin{lemma}
  $\loc$ adds a slalom which localizes all ground model reals.
\end{lemma}

\begin{proof}
  If $G$ is $\loc$ generic, then $\phi = \bigcup_{p \in G} \pi_0(p)$
  is a function $\N \rightarrow \finset{\N}$ such that $|\phi(n)| \leq
  n$, i.e. a slalom.
  
  For $f \in \BS$ define $D_f = \{ p \in \loc \mid p \forces f$ is
  localized by $\bigcup_{g \in \dot{G}} \pi_0(p)\}$ ($\dot{G}$ a name
  for the generic).  Then $D_f$ is dense: If $p=(\sigma, \varphi) \in
  \loc$, then $(\sigma \hat{\phantom{x}} \varphi(\lh(\sigma)),
  \varphi')$, with $\varphi'(n) = \varphi(n) \cup \{f(n)\}$, is an
  extension of $p$ that is a member of $D_f$.
\end{proof}

Note that \cite[page 106]{BJ} defines localization forcing,
$\mathbf{LOC}$, to have underlying set $\{S \in \funcN{(\finset{\N})}
\mid \forall n |S(n)| \leq n \wedge \exists k,N \forall n \geq N
|S(n)| \leq k \}$ with order $S \leq S'$ iff $\forall n [S'(n)
\subseteq S(n)]$.  We see that this forcing is equivalent to $\loc$ by
noting that the set $\mathcal{L}$ of those $S \in \mathbf{LOC}$ with
for all $i < l$, $|S(i)| = i$, where $l$ is the least number such that
for all $n$ we have $|S(n)| \leq l$, is dense in $\mathbf{LOC}$.  And
the suborder $(\mathcal{L}, \leq)$ densely embeds in $\loc$ by the
embedding $F: \mathbf{LOC} \rightarrow \loc$ defined by $F(S) =
(\sigma, \varphi)$ with $\varphi = S$ and $\lh(\sigma) = l$, where $l$
is the least number such that for all $n$ we have $|S(n)| \leq l$ and
$\sigma(i) = S(i)$ for $i \leq l$.

The reason for using this forcing here comes from the following
characterizations of $\add(\Null)$ and $\cof(\Null)$ given by
Bartoszy\'nski, see \cite[Chapter 2]{BJ}.

\begin{theorem}
  $\add(\Null)$ is the least cardinality of a family $F \subseteq
  \BS$ such that there is no slalom localizing all members
  of $F$.

  $\cof(\Null)$ is the least cardinality of a family $\Phi$ of
  slaloms such that every member of $\BS$ is localized by a
  member of $\Phi$.
\end{theorem}

The construction of the model for Theorem \ref{theorem template} is by
iterated forcing.  The general structure of an iterated forcing
construction consists of a system $\langle \mathbb{P}_i \mid i \in O
\rangle$ of partial orders indexed by a directed set $(O,<)$ with
maximal element $m$ such that if $i < j$ then there is a complete
embedding $p_{i,j} : \mathbb{P}_i \rightarrow \mathbb{P}_j$.  Our
forcing will be such that there is a cofinal set $C \subseteq O
\setminus \{m\}$ of order type $\mu$ such that for $\alpha \in C$,
$\mathbb{P}_\alpha = \mathbb{Q}_\alpha * \loc$ (the two step iteration
of $\mathbb{Q}_\alpha$ and $\loc$) for some $\mathbb{Q}_\alpha$ and
for all $i < \alpha$ the embedding $p_{i,\alpha}$ maps into
$\mathbb{Q}_\alpha$.  The forcing will also have the property that for
any $\mathbb{P}_m$ generic $G$ and any real $r \in V[G]$, there is
some $i < m$ such that $r \in V[\mathbb{P}_i \cap G]$.

This shows that the family of slaloms added at the coordinates in $C$
gives a family $\Phi$ of slaloms in $V^{\mathbb{P}_m}$ such that any
real is localized by a member of $\Phi$, which with the theorem above
shows that $\cof(\Null) \leq \mu$ in $V^{\mathbb{P}_m}$.

Using similar ideas we see that $\add(\Null) \geq \mu$ in
$V^{\mathbb{P}_m}$.  If $F$ is a family of reals of size $< \mu$ in
$V[G]$, then it already appears in $V[G \cap \mathbb{P}_i]$ for some
$i < m$.  But then it is localized by the slalom that gets added by
any forcing $\mathbb{P}_\alpha$ with $\alpha > i$.

This shows that after we have constructed a forcing that has the
properties claimed here, in the generic extension all cardinal numbers
in Cicho\'n's diagram will be equal to $\mu$  (all forcings used
are c.c.c. so that cardinals in $V$ and $V[G]$ are the same).

We will then show by an isomorphism-of-names argument that in the
forcing extension no cofinitary group of cardinality less than
$2^{\aleph_0}$ (which will equal $\lambda$ by a counting-of-names
argument) and bigger or equal to $\mu$ can exist.  This shows that
$\mathfrak{a}_g = 2^{\aleph_0}$ using the following result by Brendle,
Spinas, and Zhang (for $\non(\mathcal{M})$ see Definition \ref{def
  part Cichon}).

\begin{theorem}[\cite{BSZ}]
  $\mathfrak{a}_g \geq \non(\mathcal{M})$.
\end{theorem}

\subsubsection*{This Forcing Along a Template}

Here we introduce templates and describe how to iterate our basic forcing
along them.  The notion of template forcing was introduced by Saharon
Shelah in \cite{Sh}.  J\"org Brendle in \cite{JB} wrote an introduction to this
theory. In \cite{JB} Brendle indicates how to change his definitions
and proofs to work for localization forcing; we fill in the details.

In the following $L$ will always denote a linear order $(L, \leq)$,
and $L_x$ the initial segment of $L$ determined by $x$, i.e. $L_x = \{
y \in L \mid y < x \}$.

The directed system of posets will be indexed by $(\powset(L),
\subseteq)$ for some $L$.  So for each subset $A$ of $L$ we need to
construct a poset.  This is done by using a system of subsets of $A$
along which we can recursively define the poset, and which has enough
structure to prove the properties we need.  In this section we will
show how to do this.  The order $L$ (and template $\I$) we use will be
defined in the next section.

\begin{definition}
  A \emph{template} is a pair $(L,\mathcal{I})$, with $\mathcal{I}
  \subseteq \mathcal{P}(L)$ satisfying:
  \begin{enumerate}
    \item \label{template prop1} $\emptyset, L \in \mathcal{I}$;
    \item \label{template prop2} $\mathcal{I}$ is closed under finite
      unions and finite intersections;
    \item \label{template prop3} if $y,x \in L$ are such that $y < x$,
      then there is an $A \in \mathcal{I}$ such that $A \subseteq L_x$
      and $y \in A$;
    \item \label{template prop4} for all $A \in \mathcal{I}$ and $x
      \in L \setminus A$ the set $A \cap L_x$ belongs to
      $\mathcal{I}$;
    \item \label{template prop5} $\mathcal{I}$ is wellfounded with
      respect to ${\subseteq}$.
  \end{enumerate}
\end{definition}

Since $\mathcal{I}$ is wellfounded we can define a rank function $\Dp
: \mathcal{I} \rightarrow \ON$ by $\Dp(\emptyset) := 0$ and $\Dp(A) =
\sup\{ \Dp(B) +1 \mid B \in \mathcal{I} \wedge B \subsetneq A\}$.

We want to use one template to generate all the posets in the directed
system.  This is done by defining, in the next section, one big
template and, for the smaller indices, using induced templates as defined
next.

If $L' \subseteq L$ we define the \emph{induced template} $(L',
\mathcal{I}\restrict L')$ by $\mathcal{I} \restrict L' := \{ A \cap L'
\mid A \in \mathcal{I} \}$.  An induced template is in fact a
template: properties \ref{template prop1}---\ref{template prop4} are
easy to verify.  To verify \ref{template prop5} suppose, towards a
contradiction, that $A'_n$, $ n \in \N$, is an infinite strictly
decreasing sequence in $\mathcal{I}\restrict L'$, then for each $A'_n$
there is an $A_n \in \mathcal{I}$ such that $A_n \cap L' = A'_n$.
Then $\tilde A_n = \bigcap_{k \leq n} A_n$ will be an infinite
strictly decreasing sequence in $\mathcal{I}$.

For $p$ a finite sequence with domain contained in the order $L$, we
write $\md(p)$ for $\max \dom(p)$.

\begin{definition}
  The poset $\poset (L, \mathcal{I}) $ has underlying set recursively
  defined as follows:
  \begin{enumerate}
    \item
      $\poset (\emptyset, \I \restrict \emptyset) = \{\emptyset\}$;
    \item $p \in \poset (L, \I)$ if $p$ is a finite sequence with
      $\dom(p) \subseteq L$ and there exists $B \in \mathcal{I}$ such
      that $B \subseteq L_{\md(p)}$, $p \restrict L_{\md(p)} \in
      \poset (B, \I \restrict B)$ and $p(\md(p)) = (\sigma,
      \dot{\varphi})$, with $\sigma \in \finseq{\finset{\N}}$ and
      $\dot{\varphi}$ a $\poset(B, \I \restrict B)$ name such that $p
      \restrict L_{\md(p)} \forces_{\poset(B, \I \restrict B)}
      (\check{\sigma}, \dot{\varphi}) \in \loc$.
  \end{enumerate}
\end{definition}

Note that in this definition in particular we have that $p \restrict
L_{\md(p)} \forces_{\poset(B, \I \restrict B)} \dot{\varphi} \in
\funcN{([\mathbb{N}]^{\leq \lh(\check{\sigma})})}$.  Also we don't
use all names for forcing conditions in $\loc$, since the first
projection of any $p(x)$ is an element of
$\finseq{\finset{\N}}$, not a name for such an element; the
set of names we use is however dense in $\loc$.

\begin{definition}
  For $p,q \in \poset(L, \I)$, we define $q \leq p$ if $\dom(p)
  \subseteq \dom(q)$ and
  \begin{enumerate}
    \item
      if $\md(p) = \md(q)$, then there is a $B \in \mathcal{I}$ with $B
      \subseteq L_{\md(p)}$ such that $\models q \restrict L_{\md(p)}
      \leq_{\poset( B, \I \restrict B)}  p \restrict L_{\md(p)}$ and
      $q \restrict L_{\md(p)} \forces_{\poset(B, \I \restrict B)}
      q(\md(p)) \leq_{\loc} p(\md(p))$;
    \item
      if $\md(q) > \md(p)$ then there is a $B \in \mathcal{I}$ with $
      \dom(q \restrict L_{\md(q)} ) \subseteq B \subseteq L_{\md(q)}$
      such that $ q \restrict L_{\md(q)} \leq_{\poset( B, \I\restrict
      B)} p$.
  \end{enumerate}
\end{definition}

Note that in the first item $q \restrict L_{\md(p)} \forces_{\poset(B,
\I \restrict B)} q(\md(p)) \leq_{\loc} p(\md(p))$ is shorthand for the
corresponding statement with checks introduced on the first
coordinates of $q(\md(p))$ and $p(\md(p))$.

Of course at every recursive step we can take the usual action (finding a
$V_\beta$ containing names for all possible $\varphi : \N
\rightarrow \finset{\N}$ in the generic extension) to ensure
that we in fact end up with a set of forcing conditions.

Now since in defining the order there was an existential quantifier
over $B \subseteq A \cap L_x$ it is not immediately clear that this
defines an order at all.  Reflexivity is clear (antisymmetry is false,
but as is usual we can look at the poset modulo the relation $p \leq
q$ and $q \leq p$), but transitivity is not clear.  It is conceivable
that $r \leq q$ and $q \leq p$ use different witnesses that can not be
unified to give $r \leq p$.  The following lemma shows that this does
not happen.

\begin{lemma}
  $\poset(L, \I)$ is a poset.
\end{lemma}

The next lemma shows that this sort of forcing can be used for an
iterated forcing argument.

\begin{lemma}[Completeness of Embeddings]
  For any $A \subseteq B \subseteq L$ we have that $\poset(A, \I
  \restrict A)$ completely embeds into $\poset(B, \I \restrict B)$,
  written $\poset(A, \I \restrict A) \leqc \poset(B, \I \restrict B)$.
\end{lemma}

We will in fact prove both these lemmas at the same time by recursion,
showing that an apparently stronger form of complete embedding is true
(we currently don't know whether this notion really is stronger).

\begin{lemma}[\protect{\cite[Lemma 4.4]{JB}}; similar to
  \protect{\cite[Lemma 1.1]{JB}}]
  For any $B \in \mathcal{I}$ and $A \subseteq B$,
  $\poset(B, \I \restrict B)$ is a partial order, and for any $p \in
  \poset(B, \I \restrict B)$ there exists a $\red(p,A,B) \in \poset(A,
  \I \restrict A)$ such that for any $D \in \mathcal{I}$ and $C
  \subseteq D$ such that $B \subseteq D$ and $B \cap C = A$ if $r
  \leq_{\poset(C, \I \restrict C)} \red(p,A,B)$ then $r$ and $p$ are
  compatible in $\poset(D, \I \restrict D)$.
\end{lemma}

\begin{proof}
  By induction on $\alpha$ we will show that
  \begin{enumerate}
    \item
      \label{enum:templ:transitive}
      for all templates $(L, \I)$ with $\Dp(L) \leq \alpha$,
      the poset $\poset(L, \I)$ is transitive,
    \item
      \label{enum:templ:subset}
      for all templates $(L, \I)$ with $\Dp(L) \leq \alpha$
      and $A \subseteq L$,
      \begin{equation*}
        \poset(A, \I \restrict A) \subseteq \poset(L, \I),
      \end{equation*}
    \item
      \label{enum:templ:constRed}
      for all templates $(L, \I)$ with $\Dp(L) \leq \alpha$,
      all $B \in \I$, all $A \subseteq B$, and all $p \in \poset(B, \I
      \restrict B)$ there exists an element $\red(p, \I \restrict B, A) \in
      \poset(A, \I \restrict A)$ such that $\dom(\red(p, \I \restrict
      B, A)) = \dom(p) \cap A$ and for all $x \in \dom(p) \cap A$ we
      have $\pi_0(\red(p, \I \restrict B, A)(x)) = \pi_0(p(x))$ (note:
      $\red(p, \I \restrict B, A)$ depends only on $p$, $\I \restrict
      B$ and $A$, \emph{not} on the part of the template outside of
      $\I \restrict B$),
    \item
      \label{enum:templ:RedProperty}
      for all templates $(L, \I)$ with $\Dp(L) \leq \alpha$ and
      all $D \in \I$, $B,C \subseteq D$, and $p \in \poset(B, \I
      \restrict B)$, that (write $A$ for $B \cap C$) for all $q \in
      \poset(C, \I \restrict C)$,
      \begin{equation*}
        q \leq_{\poset(C, \I \restrict C)} \red(p, \I \restrict B, A)
        \quad \Rightarrow \quad p \mathrel{\|_{\poset(D, \I \restrict
        D)}} q.
      \end{equation*}
  \end{enumerate}
  
  We first show \ref{enum:templ:transitive}.  Let $p,q,r \in \poset(L,
  \I)$ such that $r \leq q$ and $q \leq p$.  We then need to analyze
  cases depending on the orders of $\md(r), \md(q)$, and $\md(p)$.
  We'll show one; the others are analogous.  The main idea here is
  that a template is closed under finite unions so that we can unify
  the witness to the order.
  
  So assume $\md(r) = \md(q) = \md(p)$.  Then there is a $B^{p,q} \in
  \I$ with $B^{p,q} \subseteq L_{\md(p)}$ such that $q \restrict
  L_{\md(p)} \leq_{\poset(B^{p.q}, \I \restrict B^{p,q})} p \restrict
  L_{\md(p)}$ and $q \restrict L_{\md(p)} \forces_{\poset(B^{p,q}, \I
    \restrict B^{p,q})} q(\md(p)) \leq_{\loc} p(\md(p))$.  There is a
  similar set $B^{q,r}$ obtained from the order $r \leq q$.  Then set
  $B^{p,r} = B^{p,q} \cup B^{q,r}$, and note that by induction
  hypothesis (since $\md(p) \not \in B^{p,r}$, $\Dp(B^{p.r}) <
  \Dp(L)$) we have that $\poset(B^{p,q}, \I \restrict B^{p,q}),
  \poset(B^{q,r}, \I \restrict B^{q,r}) \leqcirc \poset(B^{p,r}, \I
  \restrict B^{p,r})$.  From this we get that $r \restrict L_{\md(p)}
  \leq_{\poset(B^{p,r}, \I \restrict B^{p,r})} p \restrict L_{\md(p)}$
  and $r \restrict L_{\md(p)} \forces_{\poset(B^{p,r}, \I \restrict
    B^{p,r})} r(\md(p)) \leq_{\loc} p(\md(p))$, which witnesses that
  $r \leq p$ as was to be shown.
  
  Now we show \ref{enum:templ:subset}.  Let $A \subset L$ and $p \in
  \poset(A, \I \restrict A)$.  We need to show that $p \in \poset(L,
  \I)$.  The main idea is to unfold one step of the construction that
  shows that $p \in \poset(A, \I \restrict A)$, and then use a similar
  construction step to build $p \in \poset(L, \I)$.
  
  $p \in \poset(A, \I \restrict A)$ means there is $B \in \I \restrict
  A$ such that $B \subseteq L_{\md(p)}$, $p \restrict L_{\md(p)} \in
  \poset(B, \I \restrict B)$, and $p\restrict L_{\md(p)}
  \forces_{\poset(B, \I \restrict B)} p(\md(p)) \in \loc$.
  
  $B \in \I \restrict A$ means that $B = B' \cap A$ for some $B' \in
  \I$.  Since $\md(p) \in A \setminus B$ we have $\md(p) \not \in B'$
  and we can (using the definition of template) assume that $B'
  \subseteq L_{\md(p)}$.  By the induction hypothesis we have
  $\poset(B, \I \restrict B) \leqcirc \poset(B', \I \restrict B')$,
  which gives us that $p \restrict L_{\md(p)} \in \poset(B', \I
  \restrict B')$ and $p \restrict L_{\md(p)} \forces_{\poset(B', \I
    \restrict B')} p(\md(p)) \in \loc$.  This in turn shows $p \in
  \poset(L, \I)$.
  
  Now we show \ref{enum:templ:constRed}.  Assume that $\red(p,
  I\restrict B, A)$ has already been defined for $B$ with $\Dp(B) <
  \alpha$.  So all that remains at this stage is to define it for $B =
  L$.  The main idea is to use the reduction already defined for $p
  \restrict L_{\md(p)}$, and for the reduction at $\md(p)$ (if it
  needs to be defined) use a reduction of $\pi_1(\md(p))$.

  Let $A \subseteq L$ and $p \in \poset(L, \I)$.  There exists
  $B^{\md(p)} \in \I$ such that $B^{\md(p)} \subseteq L_{\md(p)}$, $p
  \restrict L_{\md(p)} \in \poset(B^{\md(p)}, \I \restrict
  B^{\md(p)})$, and $p \restrict L_{\md(p)}
  \forces_{\poset(B^{\md(p)}, \I \restrict B^{\md(p)})} p(\md(p)) \in
  \loc$.  Set $A^{\md(p)} = A \cap B^{\md(p)}$.
  
  If $\md(p) \not \in A$, define $\red(p, \I, A) := \red(p \restrict
  L_{\md(p)}, \I \restrict B^{\md(p)}, A^{\md(p)})$.  Since
  $\Dp(B^{\md(p)}) < \alpha$ we have by induction hypothesis that
  $\poset(A^{\md(p)}, \I \restrict A^{\md(p)}) \leqcirc \poset(A, \I
  \restrict A)$, so that indeed $\red(p,\I,A) \in \poset(A, \I
  \restrict A)$.

  Otherwise, define $\red(p, \I, A) \restrict L_{\md(p)} := \red(p
  \restrict L_{\md(p)}, \I \restrict B^{\md(p)}, A^{\md(p)})$, and
  $\red(p, \I, A)(\md(p)) := (\sigma, \dot{\varphi})$, where $p(\md(p)) =
  (\sigma, \dot{\psi})$ and $\dot{\varphi}$ is the reduction of the
  $\poset(B^{\md(p)}, \I \restrict B^{\md(p)})$-name $\dot{\psi}$ to
  $\poset(A^{\md(p)}, \I \restrict A^{\md(p)})$ defined as follows:

  We know that $\poset(A^{\md(p)}, \I \restrict A^{\md(p)}) \leqcirc
  \poset(B^{\md(p)}, \I \restrict B^{\md(p)})$.  Consider the complete
  Boolean algebras $\mathbb{A}$ and $\mathbb{B}$ which are the regular
  open algebras of these respective posets.
  
  Note that the $\mathbb{B}$-name $\dot{\psi}$ is completely
  determined by the Boolean values $\bv{\check{\sigma} \subseteq
    \dot{\psi}}$ where $\sigma \in \finseq{\finset{\N}}$.  $p
  \restrict L_{\md(p)} \forces \dot{\psi} \in
  \funcN{(\finsetd{\N}{\lh(\check{\sigma})})}$ is equivalent
  to
  \[
    \forall \sigma' \in \finseq{\finset{\N}} \, \big( \exists i \in
    \dom(\sigma')\ |\sigma'(i)| > \lh(\sigma) \big) \rightarrow p \restrict
    L_{\md(p)} \wedge \bv{\check{\sigma'} \subseteq \dot{\psi}} = 0.
  \]

  For $\sigma \in \finseq{\finset{\N}}$,
  define $a_\sigma$ to
  be the projection of $\bv{\check{\sigma} \subseteq \dot{\psi}}$ to
  $\mathbb{A}$ (this projection is $\bigwedge \{ a \in \mathbb{A} \mid
  \bv{\check{\sigma} \subseteq \dot{\psi}} \leq a \}$).  Enumerate
  $\finset{\N}$ as $\langle c_k \mid k \in \N \rangle$, and
  recursively on the length of $\sigma$ define $\bv{\check{\sigma}
  \subseteq \dot{\varphi}}$ to be
  \begin{equation*}
    \red(p, \I \restrict B, A) \restrict L_{\md(p)} \wedge \bv{(\sigma
    \restrict (\lh(\sigma) - 1))\check{\vphantom{A}} \subseteq
    \dot{\varphi}} \wedge (a_\sigma \setminus \bigvee_{k < l}
    a_{(\sigma\restrict \lh(\sigma) -1) \hat{\phantom{a}} c_k}),
  \end{equation*}
  where $l$ is such that $\sigma(\lh(\sigma)) = c_l$.

  Note that $\bv{ \check{\sigma_1} \subseteq \dot{\varphi}} \wedge
  \bv{\check{\sigma_2} \subseteq \dot{\varphi}} = 0$ if there is an
  $i$ such that $\sigma_1(i) \neq \sigma_2(i)$ or if there is an $i$
  such that $| \sigma_1(i) | > \lh(\sigma)$, since this is true for
  $\dot{\psi}$.  Also $\langle \bv{(\sigma \hat{\phantom{a}}
  c_k)\check{\vphantom{A}} \subseteq \dot{\varphi}} \mid k \in \N
  \rangle$ is a maximal
  antichain below $\bv{\check{\sigma} \subseteq \dot{\varphi}}$.

  Now we show that the reduction works as advertised, i.e. show
  \ref{enum:templ:RedProperty}.  This is just a matter of using the
  induction hypothesis and the properties of templates which allow us
  to construct the right sets to consider conditions in and over.

  So let $D \in \I$, $B,C \subseteq D$, and $p \in \poset(B, \I
  \restrict B)$.  Write $A$ for $B \cap C$.  Let $B^{\md(p)}$ and
  $A^{\md(p)}$ be as in the definition of $\red(p, \I \restrict B, A)$
  ($A^{\md(p)} = A$ if $\md(p) \not \in A$).  Let $q \in \poset(C, \I
  \restrict C)$ be such that $q \leq_{\poset(C, \I \restrict C)}
  \red(p, \I \restrict B, A)$.  We have to show that $q$ and $p$ are
  compatible in $\poset(D, \I \restrict D)$.
  
  \underline{{If\footnotemark} $\md(p) \not \in A$}, \footnotetext{We
    have underlined this and the next case to make them easier to
    find.} then also $\md(p) \not \in C$.  We will use the induction
  hypothesis to find an $r$ (with suitable domain) that witnesses that
  $q \restrict L_{\md(p)}$ and $\red(p, \I \restrict B, A) \restrict
  L_{\md(p)}$ are compatible.  Then we show that the element $r \cup
  \{(\md(p), p(\md(p)))\} \cup q \restrict \{ y \in L \mid \md(p) < y
  \}$ shows that $p$ and $q$ are compatible.

  Using the definition for $q \in \poset(C, \I \restrict C)$, we can
  find a set $C^q \in \I$ with $C^q \subseteq L_{\min\{x \in L \mid x \in
  \dom(q) \wedge x > \md(p)\}}$ (if there is an $x \in \dom(q)$ with
  $x > \md(p)$, otherwise let $C^q = C$) such that $q \restrict
  L_{\md(p)} \in \poset(C^q, \I \restrict C^q)$.  Since $C^q$ might be
  somewhat too big for our purposes, we unfold the definition of $q$ one
  step further and then refold it to a smaller set.
  
  Since $q \restrict L_{\md(p)} \in \poset(C^q, \I \restrict C^q)$,
  using the definition there exists ${C^q}' \in \I \restrict C^q$ such
  that ${C^q}' \subseteq L_{\md(q\restrict L_{\md(p)})}$, $q
  \restrict L_{\md(q \restrict L_{\md(p)})} \in \poset({C^q}', \I
  \restrict {C^q}')$ and $q \restrict L_{\md(q \restrict L_{\md(p)}}
  \forces_{\poset({C^q}', \I \restrict {C^q}')} q(\md(q\restrict
  L_{\md(p)})) \in \loc$.  Now there exists $C^{\md(p)} \subseteq
  L_{\md(p)}$ with ${C^q}' \subseteq C^{\md(p)}$, $A^{\md(p)}
  \subseteq C^{\md(p)}$, and $\md(q \restrict L_{\md(p)}) \in
  C^{\md(p)}$.  Then $q \in \poset(C^{\md(p)}, \I \restrict
  C^{\md(p)})$.

  Setting $D^{\md(p)} = C^{\md(p)} \cup B^{\md(p)}$, we can find by
  induction hypothesis an $r \in \poset(D^{\md(p)}, \I \restrict
  D^{\md(p)})$ such that $r \leq_{\poset(D^{\md(p)}, \I \restrict
  D^{\md(p)})} q \restrict L_{\md(p)}, p \restrict L_{\md(p))}$.

  Since $B^{\md(p)} \subseteq D^{\md(p)}$, we have $\poset(B^{\md(p)},
  \I \restrict B^{\md(p)}) \leqcirc \poset(D^{\md(p)}, \I \restrict
  D^{\md(p)})$.  Then since $r \leq p \restrict L_{\md(p)}$, we have
  that $r \forces_{\poset(D^{\md(p)}, \I \restrict D^{\md(p)})}
  p(\md(p)) \in \loc$.

  Now if $\{ y \in L \mid y \in \dom(q) \wedge \md(p) < y \}$ is
  empty, then from the above we find, by the definitions,
  that $r \cup \{(\md(p), p(\md(p)))\} \in \poset(D, \I \restrict D)$
  and $r \leq p, q$.  Otherwise we first have to enlarge $D^{\md(p)}$
  to a set containing $C^q$ so that we can from there on use the way
  $q$ was constructed to construct $r \cup \{(\md(p), p(\md(p)))\}
  \cup q \restrict \{y \in L \mid y \in \dom(q) \wedge \md(p) < y\}$.

  Since $\md(p) < \min\{ y \in L \mid y \in \dom(q) \wedge \md(p) <
  y\}$, there exists $E \in \I$ with $E \subseteq L_{\{ y \in L
  \mid y \in \dom(q) \wedge \md(p) < y\}}$ and $\md(p) \in E$.  Set
  $E' = E \cup C^q \cup D^{\md(p)}$.  Then with the above we get $r
  \cup \{(\md(p), p(\md(p)))\} \in \poset(E', \I \restrict E')$.  By
  the induction hypothesis $\poset(C^q, \I \restrict C^q) \leqcirc
  \poset(E', \I \restrict E')$, and using the definition of order we
  have $r \cup \{(\md(p), p(\md(p)))\} \leq_{\poset(E', \I \restrict
  E')} q \restrict q \restrict L_{\md(p)}$ and $r \cup \{(\md(p),
  p(\md(p)))\} \leq_{\poset(E', \I \restrict E')} p$.

  We can now add the next element from the domain of $q$ to $r \cup
  \{(\md(p),p(\md(p)))\}$.  After iteratively adding the elements from
  the domain of $q$, we find $r \cup \{(\md(p),$ $p(\md(p)))\} \cup q
  \restrict \{y \in L \mid \md(p) < x \} \leq_{\poset(\bar{C}, \I
  \restrict \bar{C})} q, p$ for a $\bar{C}$ containing $C$.  If when adding
  $\md(q)$, we take the domain set to be all of $D$, then we get $r \cup
  \{(\md(p),p(\md(p)))\} \cup q \restrict \{y \in L \mid \md(p) < x \}
  \leq_{\poset(D, \I \restrict D)} q, p$ as desired.

  \underline{If $\md(p) \in A$}, then $q \leq_{\poset(C, \I \restrict
  C)} \red(p, \I \restrict B, A)$ is witnessed by $C^{\md(p)} \in \I$
  such that $A^{\md(p)} \subseteq C^{\md(p)} \subseteq L_{\md(p)}$, $q
  \restrict L_{\md(p)} \leq_{\poset(C^{\md(p)}, \I \restrict
  C^{\md(p)})} \red(p, \I \restrict B, A) \restrict L_{\md(p)}$, and
  $q \restrict L_{\md(p)} \forces_{\poset(C^{\md(p)}, \I \restrict
  C^{\md(p)})} q(\md(p)) \leq_\loc \red(p, \I \restrict B,
  A)(\md(p))$ and from there by possibly extending $q$.

  Since $\red(p, \I \restrict B, a) \restrict L_{\md(p)} = \red(p
  \restrict L_{\md(p)}, \I \restrict B^{\md(p)}, A^{\md(p)})$ and
  $C^{\md(p)} \cap B^{\md(p)} = A^{\md(p)}$, we can apply the induction
  hypothesis with $D^{\md(p)} = B^{\md(p)} \cup C^{\md(p)}$ to obtain
  $r^{\md(p)} \in \poset(D^{\md(p)}, \I \restrict D^{\md(p)})$ such
  that $r^{\md(p)} \leq_{\poset(D^{\md(p)}, \I \restrict D^{\md(p)})}
  q \restrict L_{\md(p)}, p \restrict L_{\md(p)}$.

  Write $q(\md(p)) = (\sigma_q, \dot{\varphi}_q)$ and remember that
  $\red(p, \I \restrict B, A)(\md(p)) = (\sigma_p, \dot{\varphi})$.
  Then
  \begin{equation}
    \tag{*}
    \label{eqn:templ:star}
    r^{\md(p)} \forces_{\poset(D^{\md(p)}, \I \restrict D^{\md(p)})}
    (\check{\sigma_q}, \dot{\varphi}_q) \leq_\loc (\check{\sigma_p},
    \dot{\varphi}).
  \end{equation}
  Therefore $\sigma_p \subseteq \sigma_q$ and $r^{\md(p)} \forces
  (\forall i \dot{\varphi}(i) \subseteq \dot{\varphi}_q(i)) \wedge
  \forall i (\lh(\check{\sigma_p}) < i \leq \lh(\check{\sigma_q})
  \rightarrow \dot{\varphi}(i) \subseteq \check{\sigma_q}(i))$.

  It is enough to find a condition $d^{\md(p)}$ in $\poset(D^{\md(p)},
  \I \restrict D^{\md(p)})$ such that $d^{\md(p)} \leq r^{\md(p)}$ and
  $s^{\md(p)} \forces (\check{\sigma_q}, \dot{\varphi}_q) \leq_\loc
  (\sigma_p, \dot{\psi})$.  Then by extending $d^{\md(p)}$ first by
  $(\md(p),$ $ (\sigma_q, \dot{\varphi}_q))$ and then in the same way $q$
  is extended we find the condition that shows $p$ and $q$ are
  compatible.

  Write $\Sigma = \{\sigma \in \finseq{\finset{\N}} \mid \forall
  i (\lh(\sigma) \leq i < \lh(\sigma_q) \rightarrow \sigma(i)
  \subseteq \sigma_q(i))\}$, then by \eqref{eqn:templ:star}
  \begin{equation*}
    r^{\md(p)} \leq \bigvee_{\sigma \in \Sigma} \bv{ \sigma \subseteq
      \dot{\varphi}}.
  \end{equation*}

  Therefore
  \begin{equation*}
    \red(r^{\md(p)}, \I \restrict D^{\md(p)}, A^{\md(p)}) \leq
    \bigvee_{\sigma \in \Sigma} \bv{\sigma \subseteq \dot{\varphi}}.
  \end{equation*}
  
  The right hand side of this inequality is a reduction of
  $\bigvee_{\sigma \in \Sigma} \bv{\sigma \subseteq \dot{\psi}}$, and
  therefore there is a $b^{\md(p)} \in \poset(B^{\md(p)}, \I \restrict
  B^{\md(p)})$ such that
  \begin{equation*}
    b^{\md(p)} \leq_{\poset(B^{\md(p)}, \I \restrict B^{\md(p)})} r^{\md(p)},
    \bigvee_{\sigma \in \Sigma} \bv{ \sigma \subseteq \dot{\psi}}
  \end{equation*}
  By induction hypothesis this gives a condition $d^{\md(p)} \in
  \poset(D^{\md(p)}, \I \restrict D^{\md(p)})$ such that $d^{\md(p)}
  \leq^{\poset(D^{\md(p)}, \I \restrict D^{\md(p)})} \bigvee_{\sigma
  \in \Sigma} \bv{\sigma \subseteq \dot{\psi}}, r^{\md(p)}$.

  Therefore $d^{\md(p)} \forces \forall i (\lh(\check{\sigma_p}) \leq i <
  \lh(\check{\sigma_q}) \rightarrow \dot{\psi}(i) \subseteq \sigma_q(i))$ and
  $d^{\md(p)} \forces \forall i \dot{\psi}(i) \subseteq
  \dot{\varphi}_q(i)$; that is $d^{\md(p)} \forces (\check{\sigma_q},
  \dot{\varphi}) \leq_\loc (\check{\sigma_p}, \dot{\psi})$.  Which is as
  was to be shown.
\end{proof}

We now show that the properties of this poset that establish that the forcing
preserves cardinals and that any real in
the forcing extension is already added by a small forcing completely
embedded in it.

\begin{definition}
  A poset $\mathbb{P}$ satisfies \emph{Knaster's condition} if every
  uncountable subset of it has an uncountable subset of pairwise
  compatible elements.
\end{definition}

A poset satisfying Knaster's condition is clearly c.c.c.

\begin{lemma}[\protect{\cite[Lemma 4.5]{JB}}]
  For all templates $(L, \I)$, $\poset (L, \I)$ satisfies Knaster's
  condition.
\end{lemma}

\begin{proof}
  Let $\mathcal{A} \subseteq \mathbb{P}(L, \I)$ be an uncountable set.

  We first show by induction on the rank of $L$ that the set of $p
  \in \mathbb{P}(L, \I)$ such that for all $x \in \dom(p)$ there are a
  $B \subseteq L_x$ and $\tau \in {}^{2 \lh(\pi_0(p(x)))}({}^{\leq
    \lh(\pi_0(p(x)))}[\N])$ such that $p \restrict L_x
  \forces_{\mathbb{P}(B, \I \restrict B)} \tau \subseteq \pi_1(p(x))$
  is dense.  These are the conditions where all the second projections
  at points in the domain are determined up to twice the length of the
  first coordinate.
  
  For $p \in \mathbb{P}(L,\I)$ there is a $B \subseteq L_{\md(p)}$
  such that $p \restrict L_{\md(p)} \forces_{\mathbb{P}(B, \I
    \restrict B)} \pi_1(p(x)) \in \funcN{({}^{\leq
    \lh(\pi_0(p(\md(p))))}[\N])}$.  So there is a $q \in
  \mathbb{P}(B, \I \restrict B)$ and $\tau \in {}^{2 \lh(\pi_0(p(x)))}({}^{\leq
    \lh(\pi_0(p(x)))}[\N])$ with $q
  \leq_{\mathbb{P}(B, \I \restrict B)} p \restrict L_{\md(p)}$ such
  that $q \forces_{\mathbb{P}(B, \I \restrict B)} \check{\tau}
  \subseteq \pi_1(p(\md(p)))$.  Then by the induction hypothesis there
  is $q' \leq_{\mathbb{P}(B, \I \restrict B)} q$ satisfying the
  requirement.  So $q' \cup \{(\md(p), p(\md(p)))\}$ is as desired.

  Let $\mathcal{A}' \subseteq \mathbb{P}(L,\I)$ be uncountable, have
  all elements in this dense set, and contain an element below every
  element of $\mathcal{A}$.  It is enough to show that $\mathcal{A}'$
  has an uncountable subset with all conditions compatible.

  By the $\Delta$-system lemma we can find an uncountable subset
  $\mathcal{A}''$ of $\mathcal{A}'$ such that $\{ \dom(p) \mid p \in
  \mathcal{A}''\}$ is a $\Delta$ system with root $r$.

  For each $x \in r$ there are countably many choices for
  $\pi_0(p(x))$, and then countably many choices for $\pi_1(p(x))
  \restrict 2 \lh(\pi_0(p(x))) \in {}^{2 \lh(\pi_0(p(x)))}({}^{\leq
    \lh(\pi_0(p(x)))}[\N])$.  So we can find an
  uncountable subset $\mathcal{A}'''$ of $\mathcal{A}''$ where all
  these choices are the same.

  Now for all $p_0, p_1 \in \mathcal{A}'''$ we see they are compatible
  by constructing a $q$ below both of them.  $q$ will have $\dom(q) =
  \dom(p_0) \cup \dom(p_1)$ and is constructed by recursion along its
  domain using the way $p_0$ and $p_1$ are constructed, and if the
  next value $q(x)$ to be constructed has $x \in r$ take $q(x)$ to be
  $(\tau_{p_0 \restrict (L_x \cup \{x\})}, \pi_1(p_0(x)) \cup \pi_1(p_1(x)))$.
\end{proof}

\begin{lemma}[analogue of \protect{\cite[Lemma 1.6]{JB}}]
  \label{lem:realsEtcInSmaller}
  \label{reduce to countable}
  For all templates $(L, \I)$, every $\dot{r}$ a
  $\mathbb{P}(L,\I)$ name for a real, every $\dot{\varphi}$ a
  $\mathbb{P}(L, \I)$ name for an element of $\funcN{(^{\leq n}[\N])}$ with $n \in \N$, and every $p \in \mathbb{P}(L,
  \I)$, there are countable sets $A_r, A_\varphi, A_p \subseteq L$ such
  that $\dot{r}$ is a $\mathbb{P}(A_r, \I \restrict A_r)$ name for a
  real, $\dot{\varphi}$ is a $\mathbb{P}(A_\varphi, \I \restrict
  A_\varphi)$ name for an element of $\funcN{(^{\leq n}[\N])}$, and
  $p \in \mathbb{P}(A_p, \I \restrict A_p)$.
\end{lemma}

\begin{proof}
  This is easy to see by induction on the rank of $L$ and the facts
  that both $\dot{r}$ and $\dot{\varphi}$ are determined by countably
  many conditions in a $\mathbb{P}(B, \I \restrict B)$ for $B
  \subsetneq L$, and $p(\md(p))$ is a pair $(\sigma, \dot{\varphi})$.
\end{proof}
  
The following lemma, which is immediate from the definition of the
posets and the completeness of embeddings lemma, will be used later to
see that we indeed cofinally often take an extension by localization
forcing.

\begin{lemma}[Embedding Localization Forcing, analogue of
  \protect{\cite[Cor. 1.5]{BJ}}]
  \label{lem:embeddingLocForc}
  For any template $(L,\I)$, $x \in L$, and $A \subseteq L_x$ such
  that $A \in \I$, we have that $\mathbb{P}(A, \I \restrict A)
  \leqc \mathbb{P}(A \cup\{x\}, \I \restrict (A \cup \{x\}) \cong
  \mathbb{P}(A, \I \restrict A) * \loc \leqc \mathbb{P}(L,
  \I)$. \hfill $\qed$
\end{lemma}

The following notion is very important to us; it will be used below to
see that posets defined using different templates are isomorphic.

\begin{definition}
  For $(L,\I)$ and $(L, \I')$ two templates, we say $\I$ is an
  \emph{innocuous extension} of $\I'$ if $\I' \subseteq \I$ and for
  all $B \in \I$ with $B \subseteq L_x$ and all countable $A \subseteq
  B$ there is a $C \in \I'$ with $C \subseteq L_x$ such that $A
  \subseteq C$.
\end{definition}

For this notion we have the following.

\begin{lemma}[Innocuous extensions, \protect{\cite[Lemma 1.7]{BJ}}]
  If $(L, \I)$ and $(L, \I')$ are two templates with $\I$ an
  innocuous extension of $\I'$, then $\mathbb{P}(L,\I) \cong
  \mathbb{P}(L,\I')$.\hfill $\qed$
\end{lemma}

\section*{The Template}

The template we describe in this section is due to Shelah \cite{Sh}; we
work with the description of it in Brendle \cite{JB}.

The template we use is $(L, \I) = (L(\mu,\lambda),
\I(\mu,\lambda))$ with $\mu$ and $\lambda \geq \aleph_1$ two cardinals.

Let $\lambda^*$ be a disjoint copy of $\lambda$ with the reverse ordering,
and $\{ S^\alpha \mid \alpha < \omega_1 \}$ be a partition of
$\lambda^*$ such that each $S^\alpha$ is coinitial in $\lambda^*$.  We
call the elements of $\lambda$ \emph{positive} and those of
$\lambda^*$ \emph{negative}.  We order the set $\lambda^* \cup
\lambda$ by having all elements of $\lambda^*$ be smaller than all
elements of $\lambda$ and having the usual order between elements of
$\lambda^*$ or elements of $\lambda$.  We set $L := \{ x \in \finseq{\mu \cup
\lambda^* \cup \lambda} \mid x(0) \in \mu \ \wedge \ x(n)
\in \lambda^* \cup \lambda \text{ for all } n \text{ such that } 1
\leq n < \lh(x) \}$, and we define $x < y$ if
\begin{itemize}
  \item $x \subset y$ and $y(\lh(x)) \in \lambda$ (elements get bigger
    by extending in $\lambda$), or
  \item $y \subset x$ and $x(\lh(y)) \in \lambda^*$ (elements get
    smaller by extending in $\lambda^*$), or
  \item if $n := \min \{ m \mid x(m) \neq y(m) \}$ then $x(n) < y(n)$.
\end{itemize}

We say $x \in L$ is \emph{relevant} if
\begin{itemize}
  \item $\lh(x) \geq 3$ and is odd, and
  \item $x(n)$ is negative for odd $n$ and positive for even $n$, and
  \item $x(\lh(x)-1) < \omega_1$, and
  \item if $n < m$ are such that $x(n), x(m) < \omega_1$ and are even,
    then there are $\beta < \alpha$ such that $x(n-1) \in S^\alpha$
    and $x(m-1) \in S^\beta$.
\end{itemize}

For a relevant $x$ we define $J_x$ to be the interval
$[x\upharpoonright (\lh(x)-1), x)$. 

Now we define $\mathcal{I}$ to be generated by taking finite unions of
singletons of members of $L$, $L_\alpha$ for $\alpha \in \mu \cup
\{\mu\}$, and $J_x$ for relevant $x$.

\begin{lemma}[\protect{\cite[Lemma 3.2]{JB}}]
  $(L, \mathcal{I})$ is a template.
\end{lemma}    

\section*{The Result}

Now we are ready to put all our work together to prove the following
theorem which clearly implies Theorem \ref{thm:template}.

\begin{theorem}
  Let $\lambda > \mu > \aleph_1$ be regular cardinals with $\lambda =
  \lambda^\omega$.  Then $\mathbb{P}(L, \I)
  \forces ``\add(\Null) = \cof(\Null) = \mu \text{ and } \mathfrak{a}_g
  = \lambda = 2^{\aleph_0}$''.
\end{theorem}

Since $\mu$ (identified with sequences of length $1$) is cofinal in
$L$, and $L_\mu \in I$, the slaloms added at the
coordinates associated to $\mu$ localize all reals in the generic
extension as follows.  Let $G$ be $\mathbb{P}(L,\I)$ generic and $r$ a real in
$V[G]$.  Then by Lemma \ref{lem:realsEtcInSmaller} we find a countable
$A \subseteq L$ such that $r \in V[G \cap \mathbb{P}(A, \I \restrict
A)]$.  There is an $\alpha \in \mu$ such that $A \subseteq L_\alpha$,
so that the real $r$ is in $V[G \cap \mathbb{P}(L_\alpha, \I \restrict
L_\alpha)]$.  Then Lemma \ref{lem:embeddingLocForc} shows that $r$ is
localized by the slalom added at coordinate $\alpha$.

Since for a family of reals $\{r_\gamma \mid \gamma < \beta\}$ with
$\beta < \mu$ we can do similar reasoning, we see that these slaloms
actually capture all families of reals of size less than $\mu$ in the
generic extension as follows.  Code the family $\{ r_\gamma \mid
\gamma < \beta \}$ into a function $\dot{F} : \beta \times \omega
\rightarrow \N$.  This name is determined by $\beta \times \omega$
many antichains $\{B_{\theta,n} \mid (\theta,n) \in \beta \times
\omega\}$, where the antichain $B_{\theta,n}$ consists of conditions
deciding the value of $\dot{F}(\theta,n)$.  By the c.c.c. all these
antichains are countable; enumerate for each $\theta$ and $n$ the
antichain $A_{\theta,n}$ by $\{p_{\theta,n,k} \mid k \in \N\}$.  For
each $p_{\theta,n,k}$ there exists a countable set $A_{\theta,n,k}$
such that $p_{\theta,n,k} \in \poset(A_{\theta,n,k}, \I \restrict
A_{\theta,n,k})$ (by Lemma \ref{reduce to countable}).  Since these
are $< \mu$ many countable sets, there is an $\alpha < \mu$ such that
all are contained in $L_\alpha$.  This means $\dot{F}$, and therefore
the family $\{r_\gamma \mid \gamma < \beta\}$, are in
$V^{\poset(L_{\alpha}, \I \restrict L_{\alpha})}$.  The slalom added
at coordinate $\alpha$ then localizes all of them.

These two arguments together show that $\cof(\Null) = \add(\Null) =
\mu$ in the generic extension.

By counting names for reals we see that in the forcing extension
$2^{\aleph_0} \leq \lambda$ as follows.  A nice name for a real is
determined by countably many maximal antichains and countably many
reals.  This gives us
\[
(\#\text{maximal antichains})^{\aleph_0} \cdot 2^{\aleph_0}
\]
as an upper bound for the number of nice names for reals.  We have to
show that the number of maximal antichains is bounded by $\lambda$.
Because the forcing satisfies Knaster's condition, all antichains are
countable.  This gives us
\[
|\poset(L,\I)|^{\aleph_0 \times \aleph_0} \cdot 2^{\aleph_0}
\]
as an upper bound for the number of nice names for reals.  We show by
induction on the rank of $L$ that $|\poset(L, \I)| \leq \lambda$.
Assume that for all $B$ of rank less than $L$ we have that $|\poset(B,
\I \restrict B)| \leq \lambda$.  Then $|\poset(L, \I)| \leq \Sigma_{B
  \in \I, B\neq L} |\poset(B, \I \restrict B)|^{\aleph_0} \cdot
(2^{\aleph_0} \cdot \aleph_0) = \Sigma_{B \in \I, B\neq L}
\lambda^{\aleph_0} \cdot (2^{\aleph_0} \cdot \aleph_0) = |\I| \cdot \lambda =
\lambda$ (using the induction hypothesis, the fact that
$\lambda^{\aleph_0} = \lambda$, and the continuum hypothesis).

Now we will show by an isomorphism-of-names argument that in the
forcing extension no cofinitary group of cardinality less than
$2^{\aleph_0}$ and bigger than or equal to $\mu$ can exist (remember
that $\mathfrak{a}_g \geq \non(\mathcal{M})$, as was mentioned above,
and that $\add(\mathcal{N}) \leq \non(\mathcal{M}) \leq
\cof(\mathcal{N})$) which completes the proof.

Let $\dot{G}$ be a name for a cofinitary group of size $< \lambda$ and
$\geq \mu$, say of size $\kappa$.  We can find names $\dot{g}_\alpha$
($\alpha < \kappa$) such that in the forcing extension, $\{
\dot{g}_\alpha \mid \alpha < \kappa \}$ is this cofinitary group.  For
\looseness=-1
each $\dot{g}_\alpha$ we find maximal antichains $\{p^\alpha_{n,k,i}
\mid k,i \in \N\}$ (for each $n \in \N$) in $\mathbb{P}(L, \I)$ such
that $p^\alpha_{n,k,i} \forces \dot{g}_\alpha(n) = k$.  Let $B^\alpha
= \cup \{ \dom(p^\alpha_{n,k,i}) \mid n,k,i \in \N\}$.  Each
$B^\alpha$ is a countable set of finite sequences from $\lambda^* \cup
\lambda$.  If we close it under initial segments with respect to
$\subseteq$ (that is not with respect to the order given on $L$) we
can assume it to be a tree (in the following, ``tree'' will always
mean tree with respect to $\subseteq$).

We will analyze countable subtrees of $L$ for a moment to see how we
can use these trees.

So let $T_0, T_1 \subseteq L$ be countable subtrees of $L$.  We say
$T_0 \cong T_1$ if there is a bijection $\varphi: T_0 \rightarrow T_1$
such that
\begin{itemize}
  \item $\varphi$ respects the order of $L$: if $x < y$ then
    $\varphi(x) < \varphi(y)$;
  \item $\varphi$ is a tree map: $\varphi(x \restrict n) =
    \varphi(x) \upharpoonright n$ and $\lh(\varphi(x)) = \lh(x)$;
  \item $\varphi$ respects the structure on the trees used to determine
    relevance:  $\varphi(x)(n)$ is positive, less than $\omega_1$, or
    a member of $S^\alpha$ iff $x(n)$ is. (One could choose to have a
    somewhat less restrictive requirement here, but this one suffices
    for our purposes.)
\end{itemize}

If $T_0 \cong T_1$ as witnessed by $\varphi$, then $\varphi$ induces an
isomorphism of $\mathcal{I} \restrict T_0$ with $\mathcal{I} \restrict
T_1$.

The number of types of trees can now be determined as follows.  The
number of trees up to equivalence for maps satisfying the first two
items is less than the number of countable subtrees of $\finseq{\omega_1^*
\cup \omega_1}$.  This gives us an upper bound of
$\aleph_1^{\aleph_0} = 2^{\aleph_0}$.  For each tree the information
for the last item can be encoded in a map from the set of nodes, size
$\aleph_0$, to $\omega_1 \times \{0,1\} \times \{0,1\}$ ($\omega_1$
for the labels of the $S^\alpha$ and the two $\{0,1\}$ as labels for
positive vs. negative, and for positive nodes bigger vs.  smaller than
$\omega_1$).  There is a total of $\aleph_1^{\aleph_0} = 2^{\aleph_0}$
of these maps, which gives us a grand total of $2^{\aleph_0} \cdot
2^{\aleph_0} = 2^{\aleph_0}$ isomorphism types of trees.

If we have an isomorphism of linear orders that induces an isomorphism
on the associated templates, then this isomorphism induces an
isomorphism of the associated posets.  So from the above we can
conclude that if $T_0 \cong T_1$ then $\poset \restrict T_0 \cong
\poset \restrict T_1$.

Brendle in the proof of Theorem 3.3, \cite[pp. 21--23]{JB}, shows how
in the set $\{B^\alpha \mid \alpha < \kappa\}$ we can find an
$\omega_1$ size subset, which after renumbering we can assume to be
$\{B^\alpha \mid \alpha < \omega_1\}$, and a $B^\kappa$ such that:
\begin{itemize}
\item there is a coherent set of maps $\phi_{\alpha,\beta}$ ($\alpha,
  \beta < \omega_1$) and $\phi_{\alpha, \kappa}$ ($\alpha < \omega_1$)
  such that the $\phi_{\alpha, \beta}: B^\alpha \rightarrow B^\beta$
  are isomorphisms of trees, and $\phi_{\alpha, \kappa}: B^\alpha
  \rightarrow B^\kappa$ is order preserving;
\item $\phi_{\alpha, \beta}(p^\alpha_{n,k,i}) = p^\beta_{n,k,i}$
  (names map to names);
\item for any $\beta < \kappa$, $\I \restrict B^\kappa \cup B^\beta$
  is an innocuous extension of the image of $\I \restrict B^\alpha \cup
  B^\beta$ for some $\alpha < \omega_1$ (the image under the mapping
  induced by the mappings $\phi$).
\end{itemize}

In fact it is clear from his construction that for any $\beta_0,
\ldots, \beta_i < \kappa$, $\I \restrict B^\kappa \cup (\bigcup_{j
  \leq i} B^{\beta_j})$ is an innocuous extension of the image of $\I
\restrict B^\alpha \cup (\bigcup_{j \leq i} B^{\beta_j})$ for some
$\alpha < \omega_1$.

If we define $\dot{g}_\kappa$ to be the name for a bijection by
$p^\kappa_{n,k,i} = \phi_{\alpha, \kappa}(p^\alpha_{n,k,i})$, we get a
name such that in the generic extension $\{\dot{g}_\alpha \mid \alpha
< \kappa\} \cup \{\dot{g}_\kappa\}$ is a cofinitary group that
properly includes $\{\dot{g}_\alpha \mid \alpha < \kappa\}$, so
the group we started with was not maximal: Let $\mathcal{G}$ be
$\mathbb{P}(L,\I)$ generic, and in the generic extension, $w(x)
\in W_{\dot{G}}$.  We need to see that $w(\dot{g}_\kappa)$ is
cofinitary.  Let $\dot{g}_{\beta_j}$ with $j < i$ be the elements of
$\dot{G}$ appearing in $w$.  Then $V[\mathcal{G}] \models
``w(\dot{g}_\kappa)$ is cofinitary'', iff $V[\mathcal{G} \cap
\mathbb{P}(B^\kappa \cup (\bigcup_{j < i} B^{\beta_j}), \I \restrict
B^\kappa \cup (\bigcup_{j < i} B^{\beta_j})) \models
``w(\dot{g}_\kappa)$ is cofinitary''.  But by the above
$\mathbb{P}(B^\kappa \cup (\bigcup_{j < i} B^{\beta_j}), \I \restrict
B^\kappa \cup (\bigcup_{j < i} B^{\beta_j})) \cong \mathbb{P}(B^\alpha
\cup (\bigcup_{j < i} B^{\beta_j}), \I \restrict B^\alpha \cup
(\bigcup_{j < i} B^{\beta_j}))$ for some $\alpha < \omega_1$ fixing
the $B^{\beta_j}$ (therefore fixing the $\dot{g}_{\beta_j}$ and
mapping $\dot{g}_\kappa$ to $\dot{g}_\alpha$).  This means that
$V[\mathcal{G} \cap \mathbb{P}(B^\kappa \cup (\bigcup_{j < i}
B^{\beta_j}), \I \restrict B^\kappa \cup (\bigcup_{j < i}
B^{\beta_j}))] \models ``w(\dot{g}_\kappa)$ is cofinitary'' if
$V[\mathcal{G} \cap \mathbb{P}(B^\alpha \cup (\bigcup_{j < i}
B^{\beta_j}), \I \restrict B^\alpha \cup (\bigcup_{j < i}
B^{\beta_j}))] \models ``w(\dot{g}_\alpha)$ is cofinitary'', which is
true since $\dot{G}$ is forced to be a cofinitary group.


\section{Definability}

\subsection{There Does Not Exist a $K_\sigma$ Maximal Cofinitary
  Group}

A set is $K_\sigma$ if it is a countable union of compact sets; every
$K_\sigma$ set is eventually bounded in the following sense.

\begin{definitions}
\item
  We write $\forall^* n \, \varphi(n)$ if for all but finitely many $n\in
  \N$, $\varphi(n)$.
  
\item For $f,g \in \BS$, $f$ is \emph{eventually bounded} by $g$,
  written $f <^* g$, if $\forall^*n\, f(n) < g(n)$ (if there exists $k
  \in \N$ such that for all $l > k$, $f(l) < g(l)$.
\item
  A set $S \subseteq \BS$ is \emph{eventually bounded} if there exists
  $f \in \BS$ such that for all $g \in S$, $g <^* f$.
\end{definitions}

\begin{theorem}
  If $G$ is a cofinitary group that is eventually bounded, then $G$
  is not maximal.
\end{theorem}

The basic idea in this proof is to use the bound to get an interval
partition which can be used similarly to how the orbits were used in
the proof of Theorem \ref{thm:MCGNotCountOrb}.

\begin{proof}
  Let $G$ be a cofinitary group that is eventually bounded.  This
  means $G$ is contained in a set of the form $\{g \in \BS \mid g <^*
  f\}$ where we can assume $f: \N \rightarrow \N$ is a strictly
  increasing function with $f(0)>0$.  We will use this bound $f$
  to construct an interval partition, with a distinguished point in
  each interval, that we in turn use to construct $h \in \Sym(\N)
  \setminus G$ such that $\langle G,h \rangle$ is cofinitary.

  Define $I = \big\langle \big([i_n,i_{n+1}), p_n\big) \mid n \in \N
  \big\rangle$ by $i_0 := 0$, $p_n := f(i_n)$, and $i_{n+1} :=
  f(p_n)$.  The main property of this sequence is
  \begin{equation*}
    \forall g \in G \ \forall^*n \ g(p_n) \in [i_n,i_{n+1}),
  \end{equation*}
  which follows easily from the fact that all elements of $G$ are
  nearly everywhere strictly bounded by $f$.

  Define $h$ by finite approximations:  let $h_0 := \emptyset$. Then
  at step $s$ define $h_{s+1}$ from $h_s$ as follows:
  \begin{enumerate}
    \item Let $a := \min( \N \setminus \dom(h_s))$, let $n$ be the
      least number such that for all $l \geq n$ we have $[i_l,i_{l+1})
      \cap \big(\dom(h_s) \cup \ran(h_s) \cup \{a\} \big) =
      \emptyset$, and set $\bar{h}_{s+1} := h_s \cup \{ (a,p_n)\}$.
    \item Let $b := \min( \N \setminus \ran(\bar{h}_{s+1}))$, let $m$
      be the least number such that for all $l \geq m$ we have
      $[i_l,i_{l+1}) \cap \big(\dom(\bar{h}_{s+1}) \cup
      \ran(\bar{h}_{s+1}) \cup \{b\}\big) = \emptyset$, and set
      $h_{s+1} := \bar{h}_{s+1} \cup \{(p_m,b)\}$.
  \end{enumerate}
  
  Note that the $a$ and $b$ used in this construction satisfy $a \leq
  b \leq a+1$.  To see this first note that the $p_n$ used go
  alternately into the domain and range, starting with the range.
  Then by induction it follows that if $a = b$ in an iteration then
  the least number in $\dom(h_s) \cup \ran(h_s) \setminus a$ is in the
  range.  If $a+1=b$, then that least number is in the domain.  The note
  quickly follows from these facts.
  
  The main properties of $h$ are the following (for $n >0$).
  \begin{enumerate}
    \item If $a \in [i_n,i_{n+1}) \setminus \{p_n\}$ and $(a,b) \in
      h$, then $b > i_{n+1}$.
    \item If $b \in [i_n,i_{n+1}) \setminus \{p_n\}$ and $(a,b) \in
      h$, then $a > i_{n+1}$.
    \item If $(a,p_n), (p_n,b) \in h$, then at most one of $a$ and
      $b$ is less than $i_n$.
    \item \label{item:proph4}
      If $(a_0,b_0),(a_1,b_1) \in h \cup h^{-1}$ and $a_0 < a_1 < b_0$,
      then $b_1 < a_0$ or $b_1 > b_0$.
  \end{enumerate}

  The first three of these follow from the observation that for any
  pair added to $h$ one of the coordinates is a $p_n$ and this $p_n$
  is from a later interval than the other coordinate is in (sometimes
  both coordinates are equal to $p_n$ for some $n$, but only one is
  used as such in the construction).  The last one follows from the
  fact that any added pair has one coordinate strictly bigger than any
  number mentioned before and the note on the order of $a$ and $b$
  above.

  Taking the first three properties of $h$ together we get that for
  any $n > 0$ there is at most one pair in $h \cup h^{-1}$ with one
  coordinate in $[i_n, i_{n+1})$ and the other smaller than $i_n$.
  From this we see that if $l < i_n$ and $h^\epsilon(l) \in [i_n,
  i_{n+1})$ for $\epsilon \in \{{-1},{+1}\}$ then $h^\epsilon(l) =
  p_n$.  Moreover for $m \in [i_n, i_{n+1}) \setminus \{p_n\}$ both
  $h(m)$ and $h^{-1}(m)$ are bigger than $i_{n+1}$; this is also the
  case for $h^\epsilon(p_n)$, but not for $h^{-\epsilon}(p_n) = l$.

  Now we show that $\langle G,h \rangle$ is cofinitary.  Let us
  assume, towards a contradiction, that $w(x) = g_0 x^{k_0} g_1 \cdots g_k
  x^{k_m} g_{m+1} \in W_G$ is such that $w(h)$ has infinitely many
  fixed points.  We can also assume that $g_{m+1} = \Id$, since this
  only requires conjugation by $g_{m+1}$ and this does not change the
  number of fixed points.

  We normalize the word $w$ further.  For this we work above $M$, the
  least number such that for all $n \geq M$ and all $g \in G$
  appearing in $w$ we have $g(n) < f(n)$.  We want a conjugate $w'$ of
  $w$ of the form $g_l x^{k_l} g_{l+1} \cdots$ $g_{m+1} g_0 x^{k_0}
  \cdots x^{k_{l-1}}$ such that for infinitely many of its fixed
  points, $n$, the image after the first application of $h$ (if
  $k_{l-1} > 0 $) or $h^{-1}$ (if $k_{l-1}<0$) is bigger than $n$.
  Such a conjugate $w'$ exists if for infinitely many fixed points we
  can find a location in the evaluation path where an application of
  $h$ increases the number.  So suppose that you can't do this; then
  for all but finitely many fixed points every application of $h$
  leads to a smaller number.  In this case we can find an $n$, a fixed
  point of $w(h)$ such that no point in its evaluation path $\bar{z}$
  is less than $M$ and for all $i$ such that $w_i = h^\epsilon$,
  $\epsilon \in \{+1,-1\}$, we have $z_{i+1} < z_i$ (remember $z_{i+1}
  = w_i z_i$).

  Now since we start in $w(h)$ by applying $h$ we get $z_1 < z_0 = n$.
  After this we cannot get back to $z_0$ as any application of a $g$
  appearing in $w$ to a number less than or equal to $z_1$ will lead
  to a number strictly less than $z_0$ ($z_0$ is in the middle of an
  interval which does not contain $z_1$ and $z_0$ is the $f$ image of
  the start of the interval it is in).  And any application of $h$ to
  a number strictly less than $z_0$ will lead to a number strictly
  less than $z_1$ (follows from the assumption and \ref{item:proph4}).
  This contradiction shows that a conjugate $w'$ as desired exists.

  We will study this conjugate $w'$ of $w$; if it can't have
  infinitely many fixed points neither can $w$. 
  There are only finitely many points whose evaluation path in $w'(h)$
  involves natural numbers less than $M$.  Leave these out of
  consideration.

  Let $\bar{z}$ be the evaluation path of $w'(h)$ on $n$, a fixed
  point for this word where the image after the first application of
  $h$ is bigger than $n$.  There is a least $m$ such that there is an
  $a \in \N$ such that  $z_{m+1} < i_a \leq z_m$.
  
  If for some $l$ we have $z_{l+1} > z_l$ by an application of $h$
  (either $h$ or $h^{-1}$) we have $z_{l+1} = p_b$ for some $b \in
  \N$.  If we now apply $h$ again (the same of $h$ or $h^{-1}$) we map
  to a $p_m$ with $p_m > p_b$.  So if we are in $w'$ at some $x^l$,
  repeatedly applying $h$, once we start increasing we will keep on
  increasing.

  If after such applications of $h$ where we increase we apply a $g
  \in G$ as indicated by $w'$, then $g$ doesn't map the element out of
  the interval it is in (we are working above $M'$ where no element of
  $g$ appearing in $w'$ can map further than $f$ or $f^{-1}$).

  Now we know that $z_m = p_k$ for some $k$, $z_{m+1}$ is obtained from
  $z_m$ by an application of $h$, $z_m$ is obtained from $z_{m-1}$ by
  an application of some $g \in G$ and $z_{m-1}$ is obtained from
  $z_{m-2}$ by an application of $h$ which was increasing.  From the
  last fact in the last sentence we know $z_{m-1} = p_l$ for some $l$.
  Since $p_l$ and $p_k$ are in the same interval, $p_l = p_k$ and we
  have found a fixed point for this $g \in G$.

  So we have found from a fixed point of $w'(h)$ a fixed point for some
  $g \in G$ appearing in $w'$.  Also, any fixed point of a $g \in G$
  appearing in $w'$ can only be used in the evaluation path of
  finitely many points (and only in the evaluation path of one fixed
  point if $g$ only appears once in $w'$).  From this we see that if
  $w'(h)$ has infinitely many fixed points, so does some $g \in G$.
  This is the contradiction we were looking for.
\end{proof}


\subsection{A Coanalytic Maximal Cofinitary Group}

In this subsection we will prove the following theorem.

\begin{theorem}
  The axiom of constructibility implies that there exists a coanalytic
  maximal cofinitary group.
\end{theorem}

The proof will be very much related to the proof of the analogous
result for very mad families, which is in Section
\ref{sect:VMAD:VisL}, but there are some essential differences.

The construction of very mad families (and many other types of maximal
almost disjoint families) proceeds by adding one new member at a time.
We recursively construct the family to be $\A = \{ f_\alpha \mid
\alpha < \omega_1 \}$.  Then under the axiom of constructibility we
have to prove a coding lemma of the following form.

\begin{lemma}[Coding Lemma --- Generic Form]
  If $A$ is a countable almost disjoint family and $z \in 2^\N$, we can
  construct a new member $f$ to adjoin to the family such that
  {
    \renewcommand{\theenumi}{(\roman{enumi})}
  \begin{enumerate}
    \item
      $z$ is recursive in $f$, and
    \item if we iterate the construction $\omega_1$ many times we
      construct a maximal almost disjoint family.
  \end{enumerate}
  }
\end{lemma}

In fact the construction has to be such that $z$ is uniformly
recursive in $f$; the function computing $z$ from $f$ should not
depend on $A$ or on other parameters in the construction.

If this can be achieved, the proof as in Section \ref{sect:VMAD:VisL}
can be easily adjusted to yield a coanalytic
maximal almost disjoint family, for whichever notion of almost
disjoint you are considering.

Using this method Su Gao and Yi Zhang were able to prove the following
in \cite{SGYZ}.

\begin{theorem}
  The axiom of constructibility implies that there exists a maximal
  cofinitary group with a coanalytic generating set.
\end{theorem}

They prove a nice version of the generic type coding lemma, and the
generating set is constructed in the right way for the general method
to apply.

The difficulty in showing that the whole group can be coanalytic is that
when you add a new generator you also add countably many other new
elements.  In the construction we will use the method of good
extensions, which means that the new generator will be free over all
that came before.  Then for all $w \in W_G \setminus G$ we will have
$w(g) \not \in G$.  And all these elements need to encode $z$ for the
method to work.

The following lemma shows that in the case of cofinitary groups we
cannot get $z$ uniformly recursive in $f$ in the coding lemma, when
our construction is computable.  This does not prove that uniform
computability is not possible as the construction does not need to be
computable, and moreover, it only has to work over a fixed group (the
one constructed in previous steps of the construction).

\begin{proposition}
  There do not exist recursive functionals $\Psi(X,Z,n)$ and
  $\Phi(X,n)$ such that for all countable cofinitary groups $G$, and
  all $z \in 2^\N$ the function $g \in \Sym(\N)$ defined by $g(n) =
  \Psi(G,z,n)$ satisfies that $\langle G,g \rangle$ is cofinitary, and
  for all $w \in W_G$ we have that $z(n) = \Phi(w(g),n)$.
\end{proposition}

\begin{proof}
  Let $G$ be given to us as a countable sequence $\langle g_i \mid i
  \in \N \rangle$, and assume that $\Psi$ and $\Phi$ as in the
  statement do exist.

  Pick a countable cofinitary group $G$, and a $z \in 2^\N$ with
  $z(0)=0$.  Define $g$ from $G$ and $z$ using $\Psi$ as in the
  statement of the lemma.  Let $u = \use(\Phi,g,0)$, the use of $g$ by
  the functional $\Phi$ when calculating $\Phi(g,0)$.
  Then for all $h \in \Sym(\N)$, if $h \restrict u = g \restrict u$
  then $\Phi(h,0) = 0$.

  Let $z' \in 2^\N$ be such that $z'(0) = 1$.  Define $g'$ from $G$
  and $z'$ using $\Psi$ as in the statement of the lemma.  Let $U =
  \use(\Psi,G,z',g'\restrict u)$, the maximal use of $G$ and $z$ by
  $\Psi$ in calculating $g'(0), \ldots, g'(u-1)$. 

  So in determining $g'(0), \ldots, g'(u-1)$ no use is made of any
  group element in the enumeration $\langle g_i \mid i \in \N \rangle$
  with index $i > U$.

  Now pick a new cofinitary group $\bar{G}$ and enumeration of it
  $\langle \bar{g}_i \mid i \in \N \rangle$ such that $\bar{g}_i =
  g_i$ for $i \leq U$ and there are elements $g_l$ and $g_k$
  such that $g_l (g'\restrict u) g_k  = g \restrict u$.

  Define $g''$ from $\bar{G}$ and $z'$ using $\Psi$ as in the
  statement of the lemma.  Then $g'' \restrict u = g' \restrict u$.
  However if we choose $w(x) = g_l x g_k$, then $w(g'') = g \restrict
  u$, which means that $\Phi(w(g''),0) = 0$ contradicting the fact
  that $\Phi$ computes $z'$ from $w(g'')$.
\end{proof}

Now that we know what the difficulty is, we will show how to deal with
it.  We will recursively construct the maximal cofinitary group.  To
make the coding work out though we have to start with a specific
countable group.

Let $G_0$ be the countable cofinitary group generated by $h$ defined
as follows:
\begin{equation*}
  h(n) = \begin{cases} n-2, & \text{ if } n \text{ is even and not
    zero};\\
    n+2, & \text{ if } n \text{ is odd};\\
    1, & \text{ if } n = 0. \end{cases}
\end{equation*}

Then there is a formula only involving natural number quantifiers
$\phi_{G_0}(x)$ that defines this group as a subgroup of $\Sym(\N)$.

The coding method we use has two cases and a parameter.  But with
these it will be uniform; there exists a recursive functional
$\Phi(X,m,\epsilon,n)$ such that if $z$ is encoded in $f$ we have that
there exist $m \in \N$ and $\epsilon \in \{0,1\}$ such that for all $n
\in \N$ we have $z(n) = \Phi(f,m,\epsilon,n)$.

The encoding will be as follows; $z$ is encoded in $f$ with parameter
$(m,0)$ iff
\begin{equation*}
  (k_n, z(n)) = f^n(m), \text{ for some } k_n \in \N;
\end{equation*}
$z$ is encoded in $f$ with parameter $(m,1)$ iff
\begin{equation*}
  (k_n, z(n)) = f (hf)^{n}(m).
\end{equation*}

This encoding will be done in the following way.  At some point in the
construction, we have already constructed a finite approximation $p$
to the new generator $g$.  We then start encoding into a new word $w
\in W_G$.  Let $w = g_0 x^{k_0} g_1 \cdots x^{k_l} g_{l+1}$ with $g_i
\in G$ ($i \leq l+1$) and $k_i \in \mathbb{Z} \setminus
\{0\}$. \label{gamma explained}  Pick
$m$ such that $g_{l+1}(m) \not \in \dom(p) \cup \ran(p)$, and let
$\gamma = 0$ if $w$ does not have a proper conjugate subword, $\gamma
= 1$ otherwise.  We extend $p$ by taking good extension with respect
to certain words, extending the evaluation path of $w(p)$ for $m$.  We
do this until $a = (w \restrict \lh(w) - 2)(m)$ is defined.  Then
(assuming $k_0 > 0$, the other case is analogous) we choose a $b$ such
that $p \cup \{(a,b)\}$ is a good extension with respect to certain
words, such that $w(p \cup \{(a,b)\})(m) \in \{(k,z(0)) \mid k \in
\N\}$, and such that we can encode $z(1)$ into the next location.

This last requirement is where the two different types of encoding
play a role.  If $w$ has no proper conjugate subword, then since $G$
is cofinitary there is only finitely much restriction from the
requirement that $g_{l+1}(g_0(b)) \not \in \dom(p) \cup \ran(p) \cup
\{b\}$.  If $w$ does have a proper conjugate subword, then we will
always have that $g_{l+1}(g_0(b)) = b$.  This is why in that case we
``twist'' by $h$.  The next location we then want to encode in is
$h(g_0(b))$ and, again since $G$ is cofinitary, we will have only
finitely much restriction from requiring $g_{l+1}(h(g_0(b))) \not \in
\dom(p) \cup \ran(p) \cup \{b\}$.

With this we have enough information to state and prove the coding
lemma for cofinitary groups.

\begin{lemma}
  Let $G$ be a countable cofinitary group, $F \leq \Sym(\N) \setminus
  G$ a countable family of permutations such that for all $f \in F$
  the group $\langle G,f \rangle$ is cofinitary, and $z \in 2^\N$.
  Then there exists $g$ such that $\langle G,g \rangle$ is
  cofinitary, for all $f \in F$ the set $f \cap g$ is infinite, and
  $z$ is recursive in $w(g)$ for all $w \in W_G \setminus G$.
\end{lemma}

\begin{proof}
  $W_G \setminus G$ is countable, enumerate it by $\langle w_n \mid n
  \in \N \rangle$, and enumerate $F$ by $\langle f_n \mid n \in \N
  \rangle$.
  
  Start by setting $g:= \emptyset$, $A := \emptyset$ and $\langle c_n
  \mid n \in \N \rangle$ with all $c_n := \emptyset$.  $g$ will be the
  permutation we construct, so at any time it will be a finite
  injective function.  $A$ is a set of numbers; it is the set of
  numbers in domain and range that are being used in coding.  We have
  to avoid this set in all steps other than coding steps .  It will
  always be finite and any number will stay in it for only finitely
  many stages of the construction.  $\langle c_n \mid n \in \N
  \rangle$ is a sequence of which at any time an initial segment will
  contain triples that hold information on how far we are in the
  coding, how we are coding, and where the coding currently is being
  done.

  At step $s \in \N$ in the construction we do the following:
  \begin{itemize}
  \item Extend Domain: Set $a := \min\{\N \setminus (\dom(g) \cup
    A)\}$.  By the domain extension lemma for all but finitely many
    $b$ the extension $g \cup \{(a,b)\}$ is a good extension of $g$
    with respect to all words $w_i$, $i \leq s$.  Choose $b$ to be the
    least such number such that $b \not \in A$ and set $g = g \cup
    \{(a,b)\}$.
    
  \item Extend Range: Set $b := \min\{\N \setminus (\ran(g) \cup
    A)\}$.  By the range extension lemma for all but finitely many $a$
    the extension $g \cup \{(a,b)\}$ is a good extension of $g$ with
    respect to all words $w_i$, $i \leq s$.  Choose $a$ to be the
    least such number such that $b \not \in A$ and set $g = g \cup
    \{(b,a)\}$.
  \end{itemize}
  Note: these two sub-steps ensure that $g$ will be a permutation of
  $\N$; no number stays in $A$ long enough to cause problems.
  \begin{itemize}
    \item
      Hit $f$:  For each $j \leq s$ in turn do the following:
      
      By the Hitting $f$ lemma, for all but finitely many $a$ the
      extension $g \cup \{(a,f_j(a))\}$ is a good extension of $g$
      with respect to all words $w_i$, $i \leq s$.  Choose $a$ to be
      the least such number such that $a, f_j(a) \not \in A$ and set
      $g = g \cup \{(a,f_j(a))\}$.
  \end{itemize}
  Note: this ensures for all $f \in F$ that $f \cap g$ is infinite.
  \begin{itemize}
    \item
      Coding:  For each $j < s$ in turn do the following:

      $c_j$ is a triple $(m,l,\gamma)$, where $m$ denotes where the
      coding is taking place, $l$ denotes the next location of $z$
      to encode, and $\gamma$ determines how to encode.

      Let $n$ be the largest number such that $a := (w_j \restrict
      n)(g)(m)$ is defined.  Then $w_j = w' g_j x^k x^\delta (w_j
      \restrict n)$, where $w' \in W_G$, $g_j \in G$, and $k \geq 0$ if
      $\delta = 1$ and $k \leq 0$ if $\delta = -1$.

      \underline{Case $\delta = 1$:}

      By the domain extension lemma, for all but finitely many $b$ the
      extension $g \cup \{(a,b)\}$ is a good extension of $g$ with
      respect to
      all words $w_i$, $i \leq s$.

      \underline{Subcase $k > 0$:}

      Choose $b$ to be the least number such that $b \not \in A \cup
      \dom(p)$, set $g = g \cup \{(a,b)\}$ and replace $a$ in $A$ by
      $b$ (so $a$ is no longer a member of $A$ but $b$ now is).

      \underline{Subcase $k=0$:}
      
      \underline{SubSubcase $w' = w'' x^{\delta'}$ ($\delta' \in
        \{-1,1\}$):}

      Choose $b$ to be the least number such that $b \not \in A$ and
      $g_j(b) \not \in A \cup \dom(g) \cup \ran(g)$ (in fact depending
      on $\delta'$ we only care about avoiding one of $\dom(g)$ or
      $\ran(g)$).  Set $g = g \cup \{(a,b)\}$ and replace $a$ in $A$
      by $g_j(b)$.

      \underline{SubSubcase $w' = \emptyset$:} (This is where the
      actual coding happens.)

      Choose $b$ to be the least number such that $b \not \in A$,
      $g_j(b) \not \in A$, $g(b) \in \{ (c, z(l)) \mid c \in \N \}$
      and if $\gamma = 0$ $w_0(g_i(b)) \not \in A \cup \dom(g) \cup
      \ran(g) \cup \{b\}$ or if $\gamma = 1$ then $w_0(h(g_i(b))) \not
      \in A \cup \dom(p) \cup \ran(g) \cup \{b\}$.

      The requirements $b \not \in A$, $g_j(b) \not \in A$, and
      $w_0(g_j(b)) \not \in A \cup \dom(p) \cup \ran(p)$ or
      $w_0(h(g_j(b))) \not \in A \cup \dom(p) \cup \ran(p)$ exclude
      finitely many possibilities for $b$.  Since $G$ is cofinitary,
      $w_0(g_j(b)) \neq b$ or $w_0(h(g_j(b))) \neq b$ also excludes
      finitely many possibilities.  So we can choose $b \in g_j^{-1} [
      \{(c, z(l)) \mid c \in \N\}]$ satisfying the last condition on
      $b$.
      
      Then set $g = g \cup \{(a,b)\}$, replace $a$ in $A$ by
      $w_0(g_i(b))$ (if $\gamma = 0$) or $w_0(h(g_i(b)))$ (if $\gamma
      = 1$) and set $c_j := (g_j(b), l+1, 0)$ (if $\gamma = 0$) or
      $c_j := (h(g_j(b)), l+1, 1)$ (if $\gamma = 1$). ($\gamma$ is set
      in Extending Coding below and explained on page \pageref{gamma
        explained}.)

      \underline{Case $\delta = -1$:}  The method and (sub)subcases
      are analogous to the case $\delta = 1$ but using the range
      extension lemma.

    \item
      Extending Coding:  If $w_s$ has a proper conjugate subword, set
      $\gamma = 1$; otherwise set $\gamma = 0$.  Then let $a$ be the
      least number such that if $w_s = w' g_s$, then $g_s(a) \not \in
      \dom(g) \cup \ran(g) \cup A$.  Add $g_s(a)$ to $A$ and set $c_s
      = (a,0,\gamma)$.  This indicates that at the next stage we will
      start encoding $z(0)$ into location $a$ for $w_s$.
  \end{itemize}
  Note: with the explanation before the lemma this ensures that the
  coding happens correctly.
\end{proof}

With the above indicated changes, Lemma \ref{lem:vmad:decidemembership}
has to be modified to be the following.

\begin{lemma}
\begin{align*}
  g \in G \Leftrightarrow & \ \phi_{G_0}(g) \vee \exists (m,\epsilon) \big[
      \text{ the model encoded in } g \text{ is wellfounded } \wedge \\
  & \forall \langle E_\omega, r, u \rangle \ \varphi(\langle
      E_\omega, r, u\rangle, g) \wedge \chi(E_\omega,r) \rightarrow r(
      \code{u \in \A }, \seq{\emptyset}) = 1 \big].
\end{align*}
\end{lemma}

This is clearly still a $\Pi^1_1$ formula, showing the result.






\bibliography{biblio}   


\startabstractpage
{Cofinitary Groups and Other Almost Disjoint Families of Reals} {Bart
Kastermans} {Co-Chairs: Andreas R. Blass and Yi Zhang}

We study two different types of (maximal) almost disjoint families: very
mad families and (maximal) cofinitary groups.

For the very mad families we prove the basic existence results.  We
prove that $\MA$ implies there exist many pairwise orthogonal families,
and that $\CH$ implies that for any very mad family there is one
orthogonal to it.  Finally we prove that the axiom of constructibility
implies that there exists a coanalytic very mad family.

\pagestyle{empty} 

Cofinitary groups have a natural action on the natural numbers.  We
prove that a maximal cofinitary group cannot have infinitely many
orbits under this action, but can have any combination of any finite
number of finite orbits and any finite (but nonzero) number of
infinite orbits.

We also prove that there exists a maximal cofinitary group into which
each countable group embeds.  This gives an example of a maximal
cofinitary group that is not a free group.  We start the investigation
into which groups have cofinitary actions.  The main result there is
that it is consistent that $\bigoplus_{\alpha \in \aleph_1}
\mathbb{Z}_2$ has a cofinitary action.

Concerning the complexity of maximal cofinitary groups we prove that they
cannot be $K_\sigma$, but that the axiom of constructibility implies
that there exists a coanalytic maximal cofinitary group.

We prove that the least cardinality $\mathfrak{a}_g$ of a maximal cofinitary
group can consistently be less than the cofinality of
the symmetric group.  Finally we prove that $\mathfrak{a}_g$ can
consistently be bigger than all cardinals in Cicho\'n's diagram.

\end{document}